\documentclass[11pt]{amsart}

\usepackage{CJKutf8}
\usepackage[dvipsnames]{xcolor}
\usepackage{url}
\usepackage[hyphenbreaks]{breakurl}

\usepackage{comment}
\usepackage{bm}
\usepackage{tikz,tikz-cd,varwidth, amssymb}
\usepackage[hidelinks]{hyperref}
\usetikzlibrary{calc, decorations}
\usepackage[text={6.22 in,8in},centering]{geometry}
\usepackage{enumerate, enumitem, stmaryrd}
\usepackage[all,cmtip]{xy}

\usepackage{graphicx}

\newtheorem{thm}{Theorem}[section]
\newtheorem{prop}[thm]{Proposition}

\newtheorem{cor}[thm]{Corollary}
\newtheorem{conj}[thm]{Conjecture}

\newtheorem{defn}[thm]{Definition}

\usetikzlibrary{arrows}
\usetikzlibrary{decorations.markings,decorations.pathmorphing,arrows}
\pgfarrowsdeclare{mytip}{mytip}{%
	\setlength{\arrowsize}{1pt}
	\addtolength{\arrowsize}{.5\pgflinewidth}
	\pgfarrowsrightextend{0}
	\pgfarrowsleftextend{-5\arrowsize}
}{%
	\setlength{\arrowsize}{1pt}
	\addtolength{\arrowsize}{.5\pgflinewidth}
	\pgfpathmoveto{\pgfpoint{-5\arrowsize}{1.5\arrowsize}}
	\pgfpathlineto{\pgfpointorigin}
	\pgfpathlineto{\pgfpoint{-5\arrowsize}{-1.5\arrowsize}}
	\pgfusepathqstroke
}
\tikzset{myarrow/.style={ decoration={bent,aspect=0.3, markings,mark=at
			position 0.5 with {\arrow[scale=1.2]{latex'}}}, postaction=decorate}}
\tikzset{myarrowshort/.style={ decoration={bent,aspect=0.3, markings,mark=at
			position 0.3 with {\arrow[scale=1.2]{latex'}}}, postaction=decorate}}
\tikzset{myarrowshorter/.style={ decoration={bent,aspect=0.3, markings,mark=at
			position 0.2 with {\arrow[scale=1.2]{latex'}}}, postaction=decorate}}
\tikzset{->-/.style={decoration={markings, mark=at position #1 with
			{\arrow{>}}},postaction={decorate}}}

\tikzset{my_dot/.style={fill, circle, inner sep=0pt,minimum size=1.5pt}}
\tikzset{my_node/.style={fill, circle, inner sep=0pt,minimum size=3pt}}
\tikzset{inv/.style={fill, circle, inner sep=0pt,minimum size=0pt}}

\newcommand{\notadp}{{  
\begin{tikzpicture}[baseline=-.55ex,scale=.2, every loop/.style={}]
 \node[circle,draw,fill,inner sep=.5pt] (a) at (0,0) {};
 \draw (a) edge[loop] (a);
 \draw (-.2,-.2) -- (.2,.5);
\end{tikzpicture}}}

\newcommand{\Sym}{\on{Sym}}

\newcommand{\gm}{\mathfrak{g}^{\mathfrak{m}}}
\newcommand{\grt}{\mathfrak{grt}}
\newcommand{\ab}{\mathrm{ab}}
\newcommand{\es}{\mathrm{es}}

\newcommand{\LL}{\mathbb{L}}

\newcommand{\NN}{\mathbb{N}}

\newcommand{\QQ}{\mathbb{Q}}
\newcommand{\RR}{\mathbb{R}}

\newcommand{\ZZ}{\mathbb{Z}}

\newcommand{\cA}{\mathcal{A}}

\newcommand{\cC}{\mathcal{C}}

\newcommand{\cF}{\mathcal{F}}
\newcommand{\cG}{\mathcal{G}}

\newcommand{\cM}{\mathcal{M}}

\newcommand{\cO}{\mathcal{O}}

\newcommand{\cS}{\mathcal{S}}
\newcommand{\cT}{\mathcal{T}}
\newcommand{\cU}{\mathcal{U}}

\newcommand{\Sat}{\textrm{\begin{CJK}{UTF8}{min}サ\end{CJK}}}
\newcommand{\BKGr}{BK_{\mathrm{Gr}}}
\newcommand{\MGr}{M_{\mathrm{Gr}}}

\newcommand{\Fil}{F}  

\newcommand{\on}{\operatorname}
\newcommand{\Gr}{\operatorname{Gr}}
\newcommand{\Spec}{\operatorname{Spec}}
\newcommand{\Prim}{\operatorname{Prim}}
\newcommand{\Proj}[1]{\mathrm{Proj}_{#1}}
\newcommand{\St}{\operatorname{St}}
\newcommand{\Tot}{\operatorname{Tot}}

\newcommand{\ov}{\overline}

\newcommand{\GL}{\mathrm{GL}}
\newcommand{\SL}{\mathrm{SL}}
\newcommand{\SO}{\mathrm{SO}}

\newcommand{\ad}{\mathrm{ad}}

\newcommand{\bu}{\bullet}

\DeclareMathOperator*{\colim}{colim}

\newcommand{\GC}{\mathsf{GC}}

\newcommand{\xra}{\xrightarrow}

\newcommand{\wt}{\widetilde}

\newcommand{\eps}{\epsilon}

\newcommand{\tr}{\on{tr}}

\newcommand{\Hom}{\on{Hom}}
\newcommand{\Sp}{\mathrm{Sp}}

\newcommand{\trop}{\mathrm{trop}}

\newcommand{\col}{\colon}

\newcommand{\computed}[1]{\textcolor{RoyalBlue}{\bm{#1}}}
\newcommand{\vanish}[1]{\textcolor{violet}{\bm{#1}}}
\newcommand{\stable}[1]{\textcolor{OliveGreen}{\bm{#1}}}
\newcommand{\blank}[1]{}

\newcommand{\hide}[1]{}
\newcommand{\stub}[1]{}

\newcommand{\matrixtwo}[4]{\left(\begin{matrix}#1&#2\\#3&#4\end{matrix}\right)}

\newcommand{\down}[2]{\xymatrix@R=6mm@C=2mm{
#1\ar[d]\\ #2
}}

\newcommand{\downlabel}[3]{\xymatrix@R=6mm@C=2mm{
{#1}\ar[d]^<<<{#3} \\ #2
}}

\newcommand{\squarediagram}[4]{\xymatrix@R=8mm@C=8mm{
#1\ar[d]\ar[r] & #2\ar[d] \\ #3\ar[r] &#4
}}
\newcommand{\squarediagrammapsto}[4]{\xymatrix@R=8mm@C=8mm{
#1\ar@{|->}[d]\ar@{|->}[r] & #2\ar@{|->}[d] \\ #3\ar@{|->}[r] &#4
}}
\newcommand{\squarediagramlabel}[8]{\xymatrix@R=8mm@C=8mm{
#1\ar[d]_{#6}\ar[r]^{#5} & #2\ar[d]^{#7} \\ #3\ar[r]^{#8} &#4
}}

\newcommand{\isocelesdown}[3]{\xymatrix@R=6mm@C=0mm{
& {#1}\ar[dl] \ar[dr] & \\
{#2} \ar[rr] && {#3}
}}

\newcommand{\isocelesdownlabel}[6]{\xymatrix@R=6mm@C=0mm{
& {#1}\ar[dl]_<<<<{#4} \ar[dr]^<<<<{#5} & \\
{#2} \ar[rr]_{#6} && {#3}
}}

\newcommand{\isocelesup}[3]{\xymatrix@R=6mm@C=0mm{
 #1\ar[rr]\ar[dr]  && #2\ar[dl] \\
 & #3 &
}}

\newcommand{\isocelesuplabel}[6]{\xymatrix@R=6mm@C=0mm{
 #1\ar[rr]^{{#4}} \ar[dr]_<<<{#5} && #2\ar[dl]^<<<{#6} \\
 & #3 &
}}

\newtheorem{Definition}[thm]{Definition}
\newenvironment{definition}
  {\begin{Definition}}{\end{Definition}}

\newtheorem{Example}[thm]{Example}
\newenvironment{example}
  {\begin{Example}\rm}{\end{Example}}
  
\newtheorem{Exercise}[thm]{Exercise}
\newenvironment{exercise}
  {\begin{Exercise}}{\end{Exercise}}
    
\newtheorem{Exploration}[thm]{Exploration}
\newenvironment{exploration}
  {\begin{Exploration}}{\end{Exploration}}

\newtheorem{Fact}[thm]{Fact}
\newenvironment{fact}
  {\begin{Fact}}{\end{Fact}}

\newtheorem{Theorem}[thm]{Theorem}
\newenvironment{theorem}
  {\begin{Theorem}}{\end{Theorem}}
  
\newtheorem{Lemma}[thm]{Lemma}
\newenvironment{lemma}
  {\begin{Lemma}}{\end{Lemma}}

\newtheorem{Proposition}[thm]{Proposition}
\newenvironment{proposition}
  {\begin{Proposition}}{\end{Proposition}}

\newtheorem{Corollary}[thm]{Corollary}
\newenvironment{corollary}
  {\begin{Corollary}}{\end{Corollary}}

\newtheorem{Question}[thm]{Question}
\newenvironment{question}
  {\begin{Question}}{\end{Question}}

\newtheorem{Conjecture}[thm]{Conjecture}
\newenvironment{conjecture}
  {\begin{Conjecture}}{\end{Conjecture}}
  
  \newtheorem{Problem}[thm]{Problem}

\theoremstyle{remark}

\newtheorem{Remark}[thm]{Remark}
\newenvironment{remark}
  {\begin{Remark}\rm}{\end{Remark}}

\newcommand \defnow[1]{\begin{definition}{#1}\end{definition}}

\newcommand \lemnow[1]{\begin{lemma}{#1}\end{lemma}}

\newcommand \proofnow[1]{\begin{proof}{#1}\end{proof}}
\newcommand \remnow[1]{\begin{remark}{#1}\end{remark}}
\newcommand \propnow[1]{\begin{proposition}{#1}\end{proposition}}
\newcommand \cornow[1]{\begin{corollary}{#1}\end{corollary}}

\newcommand \enumnow[1]{\begin{enumerate}{#1}\end{enumerate}}

\newcommand \sseq[1]{{}^{#1}\!E}

\title{Hopf algebras in the cohomology of $\cA_g$, $\GL_n(\ZZ)$, and $\SL_n(\ZZ)$}

\author[F.~Brown]{Francis Brown}\address{All Souls College, Oxford, OX1 4AL, United Kingdom}\email{francis.brown@all-souls.ox.ac.uk}
\author[M.~Chan]{Melody Chan}\address{Department of Mathematics, Brown University, Box
1917, Providence, RI 02912}\email{melody\_chan@brown.edu}
\author[S.~Galatius]{S{\o}ren Galatius}\address{Department of Mathematics, University of Copenhagen, Denmark}\email{galatius@math.ku.dk}
\author[S.~Payne]{Sam Payne}\address{Department of Mathematics, University of Texas at Austin, Austin, TX 78712}\email{sampayne@utexas.edu}
\date{\today}

\begin{document}

\begin{abstract}
We describe a bigraded cocommutative Hopf algebra structure on the weight zero compactly supported rational cohomology of the moduli space of principally polarized abelian varieties. By relating the primitives for the coproduct to graph cohomology, we deduce that $\dim H^{2g+k}_c(\cA_g)$ grows at least exponentially with $g$ for $k = 0$  and for all but finitely many positive integers $k$.  Our proof relies on a new result of independent interest; we use a filtered variant of the Waldhausen construction to show that Quillen's spectral sequence abutting to the cohomology of $BK(\ZZ)$ is a spectral sequence of Hopf algebras. From the same construction, we also deduce that $\dim H^{\binom{n}{2} - n - k}(\SL_n(\ZZ))$ grows at least exponentially with $n$, for $k = -1$ and for all but finitely many non-negative integers $k$.
\end{abstract}

\maketitle

\tableofcontents

\section{Introduction}

In this paper, we introduce and study new algebraic structures in the cohomology of the complex moduli space $\cA = \bigsqcup_{g \geq 0} \cA_g$ of principally polarized abelian varieties of all dimensions $g\ge 0$. Except where stated otherwise, all cohomology and compactly supported cohomology groups will be taken with coefficients in $\QQ$.  

The proofs of our results on the cohomology of $\cA$ involve several technical constructions of independent interest. We produce a filtered coproduct on the Waldhausen construction of $BK(\ZZ)$, the de-looping of the $K$-theory space of $\ZZ$.  
We also produce a filtered coproduct on a cubical space of graphs and a filtered map to $BK(\ZZ)$ that respects the relevant structures, inducing a morphism of spectral sequences of bialgebras. From the $E^1$-pages, we obtain new relations between the homology of the commutative graph complex $\GC_2$ and $K(\ZZ)$.

The terms on the $E^1$-page of Quillen's spectral sequence, induced by the rank filtration on $BK(\ZZ)$, have natural interpretations in terms of the cohomology of $\GL_n(\ZZ)$ and $\SL_n(\ZZ)$. Thus, our constructions give new structures on the unstable cohomology of these groups.

\subsection{A Hopf algebra structure on weight zero cohomology} The compactly supported cohomology $H^*_c(\cA)$ is isomorphic to $\bigoplus_g H^k_c(\cA_g)$, which we consider as a bigraded $\QQ$-vector space, in which $H^k_c(\cA_g)$ has bidegree $(g,k\!-\!g)$.  It inherits a natural mixed Hodge structure from those on $H^*_c(\cA_g)$, and the new structures we study are in its weight $0$ subspace.

\begin{theorem} \label{thm:HopfAg} 
{\hspace{-5 pt} There} is a bigraded Hopf algebra structure on the weight zero subspace $W_0H^*_c(\cA)$. 
\end{theorem}

\noindent The coproduct in this Hopf algebra structure is graded-cocommutative, and is induced by the proper maps $\cA_g \times \cA_{g'} \to \cA_{g + g'}$ via pullback.
The product we construct is more subtle; in particular, it is not graded-commutative. 

We now identify a bigraded subspace of the primitives for the coproduct that is closely related to invariant differential forms on symmetric spaces. Let $\Omega^*_c$ denote the  vector  space over $\QQ$ spanned by  non-trivial exterior products of symbols $\omega^{4k+1}$ for $k \geq 1$. 
It is  bigraded  by \emph{genus} and \emph{degree minus genus}, where $ \omega^{4k_1+1} \wedge \cdots \wedge \omega^{4k_r +1}$, with $k_1 < \cdots < k_r$, has genus $2k_r+1$ and degree $(4k_1+1) + \cdots + (4k_r + 1)$. 
Let $\Omega^*_c[-1]$ denote the shift of $\Omega^*_c$ in which $\omega^{4k_1+1} \wedge \cdots \wedge \omega^{4k_r +1}$ has genus $2k_r+1$ and degree $(4k_1 + 1) + \cdots + (4k_r + 1) + 1$.
Let $\Prim(H)$ denote the subspace of primitives in a Hopf algebra $H$.

\begin{theorem} \label{thm:canonical-inj}
There is an injection of bigraded vector spaces 
\[
\Omega^*_c[-1] \otimes \RR \to \Prim(W_0H^*_c(\cA;\RR)).
\]
\end{theorem}

\noindent This  injection uses the identification $W_0H^*_c(\cA_g) \cong H^*_c(A_g^\trop)$, where $A_g^\trop$ is the moduli space of principally polarized tropical abelian varieties \cite{bbcmmw-top, bmv}, the stratification by locally symmetric spaces $A_g^\trop = \bigsqcup_{g' \leq g} P_{g'} / \GL_{g'}(\ZZ)$, and the inclusion of the genus $g$ subspace of $\Omega^*_c[-1]$ in the compactly supported cohomology of $P_g/\GL_g(\ZZ)$, as in \cite{brown-bordifications}. Here $P_g$ denotes the cone of positive definite symmetric bilinear forms on $\RR^g$.

The subspace of primitives in a graded Hopf algebra is a Lie algebra with respect to the bracket $[x,y] = xy-(-1)^{\deg(x)\deg(y)}yx.$ 
By relating this Lie algebra to the cohomology of the commutative graph complex $\GC_2$ and the Grothendieck--Teichm\"uller Lie algebra, we find a subspace of $\Omega^*_c[-1]$, in diagonal bidegree $(d,d)$, that generates a free Lie subalgebra.

\begin{theorem} \label{thm:free}
For $N = 11$, the image of $\{ \omega^5, \ldots, \omega^{4N+1} \}$ in $\Prim(W_0H^*_c(\cA;\RR))$ generates a free Lie subalgebra.
\end{theorem}

\noindent The statement for $N = 11$ reflects what we can prove with the current knowledge of the Grothendieck--Teichm\"uller Lie algebra; it is expected to hold for all $N$.

By the Milnor--Moore theorem (\cite[Theorem 5.18]{MR0174052}), the cocommutative Hopf algebra $W_0H^*_c(\cA)$ is isomorphic to $\cU(\Prim(W_0 H^*_c(\cA)))$, the universal enveloping algebra of its Lie algebra of primitives.  
The universal enveloping algebra on the free Lie algebra generated by the images of $\omega^5,\ldots,\omega^{4N+1}$ is isomorphic to a tensor algebra generated by those elements. Therefore Theorems~\ref{thm:canonical-inj} and \ref{thm:free} have the following  implications for dimensions.  Let $\cT_*$ denote a free genus-graded tensor algebra with one generator in each genus  $3, 5, \ldots, 23$.  

\begin{corollary} \label{cor:tensor}
The dimension of $W_0 H^{2g}_c(\cA_g)$ is greater than or equal to the dimension of $\cT_g$.  More generally, whenever $\Omega^*_c[-1]$ is nonzero in bidegree $(g_0, g_0+k)$ for some nonnegative integer $k$, the dimension of $H^{2g + k}_c(\cA_{g})$ is greater than or equal to the dimension of $\cT_{g-g_0}$.
\end{corollary}

\begin{corollary}
\label{cor:expAg}
The dimension of $W_0H^{2g+k}_c (\cA_g)$ grows at least exponentially with $g$ for $k = 0$ and all but finitely many positive integers $k$.
\end{corollary}

\noindent See Section~\ref{subsec:diagonals-and-degrees} for proofs and refinements of Corollaries~\ref{cor:tensor} and~\ref{cor:expAg}.  In particular, there are at most 20 exceptional values of $k$ for which $\dim H^{2g+k}_c(\cA_g)$ does not grow at least exponentially.  The precise dimension bounds given by  Corollary~\ref{cor:tensor} depend on knowing that the image of $\{\omega^5, \ldots, \omega^{4N+1}\}$ generates a free Lie subalgebra for $N = 11$. However, the exponential growth stated in Corollary~\ref{cor:expAg} already follows from the statement for $N = 2$. Extending  Theorem~\ref{thm:free} to larger values of $N$ improves the base of the exponential lower bound only slightly; see Section~\ref{sec:Poincare}. 

\bigskip

The complex moduli space $\cA_g$ is a smooth Deligne--Mumford stack and has the homotopy type of a classifying space for $\Sp_{2g}(\ZZ)$. Thus, each of the results stated above has equivalent reformulations in terms of the group cohomology of $\Sp_{2g}(\ZZ)$. Indeed, applying Poincar\'e duality and identifying the singular cohomology of $\cA_g$ with the group cohomology of $\Sp_{2g}(\ZZ)$ gives  
$$H^*_c(\cA_g) \cong H^{g^2+g - *}(\Sp_{2g}(\ZZ);\QQ)^\vee.$$ In particular, we have the following immediate consequence of Corollary~\ref{cor:expAg}.

\begin{corollary} For $k = 0$ and all but finitely many positive integers $k$,
the dimension of $H^{g^2-g-k}(\Sp_{2g}(\ZZ);\QQ)$ grows at least exponentially with $g$.
\end{corollary}

It was previously known that $H^{g^2-g}(\cA_g)$ is nonvanishing for all $g$, because the tautological subring of $H^*(\cA_g)$, i.e., the subring generated by the $\lambda$-classes, is Gorenstein with socle in this degree \cite{van-der-geer-cycles}. It is expected that $H^k(\cA_g)$ may vanish for $k > g^2 - g$ \cite[Question~1.1]{bruck-patzt-sroka}. We note that all classes in the tautological subring are of pure weight, i.e., weight equal to cohomological degree, and the tautological subspace of $H^{g^2-g}(\cA_g)$ always has rank 1. The growth we find is in the top weight cohomology, i.e., the graded piece of the weight filtration Poincar\'e-dual to the weight zero compactly supported cohomology.

\subsection{Filtered coproduct on the Waldhausen construction}

The heart of our proof of Theorems~\ref{thm:HopfAg} and \ref{thm:canonical-inj} is the construction of a coassociative but not cocommutative  coproduct on the Waldhausen construction of $BK(\ZZ)$ \cite{Waldhausen} that is compatible with the rank filtration.  Here, $BK(\ZZ)$ is the 1-fold delooping of the $K$-theory space for $\ZZ$, so $\pi_{i+1}(BK(\ZZ)) = K_i(\ZZ)$.  
The homological spectral sequence associated to the rank filtration on $BK(\ZZ)$ is called the \emph{Quillen spectral sequence}, and we denote it by $\sseq{Q}^*$. The terms on the $E^1$-page satisfy 
\begin{equation} \label{eq:Qss}
\sseq{Q}^1_{n,k} \cong H_{k} (\GL_n(\ZZ); \St_n \otimes \QQ),
\end{equation}
where $\St_n$ denotes the Steinberg module. Our filtered coproduct gives rise to a bigraded Hopf algebra structure on each page of $\sseq{Q}^*$, as in Theorem~\ref{thm:QSS}. Let  $\mathrm{Indec}(H)$ denote the indecomposables in a Hopf algebra $H$. For $\sseq{Q}^1$ we have the following.

\begin{theorem} \label{thm:HopfSL}
There is a 
bigraded commutative Hopf algebra structure on the $E^1$-page of the Quillen spectral sequence
\[
\sseq{Q}^1_{n,k} = 
H_{k} (\GL_n(\ZZ); \St_n \otimes \QQ)
,
\]
and a surjection of bigraded vector spaces
\begin{equation}\label{eq:indec-homological-quillen}\mathrm{Indec}(\sseq{Q}^1 \otimes \RR) \xrightarrow{\pi} \Omega^*_c[-1]^\vee \oplus \RR \cdot e.
\end{equation}
Here, $e$ is in bidegree $(1,0)$.
\end{theorem}
\noindent These results follow from Theorem~\ref{thm:QSS} and the proof of Theorem 1.2 in Section~\ref{sec:proof12} (establishing an injection dual to~\eqref{eq:indec-homological-quillen}).

\subsection{Relation to the graph spectral sequence} \label{sec:intro-graphss}
Our proof of Theorem~\ref{thm:free} relies on relations between the filtered coproduct on the Waldhausen construction of $BK(\ZZ)$ and an analogous filtered coproduct on a space of graphs $\BKGr$.   
This filtered space of graphs, which will be defined in Section~\ref{sec:Graphss}, has the rational homology of the disjoint union of a circle and a point, but 
the graded dual of the $E_1$-page of the associated cohomological spectral sequence is a bialgebra whose primitive elements contain a copy of the Grothendieck--Teichm\"uller Lie algebra.

Moreover, the space $\BKGr$ admits a filtered map to the Waldhausen construction of $BK(\ZZ)$ that respects all of the relevant structures.   Let $\sseq{Q}_r^{*,*}$ denote the pages of the cohomological spectral sequence dual to $\sseq{Q}^r_{*,*}$.  The induced map on the diagonal subspace $E_1^{g,g}$ will be most important. 

\begin{proposition}  \label{prop:diagonalLie}
 There is a morphism of  graded  Lie algebras 
\begin{equation}  \label{eq:q-to-gc2}   \Prim \Big( \bigoplus_{g} \sseq{Q}_1^{g,g}\Big) \to   H^0 (\GC_2)    \end{equation} 
which, after tensoring with $\RR$, sends $\omega^{4k+1}$ for all $k\geq 1$  to an  element that is non-trivial in the abelianization of $H^0(\GC_2)\otimes \RR$ with respect to its graded  Lie algebra structure.

Moreover, for $N=11$, there is an injection \[T(\langle \omega^5, \omega^9, \ldots, \omega^{4N+1} \rangle)\hookrightarrow \bigoplus_{g\ge0} \sseq{Q}_1^{g,g}\otimes \RR.\] 
    \end{proposition}
\noindent This result is established in~\eqref{eq:prim-Quillen to grt},  Proposition~\ref{prop: omegasnonzeroingrtab}, and Corollary~\ref{cor: partialfreeness}.  The morphism~\eqref{eq:q-to-gc2} is induced by the tropical Torelli map, up to a zig-zag of rationally equivalent spaces that is explained in Section~\ref{sec: Loc-sym-Quillen} and further in Section~\ref{sec:freeness}.
The class $\omega^{4k+1}$ pairs non-trivially  with the homology class of the wheel graph $[W_{2k+1}] \in H_0(\GC_2^\vee)$. 
This pairing  is given by  an integral which is a non-zero rational multiple of $\zeta(2k+1)$ \cite{brown-schnetz}. 
To prove  Theorem~\ref{thm:free}, we use known results on the Grothendieck--Teichm\"uller Lie algebra, which is isomorphic to $H^0(\GC_2)$ by \cite{willwacher-kontsevich}, the fact that it contains a free Lie algebra with one generator in each odd degree greater than one \cite{brown-mixed}, and best-known bounds on agreement between the two with respect to their weight filtrations. See Section~\ref{sec:proof3}.

\medskip

We may pass from such statements about $\sseq{Q}_*$ to statements on $W_0H^*_c(\cA_g)$. Using one of the main results of \cite{bbcmmw-top}, we show that there is a canonical choice of an element $e$ in $\sseq{Q}^1_{1,0}$ and an isomorphism of bigraded algebras
\begin{equation} \label{eq:inflation}
 (W_0 H^*_c(\cA))^\vee \otimes_\QQ \QQ[x] / (x^2) \overset{\sim}{\to}  \sseq{Q}^1_{*,*} 
\end{equation}
taking $x$ to $e$.  
The element $e$ is 
a primitive in the Hopf algebra structure on $\sseq{Q}^1_{*,*}$, and 
Theorem~\ref{thm:HopfAg} is proved by taking the quotient 
by the Hopf ideal generated by this primitive, and passing to the graded dual.
See Section~\ref{sec:proof12}.   

\subsection{Applications to the unstable cohomology of \texorpdfstring{$\GL_n(\ZZ)$}{GLnZ} and \texorpdfstring{$\SL_n(\ZZ)$}{SLnZ}}

Our construction of a Hopf algebra structure on Quillen's spectral sequence, together with the injection in Theorem~\ref{thm:canonical-inj}, gives the following result and produces a bigraded commutative Hopf algebra structure on a large subspace of $\bigoplus_n H^*(\SL_n(\ZZ);\QQ)$.  

\begin{corollary} \label{cor:expSL}
The dimensions of $H^{\binom{n}{2} - n - k}(\GL_n(\ZZ);\QQ_\mathrm{or})$ and of $H^{\binom{n}{2} - n - k}(\SL_n(\ZZ);\QQ)$ grow at least exponentially with $n$ for $k = -1$ and for all but finitely many nonnegative integers $k$.
\end{corollary}
\noindent See Section~\ref{rem:kSL}, in particular for the refined statement that there are at most 10 exceptional values of $k$ for which these groups do not grow at least exponentially.

The cohomology groups $H^i(\SL_n(\ZZ);\QQ)$ are stable in low degrees, for $i \leq n - 2$ \cite{li-sun-lowdegree}; the growth in Corollary~\ref{cor:expSL} takes place entirely in the unstable range.  In the highest degrees, $H^{\binom{n}{2}- n - k}(\SL_n(\ZZ);\QQ)$ is conjectured to vanish for $k < -1$ \cite[Conjecture~2]{church-farb-putman-stability}. Hence, Corollary~\ref{cor:expSL} says that the cohomology groups grow exponentially in the degree immediately below the threshold at and above which they are expected to vanish, and at almost every fixed distance below the threshold. 

The following corollary can be deduced from dualizing~\eqref{eq:indec-homological-quillen} in Theorem~\ref{thm:HopfSL}, together with the Poincar\'e--Birkhoff--Witt theorem and Shapiro's Lemma.

\begin{cor}\label{cor:sym} There is an injection of bigraded  vector spaces
    \[ \Sym \left( \Omega^*_{c}[-1]  \oplus \QQ \cdot \epsilon \right) \otimes \RR   \hookrightarrow \bigoplus_{n} H_{\binom{n+1}{2} - *}(\SL_n(\ZZ);\RR)\]
    Here, $\epsilon$ is in bidegree $(1,0)$, and $H_{\binom{n+1}{2} - k}(\SL_n(\ZZ))$ is in bidegree $(n,k-n)$.
\end{cor}

\noindent Here, and elsewhere, $\Sym$ denotes the free graded-commutative algebra on a graded vector space. It is the tensor product of an exterior algebra on elements of $\Omega^{*}_{c}[-1] \oplus\RR\epsilon$ of odd degree, and a commutative polynomial algebra on the elements of even degree. 
A  version of this statement  (without the class $\epsilon$)
 was announced by Ronnie Lee \cite{lee-unstable}, but no proof has appeared in the literature.   Our results  imply that the  symmetric algebra  on a much larger set of primitive elements embeds into the cohomology of the special linear group: in addition to the classes in $\Omega^*_{c}[-1] \oplus \QQ \cdot \epsilon$, we may also take infinitely many independent commutators in the  free Lie algebra on $\{\omega^5,\ldots, \omega^{45}\}$; see Remark~\ref{rem:allknowncohomology}.   

The injection in Corollary~\ref{cor:sym} is defined with real coefficients. We also give related constructions with rational coefficients. The wheel graphs $W_{2k+1}$, for positive integers $k$, give non-trivial homology classes in the graph homology $H_0(\GC_2^\vee)$, i.e., the graded dual of $H^0(\GC_2)$.  Moreover, $[W_{2k+1}]$ pairs nontrivially with $\omega^{4k+1},$ by pushing forward from the filtered space of graphs $\BKGr$ discussed in Section~\ref{sec:intro-graphss} and identifying $[\omega^{4k+1}]$ with a compactly supported  cohomology class on the rank-$(2k+1)$ stratum of $BK(\ZZ)$.  By showing that these odd wheel classes are primitive with respect to the coproduct, we deduce the following statement in Section~\ref{sec:polynomials-wheel}.
\begin{cor} There is an injective map of commutative bigraded algebras
\[  \QQ[ W_3, W_5,\ldots , W_{2k+1},\ldots    ] \otimes \QQ[x]/(x^2) \longrightarrow   \sseq{Q}^1_{*, *}   \]
     where the element $x$ is in bidegree $(1,0)$. 
\end{cor}

\begin{remark}
Note that the last two corollaries only give polynomial growth in the cohomology of $\cA$ and $\SL_n(\ZZ)$. Our results on exponential growth are stronger and fundamentally use the fact that our co-commutative Hopf algebras are far from commutative. 
\end{remark}

\begin{remark}
    When this paper was in the final stages of preparation, we learned that Ash, Miller, and Patzt independently discovered a bigraded commutative Hopf algebra structure on $\bigoplus_{k,n} H_{k} (\GL_n(\ZZ); \St_n \otimes \QQ)$ and described some indecomposable elements. It is not immediately clear whether the two Hopf algebra structures agree. They also found a related Hopf algebra structure on $\bigoplus_{k,n} H_{k} (\GL_n(\ZZ); \St_n \otimes \widetilde{\QQ})$ \cite{ash-miller-patzt}.
\end{remark}

\begin{remark}
  Recent work of Yuan and Jansen introduces an $E_\infty$-monoid denoted $K^\partial(\ZZ)$ in \cite{Yuan} and $|\mathcal{M}_\mathrm{RBS}(\ZZ)|$ in \cite{Jansen} which seems related to the Hopf algebra structure on the Quillen spectral sequence, presumably by Koszul duality.  In particular, \cite[Theorem 10.11]{Jansen} implies that the $E^1$-page of Quillen's spectral sequence arises as homology of the bar construction of an $E_\infty$-algebra, which can perhaps be viewed as a conceptual reason for the Hopf algebra structure.
\end{remark}

\subsection{Relations to moduli spaces of curves} 

Theorems~\ref{thm:HopfAg} and \ref{thm:free}, and the resulting  exponential growth of $\dim W_0H^{2g}_c(\cA_g)$ (Corollary~\ref{cor:expAg} for $k = 0$) are analogs of the main results of \cite{cgp-graph-homology} for the moduli space of curves $\cM = \bigsqcup_{g \geq 2} \cM_g$. Neither of these collections of results appears to be directly deducible from the other.  However, $W_0H^*_c(\cA)$ and $W_0H^*_c(\cM)$ can be related to each other through a zig-zag, as follows.  Recall the Torelli map $\cM\to\cA$, taking a curve to its Jacobian.  The Torelli map is not proper, and hence does not give rise to a natural morphism of mixed Hodge structures between $H^*_c(\cM)$ and $H^*_c(\cA)$, but it factors as $\cM\to \cM^{\mathrm{ct}} \to \cA$, through the open inclusion of $\cM$ into the moduli space $\cM^{\mathrm{ct}}$ of curves of compact type 
(or the image of $\cM^{\mathrm{ct}}$ in $\cA$). The Torelli map extends to a proper morphism on $\cM^{\mathrm{ct}}$, giving rise to a zig-zag of mixed Hodge structures \[H^*_c(\cM) \to H^*_c(\cM^{\mathrm{ct}}) \leftarrow H^*_c(\cA).\] 
However, we do not yet have sufficient understanding of the compactly supported cohomology of $\cM^{\mathrm{ct}}$ (or its image in $\cA$) to use such a  zig-zag effectively.

Nevertheless, the Grothendieck--Teichm\"uller Lie algebra, via its interpretation as the zeroth degree cohomology of  the commutative graph complex $\GC_2$, plays a  common   role in  the top weight compactly supported cohomology of both $\mathcal{A}$ and $\mathcal{M}$. 
Indeed,   the main results of  \cite{cgp-graph-homology} identify $W_0H^*_c(\cM)$ with $H^*(\GC_2)$ and thereby endow it with the structure of a bigraded Lie algebra. Combining this identification with results from Grothendieck--Teichm\"uller theory \cite{brown-mixed, willwacher-kontsevich} shows  that $\bigoplus_g W_0 H^{2g}_c(\cM_g)$ contains a free Lie subalgebra with one generator $\sigma_g$ in each odd genus $g \geq 3$ and hence $\dim W_0 H^{2g}_c(\cM_g)$ grows at least exponentially with $g$. In particular the graph cohomology $H^*(\GC_2)$ plays the role of an intermediary between $W_0H^*_c(\cM)$ and $W_0H^*_c(\cA)$. Indeed, the main construction of \cite{cgp-graph-homology} is a map from $\GC_2$ to a cellular chain complex that computes $W_0H^*_c(\cM)$, which is in fact a quasi-isomorphism. Here, we construct a map from $\Prim \big( \sseq{Q}_1 \big)$ to $H^*(\GC_2)$ that vanishes on $\sseq{Q}_1^{1,0}$ and hence factors as
\[
\Prim \big( \sseq{Q}_1 \big) \to \Prim \big( W_0H^*_c(\cA) \big) \to H^*(\GC_2).
\]
This map $\Prim \big( W_0H^*_c(\cA) \big) \to H^*(\GC_2)$ is not an isomorphism, but is nevertheless essential to our proof of Theorem~\ref{thm:free} and its corollaries. It remains possible that this map may restrict to an isomorphism $\Prim \big( \bigoplus_g H^{2g}_c(\cA) \big) \to H^0(\GC_2)$; see Question~\ref{question: diagonaliso}.

\begin{remark}
The analog of Corollary~\ref{cor:expAg} for $H^*_c(\cM)$ is an open problem; it is conjectured but not known that $\dim H^{2g+k}_c(\cM_g)$ grows at least exponentially with $g$ for all but finitely many $k \geq 0$ \cite[Conjecture~1.3]{payne-willwacher-weight11}. This is proved for a few dozen values of $k$. In each known case, the exponential growth is established in one fixed graded piece of the weight filtration; see  \cite[Corollary~1.3]{payne-willwacher-weight2} and  \cite[Corollary~1.2]{payne-willwacher-weight11}. 
\end{remark}

\subsection{Further results and conjectures}

We now state questions, conjectures, and further results related to injectivity, vanishing, and extensions of Hodge structures involving $W_0H^*_c(\cA)$. We also state a generalization of our results related to the Hopf algebra on the Quillen spectral sequence (Theorem~\ref{thm:HopfSL} and Proposition~\ref{prop:diagonalLie}) for rings of integers in number fields.

\subsubsection{Injectivity} Theorems \ref{thm:canonical-inj} and \ref{thm:free}  suggest the following conjecture. 
\begin{conj} \label{conj: Tinjects} The inclusion of $\Omega^*_c[-1] \otimes \RR$ into the primitives for the coproduct on $W_0H^*_c(\cA;\RR)$ induces an injection 
 $T(\Omega^{*}_c[-1])  \otimes \RR  \to W_0H^*_c(\cA; \RR).$
\end{conj}
\noindent To support this conjecture, we construct a spectral sequence of bigraded Hopf algebras whose $E_1$-page is  $ T(\Omega^{*}_c[-1])$ and show that the abutment of this spectral sequence is isomorphic, as a bigraded vector space, to the abutment of the Quillen spectral sequence.   
Note that Conjecture~\ref{conj: Tinjects} implies  in particular that the tensor algebra on $\{\omega^{4k+1}\}$ injects into $W_0H^*_c(\cA; \RR)$, or equivalently, that the classes $\omega^{4k+1}$ generate a free Lie subalgebra of  $\Prim(W_0H^*_c(\cA; \RR))$. 
To prove this weaker statement, it would suffice to show that each class $\omega^{4k+1}$ maps into the motivic Lie subalgebra of the Grothendieck--Teichm\"uller Lie algebra \cite{brown-mixed}.

The map  in Conjecture~\ref{conj: Tinjects} is injective for $g\leq 9$ and an isomorphism for $g\leq 7$; the former follows from the methods developed in this paper, the latter from the isomorphism \eqref{eq:inflation} and explicit calculations in \cite{elbaz-vincent-gangl-soule-perfect}. It is not expected to be an isomorphism in general.  Indeed,  the dual vector space of $\Omega^*_c[-1]$ in genus $g$ is expected to be spanned by cohomology classes for $\GL_g(\ZZ)$ of non-cuspidal (Eisenstein) type \cite{Grobner}. 
Using automorphic methods,   cuspidal  cohomology classes for $\GL_g(\ZZ)$ for $g = 79$ and $105$ were recently constructed in  \cite{boxer-calegari-gee-cuspidal}. 
 Also, starting from $g = 9$, the Euler characteristic of $T(\Omega^{*}_c[-1])$ and the isomorphism \eqref{eq:inflation} are not compatible with homological Euler characteristic computations for $\GL_g(\ZZ)$ \cite[Theorem 3.3]{horozov-euler}.  

The situation potentially changes if we restrict to the  diagonal $\bigoplus_g W_0H^{2g}_c(\cA_g;\RR)$, which, by our earlier results, is a graded cocommutative Hopf algebra of particular interest. 
\begin{question} \label{question: diagonaliso} Is the induced map 
\[  T \Big( \bigoplus_{k\geq 1}  \omega^{4k+1} \RR \Big) \to   \bigoplus_g W_0H^{2g}_c(\cA_g;\RR)\]
an isomorphism? 
\end{question}

As noted above, the injectivity of this map  would follow from Drinfeld's conjecture that  the  motivic  Lie algebra surjects onto the Grothendieck--Teichm\"uller Lie algebra. Assuming the same conjecture and restricting to the Lie algebra of primitives, the analog of Question \ref{question: diagonaliso} for the cohomology of the moduli stack of curves has an affirmative answer, since $\bigoplus_g W_0 H^{2g}_c(\cM_g)$ is isomorphic to $H^0(\GC_2)$.

\subsubsection{Vanishing}

Note that $H^i_c(\cM_g)$ vanishes for $i < 2g$ \cite{harer-virtual, church-farb-putman-rational, morita-sakasai-suzuki-abelianizations}. Passing to weight zero, this implies the vanishing of $H^*(\GC_2)$ in negative degrees; see \cite[Theorem~1.1]{willwacher-kontsevich} and \cite[Theorem~1.4]{cgp-graph-homology}. It is expected that $H^i_c(\cA_g)$ may also vanish for $i < 2g$ \cite[Question 1.1]{bruck-patzt-sroka}; this is known for $i < g + \max\{2,g\}$  \cite{borel-serre-corners, gunnels-symplectic, bruck-patzt-sroka}. 

The vanishing of $H^{2g+1}_c(\cM_g)$ for all $g$ is a compelling open question, closely related to results in deformation theory, such as formality of deformation quantization \cite[Question~1.4]{payne-willwacher-weight11}. The analogous statement for $H^*_c(\cA)$ is also open.

\begin{question}
Does $H^{2g+1}_c(\cA_g)$ vanish for all $g$?
\end{question}

\noindent The answer is yes for $g \leq 4$. See \cite{hain-rational} and \cite[Corollary~3]{hulek-tommasi-cohomology}.

\subsubsection{Extensions of Tate Hodge structures}
Our study of $W_0H^*_c(\cA)$ also sheds new light on the full mixed Hodge structure $H^*_c(\cA)$. In particular, we can now show that this mixed Hodge structure contains a nontrivial extension of Tate Hodge structures in genus 3, answering questions of Hain and Looijenga.

Following Namikawa \cite{namikawa-toroidal-book}, we shall use the notation $\cA_g^\Sat$ for the Satake--Baily--Borel compactification of $\cA_g$.  
Hain observed over twenty years ago that the mixed Hodge structure on $H^6(\cA_3)$ is an extension of $\QQ(-6)$ by $\QQ(-3)$ and stated the expectation that it should be a multiple (possibly trivial) of the extension given by $\zeta(3)$ \cite[pp.~473--474]{hain-rational}. More recently, Looijenga showed that a Tate twist of a nontrivial multiple of this same extension appears in the stable cohomology of the Satake compactification $H^6(\cA_\infty^{\Sat})$ and asked whether the pullback $H^6(\cA_\infty^\Sat) \to H^6(\cA_3^\Sat)$ is injective \cite[p.~1370]{looijenga-goresky-pardon}. 
Here, we confirm Hain's expectation and give an affirmative answer to Looijenga's question. Moreover, we show that the extension in $H^6(\cA_3)$ is nontrivial.

\begin{theorem} \label{thm:extension}
The restriction map $H^6(\cA_\infty^\Sat) \to H^6(\cA_3^\Sat)$ is injective. Moreover, the image of this restriction map is equal to the image of $H^6_c(\cA_3)$ under push forward for the open inclusion $\cA_3 \subset \cA_3^\Sat$. In particular, the mixed Hodge structure on $H^6_c(\cA_3)$ is the nontrivial extension of $\QQ(-3)$ by $\QQ$ given by a nonzero rational multiple of $\zeta(3)$.
\end{theorem}

\noindent Theorem~\ref{thm:extension} is proved in Section~\ref{sec:proof-mhs}.  By  Poincar\'e duality, we see that  $H^6(\cA_3)$ is the nontrivial extension of $\QQ(-6)$ by $\QQ(-3)$ given by a nonzero rational multiple of $\zeta(3)$. 

\begin{remark} The motivic structure of $H^6(\cA_3)$ (i.e. its associated mixed Hodge structure and  $\ell$-adic Galois representations) and the extensions of Tate structures in $H^{2g}(\cA_g^\Sat)$ have generated sustained  interest. See, for instance,  \cite[Section~13]{clery-faber-vandergeer} for a discussion of the Siegel and Teichm\"uller modular forms that occur in $H^6(\cA_3)$ and $H^6(\cM_3)$ and \cite{vandergeer-looijenga} for a discussion of the differences that appear when working over fields of positive characteristic.
\end{remark}

We predict that Theorem~\ref{thm:extension} generalizes to higher genera, as follows. 

\begin{conjecture} \label{conj:omega-extension}
    For odd $g \geq 3$, the mixed Hodge structure on $H^{2g}_c(\cA_g)$ has a subquotient isomorphic to the extension of $\QQ(-g)$ by $\QQ$ given by a nonzero rational multiple of $\zeta(g)$. 
\end{conjecture}

\noindent 
As evidence for  Conjecture~\ref{conj:omega-extension}, we note that $\Gr_W^{2g} H^{2g}_c(\cA_g)$ contains a copy of $\QQ(-g)$, because the tautological subring of $H^*(\cA_g)$ has its socle in degree $g^2 - g$ \cite{van-der-geer-cycles}.

\subsection{Quillen spectral sequences for rings of integers in number fields}
Our proof of the existence of a Hopf algebra structure on the Quillen spectral sequence in  Theorem~\ref{thm:HopfSL}  holds  for  much more general rings $R$.  A  particular case of interest is when  $R=\mathcal{O}_K$ is the ring of integers in a  number field $K$. 
In this setting,  its  $E^1$-page was computed by Quillen and satisfies
\[  \sseq{Q}_{n,k}^1 (\mathcal{O}_K)  \cong  \bigoplus_{\mathfrak{a}} H_{k} ( \GL(P_\mathfrak{a}); \St(P_{\mathfrak{a}}\otimes K)) \]
where the sum is over representatives $P_{\mathfrak{a}}$  of the isomorphism classes of projective  $\mathcal{O}_K$-modules of rank $n$ (which for $n > 0$ are in bijection with the ideal class group of $\mathcal{O}_K$).  Our results imply that $\sseq{Q}^r_{*,*} (\mathcal{O}_K)$, and in particular its $E^1$-page, has the structure of  a commutative bigraded  Hopf algebra.  It gives a spectral sequence of Hopf algebras converging  to the rational homology of  $BK (\mathcal{O}_K)$, which was computed by Borel.

\bigskip

\noindent \textbf{Acknowledgments.}  We are most grateful to Samuel  Grushevsky for illuminating discussions about the geometry of $\cA_g$ and its compactifications. We also thank Avner Ash, Dick Hain, Ralph Kaufmann, Jeremy Miller, Natalia Pacheco-Tallaj, Peter Patzt, Dan Petersen, and Jan Steinebrunner for helpful conversations. 

FB thanks Trinity College Dublin for a Simons Visiting Professorship, and the University of Geneva for hospitality during which part of this work was carried out.  MC was supported by NSF CAREER DMS-1844768, NSF FRG DMS--2053221, and a grant from the Simons Foundation (1031656, Chan). SG thanks Alexander Kupers and Oscar Randal-Williams for previous collaborations on related topics \cite{GKRW20}, and was supported by the Danish National Research Foundation (DNRF151).  Part of this work was carried out while SG held a one-year visiting position at Columbia University, and he thanks the department for its hospitality and support.
SP was supported by NSF Grant DMS--2302475, NSF FRG Grant DMS--2053261, and a CRM--Simons Visiting Professorship in Montr\'eal.  Part of this collaboration occurred at Stanford University supported by the Poincar\'e Distinguished Visiting Professorship of SP.

\section{Preliminaries}

\subsection{Weight zero cohomology and the tropical spectral sequence}
  
Let $\cA_g$ denote the complex moduli stack of principally polarized abelian varieties of dimension $g$.  For background on geometry and topology of $\cA_g$ and its compactifications, we refer to the articles \cite{grushevsky-geometry} and \cite{hulek-tommasi-topology}. For each $k\ge 0$, its compactly supported rational cohomology $H^k_c(\cA_g;\QQ)$ admits a {\em weight filtration} 
\[W_0H^k_c(\cA_g;\QQ)\subset W_1H^k_c(\cA_g;\QQ) \subset \cdots  \subset  H^k_c(\cA_g;\QQ)\]
as part of the mixed Hodge structure on $H^k_c(\cA_g;\QQ)$.  Let
$\cA = \coprod_{g\ge 0} \cA_g$
be the moduli space of principally polarized abelian varieties of any dimension.  Then
$H^*_c(\cA) = \bigoplus_{g \ge 0} H^*_c(\cA_g)$
and we take $W_k H^*_c(\cA) = \bigoplus_g W_k H^*_c(\cA_g)$.

There is no single canonical choice of normal crossings compactification for $\cA_g$. Rather, by \cite{amrt}, there is a toroidal compactification $\ov{\cA_g}^\Sigma$ for every choice of certain polyhedral data $\Sigma$, as we now recall.  Let $P_g \subset \mathrm{Sym}^2((\RR^g)^\vee)$ denote the set of symmetric bilinear forms on $\RR^g$ that are positive definite. It is an open convex cone of full dimension inside the $\binom{g+1}{2}$-dimensional Euclidean vector space $\mathrm{Sym}^2((\RR^g)^\vee)$.  Let $P_g^{\mathrm{rt}}$ denote the set 
of positive semidefinite forms on $\RR^g$ whose kernel is rational, i.e., the kernel is of the form $W\otimes \RR$ for a vector subspace $W$ of $\QQ^g$.  Thus $P_g \subset P_g^{\mathrm{rt}}$.  
Let $\Sigma$ denote any {\em admissible decomposition} of $P_g^{\mathrm{rt}}$.  That is, $\Sigma$ is an infinite rational polyhedral cone decomposition supported on $P_g^{\mathrm{rt}}$, whose cones are permuted under the action given by $X\cdot A := A^TXA$ for $A\in \GL_g(\ZZ)$, and there are only finitely many orbits of cones of $\Sigma$ under this action. 
It is a nontrivial, but classical, fact that admissible decompositions exist. Famous examples include the decomposition into perfect cones;  
see \cite{amrt} and references therein.

It will be convenient to topologize $P_g^{\mathrm{rt}}$ not with its subspace topology induced from the ambient vector space $\mathrm{Sym}^2((\RR^g)^\vee)$, but rather with its {\em Satake} topology, which we now define.  Let $\Sigma$ denote any admissible decomposition of $P_g^{\mathrm{rt}}$.  The Satake topology on the set $P_g^{\mathrm{rt}}$ is the finest topology such that for every cone $\sigma$ in $\Sigma$, the map $\sigma \to P_g^{\mathrm{rt}}$ is continuous.  The Satake topology agrees with the Euclidean topology on $P_g^{\mathrm{rt}}$ for $g=0,1$, but is strictly finer for $g\ge 2$.  On the other hand, the Satake topology restricts to the Euclidean topology on the open subset $P_g$, since every point of $P_g$ is contained in only finitely many cones of $\Sigma$.  Moreover, it is independent of the choice of $\Sigma$, since any two choices of admissible decompositions admit a common refinement \cite[IV.2, p.~97]{faltings-chai-degenerations}.  Throughout, we consider $P_g^{\mathrm{rt}}$ as a topological space with the Satake topology.

Now define \[A_g^{\mathrm{trop}} := P_g^{\mathrm{rt}}/\GL_g(\ZZ).\]
For an interpretation of $A_g^{\mathrm{trop}}$ as a moduli space of principally polarized tropical abelian varieties, see \cite{bmv}.  A comparison theorem, as in \cite[Theorem 5.8]{cgp-graph-homology}, implies the following.

\propnow{\cite[Theorem 3.1]{bbcmmw-top}, \cite[Corollary 2.9]{odaka-oshima-collapsing} For each $k\ge 0$, there is a canonical isomorphism
\[H^k_c(A_g^{\mathrm{trop}}) \cong W_0 H^k_c(\cA_g).\]
}

\lemnow{\label{lem:loc-symmetric} For each $g>0$, we have a map
$A_{g-1}^{\mathrm{trop}}\to A_{g}^{\mathrm{trop}}$
which is a homeomorphism onto the closed subspace of $A_{g}^{\mathrm{trop}}$ whose  
complement is $P_g/\GL_g(\ZZ)$. 
}

\proofnow{Consider the linear injection 
$\Sym^2((\RR^{g-1})^\vee)\to \Sym^2((\RR^{g})^\vee)$ induced by the map $(\RR^{g-1})^\vee\to (\RR^g)^\vee$ that extends by zero on the last basis vector in the standard ordered basis of $\RR^g$.  It restricts to a map 
\begin{equation}\label{eq:skin} P_{g-1}^{\mathrm{rt}} \to P_g^{\mathrm{rt}}\end{equation} which we claim is continuous.  
Indeed, since the Satake topology identifies $P_g^{\mathrm{rt}}$ with the colimit of the cones in any admissible decomposition, it suffices to observe that the image of each perfect cone in $P_{g-1}^{\mathrm{rt}}$ is a perfect cone in $P_g^{\mathrm{rt}}$. 
The map \eqref{eq:skin} descends to a closed embedding
\[
A_{g-1}^\trop \hookrightarrow A_g^\trop;
\] 
the complement of the image is $P_g/\GL_g(\ZZ)$ \cite[Lemma~4.9 and Proposition 4.11]{bbcmmw-top}.
}

\subsubsection{The tropical spectral sequence}
\label{subsubsection:tropical}

Let $\sseq{T}$ 
denote the spectral sequence on Borel--Moore homology, with rational coefficients, associated to the sequence of spaces
\[\emptyset \subset A_0^{\mathrm{trop}} \subset A_1^{\mathrm{trop}} \subset \cdots. \] 
We henceforth call $\sseq{T}$ the {\em tropical spectral sequence}.  Since $A_s^\mathrm{trop} \setminus A_{s-1}^\mathrm{trop}$ is homeomorphic to $P_s/\GL_s(\ZZ)$ (Lemma~\ref{lem:loc-symmetric}), we have
\begin{equation}\label{eq:thing}
\sseq{T}^1_{s,t} = H_{s+t}^{\mathrm{BM}} (P_s / \GL_s(\ZZ); \QQ).\end{equation}
\remnow{Alternatively, let $A_g^{\mathrm{trop}} \cup \{\infty\}$ denote the one-point compactification of $A_g^{\mathrm{trop}}$.  Then $\sseq{T}$ is the first quadrant spectral sequence on reduced rational homology of the sequence of pointed spaces
\begin{equation}\label{eq:ag-infty}\{\infty\} \subset A_0^{\mathrm{trop}} \cup \{\infty\} \subset  A_1^{\mathrm{trop}} \cup \{\infty\} \cup \cdots,\end{equation}
with $\sseq{T}^1_{s,t} = H_{s+t}(A_s^{\mathrm{trop}}\cup\{\infty\}, A_{s-1}^{\mathrm{trop}}\cup\{\infty\};\QQ) \cong \wt H_{s+t}(P_s/\GL_s(\ZZ)\cup\{\infty\};\QQ)$ for $s \geq 0$. The one-point compactification point of view will play a role in Section~\ref{sec: Loc-sym-Quillen}, where the associated graded spaces of the filtered space $A_\infty^{\mathrm{trop}}\cup\{\infty\}$ arising as the colimit of~\eqref{eq:ag-infty} shall be related to the associated graded spaces of a filtration of the space $BK(\ZZ)$.
}
By Poincar\'e duality, $\sseq{T}^1_{s,t}$ is isomorphic to \begin{equation} \label{eq:thing2}
H^{\binom{s+1}{2} - s-t} (P_s / \GL_s(\ZZ);\QQ_\mathrm{or}).
\end{equation}
Here and throughout, $\ZZ_{\mathrm{or}}$ is the $\GL_s(\ZZ)$-module induced by the action on orientations of the symmetric space $P_s$ for $\GL_s(\RR)$, and $\QQ_{\mathrm{or}} := \ZZ_{\mathrm{or}} \otimes \QQ$. 
Note that $\QQ_\mathrm{or}$ is nontrivial if and only if $s$ is even \cite[Lemma 7.2]{elbaz-vincent-gangl-soule-perfect}.  Next, \eqref{eq:thing2} is identified with group cohomology \[H^{\binom{s+1}{2} - s-t}(\GL_s(\ZZ),\mathbb{Q}_{\mathrm{or}}).\]

Since $\GL_s(\ZZ)$ is a virtual duality group of virtual cohomological dimension $\binom{s}{2}$, with dualizing module
$\mathrm{St}_s \otimes \ZZ_\mathrm{or},$  it follows that $\sseq{T}^1_{s,t}$ is isomorphic to 
\begin{equation}\label{eq:thing3}
H_t (\GL_s(\ZZ);\mathrm{St}_s \otimes \QQ).
\end{equation}

The following proposition is a consequence of acyclicity of inflation \cite[\S5]{bbcmmw-top}, as we now explain.

\propnow{\label{prop:T-on-E2} The spectral sequence
\begin{equation}\label{sseq:T}\sseq{T}^1_{s,t} = H_{t}(\GL_s(\ZZ), \mathrm{St}_s \otimes \QQ) \end{equation}
has exact $E^1$-page and hence $E^2_{s,t} = 0$ for all $s,t$.}

\begin{proof}Let \[A_\infty^{\mathrm{trop}} = \cup_{g\ge 0} A_g^{\mathrm{trop}} = \varinjlim_g A_g^{\mathrm{trop}}.\] 
Let $C^{\mathrm{BM}}_*(X)$ denote the locally finite chain complex associated to a space $X$. 

Note that $\sseq{T}$ is the spectral sequence associated to the filtration on $C=\varinjlim_g C^{\mathrm{BM}}_*(A_g^{\mathrm{trop}})$ by
\[F_s C= C^{\mathrm{BM}}_*(A_s^{\mathrm{trop}}).\]

For each $g\ge 0$, let $\wt I_g \subset P_g^{\mathrm{rt}}$ denote the 
union of the cones in the $\GL_g(\ZZ)$-orbit of the set
\begin{equation}\label{eq:inflated}\{\sigma + \RR_{\ge 0}e_g e_g^T \mid \sigma \in \Sigma^{\mathrm{perf}}_{g-1}\}\end{equation}

Let $I_g = \wt I_g / \GL_g(\ZZ) \subset A_g^{\mathrm{trop}}$, henceforth referred to as the {\em inflation locus}. 

\lemnow{\label{prop:LI_g}\cite[Theorem 5.15]{bbcmmw-top} The link $LI_g$ has the rational homology of a point. }
\proofnow{The cellular chain complex for $LI_g$ is the one denoted $I^{(g)}_*$ in op.~cit., where it is shown that $I^{(g)}_*$  is acyclic.}
\remnow{In fact $LI_g$ is contractible.  This is not in the current literature, but can be proved using a slight modification of  Yun's Morse-theoretic enhancement of the proof of acyclicity in \cite[Theorem 5.15]{bbcmmw-top}. See \cite{yun-discrete}.  Yun's theorem deals not with $LI_g$ but with the link of a similarly defined {\em matroidal coloop} locus,    consisting of the $\GL_g(\ZZ)$-orbits of the unions of polyhedral cones~\eqref{eq:inflated} in which $\sigma$ is of the form
\[\RR_{\ge 0} \langle v_1v_1^T,\ldots,v_n v_n^T\rangle,\]
where $v_1,\ldots,v_n$ are the column vectors of a $g\times n$ totally unimodular matrix.
}

\cornow{$H_*^\mathrm{BM}(I_g) = 0$.}
\proofnow{We have $H_*^\mathrm{BM}(I_g) = H_*(I_g^+,\infty) = \widetilde H_* (S(LI_g)) = \widetilde H_{*-1}(LI_g)$, where $S(-)$ denotes suspension.  And $\widetilde H_{*-1}(LI_g)=0$ by Lemma~\ref{prop:LI_g}.}
  Now the inclusions
\[A_{g-1}^\mathrm{trop} \subset I_{g} \subset A_{g}^\mathrm{trop},\]
are closed and hence proper, and  \[H_k^{\mathrm{BM}}(I_g) = 0\qquad\text{for all }k\ge 0.\]

It follows that for each $k$ and for each $g$,
\begin{equation}  \label{eqn: InflationZeroMap}
H_k^{\mathrm{BM}}(A_{g-1}^\mathrm{trop}) \to H_k^{\mathrm{BM}}(A_{g}^\mathrm{trop})
\end{equation}
is the zero map, since it factors through $H_k^{\mathrm{BM}}(I_{g}) = 0$.  Now the fact that $\sseq{T}^2_{*, *} = 0$ follows from the following general proposition, concluding the proof of Proposition~\ref{prop:T-on-E2}.
\end{proof}

\propnow{\label{prop:E2-0} Let $(C,d)$ be a chain complex, and let 
\[0=F_{-1} C \subset F_0C \subset F_1C \subset \cdots  \]
be an increasing filtration on $C$, such that for each $i, k \ge 0$, the map 
\begin{equation}
    \label{eq:zerohomology}
H_k(F_i C) \to H_k(F_{i+1} C )
\end{equation}
is zero. Then the spectral sequence of the filtered chain complex
\[E^1_{s,t} = H_{s+t}(F_s C, F_{s-1} C)\]
collapses at $E^2_{s,t} = 0$ for all $s$ and $t$.
}
\proofnow{For each $i$, the short exact sequence of chain complexes
\[0\to F_i C \to F_{i+1}C \to F_{i+1}C / F_{i}C\to 0\]
has associated long exact sequence
\begin{eqnarray*}
    \cdots&\longrightarrow& H_k(F_i C) \xrightarrow{0} H_k(F_{i+1} C) \to H_k(F_{i+1}C, F_i C) \\
    &\longrightarrow& H_{k-1}(F_i C) \xra{0} H_{k-1}(F_{i+1} C) \to H_{k-1}(F_{i+1}C, F_i C) \to \cdots 
\end{eqnarray*}
which, by~\eqref{eq:zerohomology}, splits into short exact sequences
\[0\to H_k(F_{i+1} C) \to H_k(F_{i+1}C, F_i C) \to H_{k-1}(F_i C) \to 0\]
 for each $k\ge 0$.  Now the rows of $\sseq{T}^1$ of the spectral sequence read
\[\cdots \leftarrow H_{s+t-1}(F_{s-1}C, F_{s-2}C) \leftarrow H_{s+t}(F_{s}C, F_{s-1} C) \leftarrow H_{s+t+1}(F_{s+1}C, F_{s}C) \leftarrow \cdots \]
and 
fit into the commuting diagram below,
\[\xymatrix@R=4mm@C=-6mm{
&&&&0\ar[dl]\\
&&&H_{s+t}(F_sC)\ar@{_{(}->}[dl]&\\
H_{s+t-1}(F_{s-1}C,F_{s-2}C) & &H_{s+t}(F_sC, F_{s-1}C)\ar[ll]\ar@{->>}[dl] &&H_{s+t+1}(F_{s+1}C, F_sC) \ar[ll]\ar@{->>}[ul]\\
& H_{s+t-1}(F_{s-1}C) \ar@{_{(}->}[ul] \ar[dl] &&&\\
0 &&&&
}\]
from which it follows that the rows of $\sseq{T}^1$ are exact.  Hence $\sseq{T}^2=0$.  
}

Let $(C_{\le s}^*,d)$ denote the truncation of a cochain complex $(C^*, d)$, where $C_{\le s}^i = C^i$ for $i\le s$ and is zero 
for $i>s$.  The category of cochain complexes over a field $k$ is equivalent to the category of graded $k[x]/(x^2)$-modules, with $\deg (x)  = +1$.  We record the following standard fact.

\lemnow{\label{lem:acyclic-truncations} Suppose $(C^*,d)$ is an acyclic cochain complex over a field $k$. Then
\[(C^*, d) \cong \Big( \bigoplus_{s} H^s((C_{\le s}^*, d))\Big)\otimes k[x]/(x^2)\]
as graded $k[x]/(x^2)$-modules.
}

It will be convenient to consider the cohomological spectral sequence $\sseq{T}_r^{*,*}$ dual to the tropical spectral sequence~$\sseq{T}^r_{*,*}$ in \eqref{eq:thing}. The following is a consequence of Proposition~\ref{prop:T-on-E2}, by applying Lemma~\ref{lem:acyclic-truncations} to $\sseq{T}_1$.

\begin{theorem} \label{thm:epsilon} We have an isomorphism 
\[\big(\sseq{T}_1, d\big) \cong \Big(\bigoplus_{s \ge 0} W_0 H^{s+t}_c (\cA_s;\QQ) \Big) \otimes \QQ[x]/(x^2)\]
of $\ZZ^2$-graded cochain complexes with differential in degree $(1,0)$.  Here, $W_0 H^{s+t}_c (\cA_s;\QQ)$ is in bidegree $(s,t)$, and $x$ is in bidegree $(1,0)$.  
\end{theorem}

\begin{proof}
Proposition \ref{prop:T-on-E2} shows that $\sseq{T}_1$ is acyclic. The truncation of $\sseq{T}_1$ after the first $s$ columns is the $E_1$-page of the spectral sequence that converges on $E_2$ to $W_0 H^{s+t}_c(\cA_s;\QQ)$, supported in column $s$.  The statement then follows from Lemma~\ref{lem:acyclic-truncations}.
\end{proof}

In Proposition~\ref{prop:quillen-to-A} we shall upgrade  Theorem~\ref{thm:epsilon} to an isomorphism of graded-commutative algebras over $\QQ$.

\subsection{The Waldhausen construction of \texorpdfstring{$K$}{K}-theory}

We briefly review here the Waldhausen $S_\bu$-construction of $K$-theory, following the survey \cite[IV.8]{weibel-k-book}.  The construction applies to arbitrary Waldhausen categories.  We shall first recall the definition of a Waldhausen category. We do this only briefly, keeping in mind that our main example shall be the category of finitely generated projective left modules over a ring, as in Example~\ref{ex:projectives} below.  We refer to \cite[II.9]{weibel-k-book} for full and precise definitions.
 
A {\em category with cofibrations} is a category $\cC$ equipped with a special subcategory of morphisms $\mathrm{co}(\cC)$ called \emph{cofibrations}.  Every isomorphism is a cofibration. There is a distinguished $0$ object in $\cC$ and the unique map from $0$ to any object is a cofibration. Cofibrations are preserved under pushouts along arbitrary morphisms; in particular, such pushouts  exist.  
We denote cofibrations with a feathered arrow $\rightarrowtail$.  

A {\em Waldhausen category} is a category with cofibrations together with a family $w(\cC)$ of morphisms, called \emph{weak equivalences}.  Weak equivalences will be denoted $\xrightarrow{\sim}$.  All isomorphisms are weak equivalences. Weak equivalences are closed under composition and satisfy the following gluing condition: 
for every commutative diagram
\[\xymatrix{C \ar[d]^{\sim}& A\ar[l]\ar[d]^{\sim} \ar@{>->}[r]& B \ar[d]^{\sim}\\ C' & A'\ar[l]\ar@{>->}[r] & B'}\]
the induced map of pushouts
\[B\cup_A C\to B'\cup_{A'} C'\]
is a weak equivalence. 
The following is our main example of interest.
\begin{example}\label{ex:projectives}
    Let $R$ be a ring; we do not require that $R$ is commutative.  
    Let $\Proj{R}$ denote the category of finitely generated projective left $R$-modules, let $\mathrm{co}(\Proj{R})$ be the injections with projective cokernel, and let $w(\Proj{R})$ be the  isomorphisms.  Then $(\Proj{R}, \mathrm{co}(\Proj{R}), w(\Proj{R}))$ is a Waldhausen category.
\end{example}

We now review Waldhausen's $S_\bullet$-construction of $K$-theory, following \cite[IV.8]{weibel-k-book}. Let $\cC$ be any Waldhausen category.  
For each $p\ge 0$, define a category $S_p(\cC)$ whose objects are commutative diagrams 
\begin{equation}\label{eq:triangle-of-projectives}
\begin{aligned} 
\xymatrix{0 \ar@{>->}[r] & P_{0,1} \ar[d] \ar@{>->}[r] & P_{0,2} \ar[d] \ar@{>->}[r] & \cdots\ar[d]  \ar@{>->}[r] & P_{0,p}\ar[d]\\  & 0 \ar@{>->}[r] &P_{1,2} \ar[d] \ar@{>->}[r]  & \cdots \ar[d] \ar@{>->}[r] & P_{1,p} \ar[d] \\ && 0 \ar@{>->}[r] &\cdots \ar[d] \ar@{>->}[r] &\vdots \ar[d] \\ &&&0\ar@{>->}[r] & P_{p-1,p} \ar[d] \\ &&&& 0} 
\end{aligned}
\end{equation}
in which the horizontal morphisms are cofibrations and $P_{i,j}$ is a chosen cokernel for the horizontal map $P_i\to P_j$, where $P_i := P_{0,i}$. The arrows in $S_p(\cC)$ are morphisms $P_\bullet \to Q_\bullet$ of such diagrams that are weak equivalences.  By a weak equivalence, we mean a morphism $P_\bullet \to Q_\bullet$ such that each component $P_{i,j}\to Q_{i,j}$ is a weak equivalence.  

\remnow{In other sources such as \cite[IV.8]{weibel-k-book}, what we write as $S_p(\cC)$ above is written as $wS_p(\cC)$, while $S_p(\cC)$ is reserved for the category whose objects are diagrams ~\eqref{eq:triangle-of-projectives} and whose morphisms are morphisms of diagrams, with no restrictions.  
This more general notion of morphism in $S_p(\cC)$ is useful when iterating the $S_\bu$-construction, which we will not need here.  Since the only morphisms in $S_p(\cC)$ relevant to us are the weak equivalences, we shall
write $S_p(\cC)$ instead of $wS_p(\cC)$ for the category whose objects are triangular diagrams~\eqref{eq:triangle-of-projectives} and whose morphisms are weak equivalences.
}

Returning to Waldhausen's $S_\bu$-construction, we define functors
$d_i\colon S_\bullet(\cC) \to S_{\bullet -1}(\cC)$ for each $i$ by deleting both the $i$th row and the $i$th column in \eqref{eq:triangle-of-projectives} (i.e., the row with objects labeled $P_{i,j}$ and  the column with objects labeled $P_{k,i}$) and then re-indexing the remaining terms.
Similarly, we define functors
$s_i \colon S_\bullet(\cC) \to S_{\bullet + 1}(\cC)$ by duplicating both the $i$th row and the $i$th column in \eqref{eq:triangle-of-projectives}, inserting identity morphisms, and then re-indexing.  In this way, $S_\bullet(\cC)$ is a simplicial object in the category of categories.  
Taking the nerve yields a bisimplicial set 
\begin{align*}
    N_\bullet S_\bullet(\cC) \colon \Delta^{\mathrm{op}}\times \Delta^{\mathrm{op}} & \to \mathsf{Set}\\
    ([p],[q]) & \mapsto N_q S_p(\cC).
\end{align*}

\begin{remark}\label{rem:skeleton}
    In order to get a bisimplicial \emph{set}, the category $\cC$ must be small, i.e.\ there is a set of objects.  In that case $S_\bullet(\cC)$ is also small.  The main examples, such as finitely generated projective left modules over some ring, are \emph{essentially small} but not small if one takes \emph{all} such objects.  In those cases one chooses (often tacitly) a small category $\cC^\mathrm{small}$ equivalent to $\cC$ and applies the Waldhausen construction to that.  The resulting space $|N_\bullet S_\bullet(\cC^\mathrm{small})|$ is independent of this choice, up to homotopy equivalence.

    In the case $\cC = \Proj{\ZZ}$ of finitely generated projective modules over $\ZZ$, we may make the following explicit choice of $(\Proj{\ZZ})^\mathrm{small}$. The set of objects of $(\Proj{\ZZ})^\mathrm{small}$ is $\NN = \{0,1,2,\ldots\}$.
    The set $\mathrm{Mor}_{(\Proj{\ZZ})^\mathrm{small}}(a,b)$ is the set of $b\times a$ integer matrices. Composition is given by matrix multiplication.  This category may be further equipped with a symmetric monoidal structure corresponding to direct sum of $\ZZ$-modules: the product is given on objects by $m(a,b) = a+b$, and on morphisms by block sum of matrices.  The associator 
      \[\begin{tikzcd}
      (\Proj{\ZZ})^\mathrm{small}\times (\Proj{\ZZ})^\mathrm{small}\times (\Proj{\ZZ})^\mathrm{small}
      \arrow[bend left,
      start anchor=north east,
      end anchor=north west]{r}[name=U]{}
      \arrow[bend right,
      start anchor=south east,
      end anchor=south west]{r}[name=D]{} &
      (\Proj{\ZZ})^\mathrm{small}
      \arrow[Rightarrow, shorten <= 5pt, to path={(U) -- (D)}]{}
    \end{tikzcd}\]
    is given on $(a,b,c)$ by the identity matrix of size $a+b+c$.
    The symmetry natural isomorphism is given on $(a,b)\in (\Proj{\ZZ})^\mathrm{small}\times (\Proj{\ZZ})^\mathrm{small}$ 
    by the $(a+b)\times (a+b)$ block matrix 
    \[\matrixtwo{0}{\mathrm{Id}_a}{\mathrm{Id}_b}{0}.\]
    In this case, sending $n \mapsto \ZZ^n$ extends to a symmetric monoidal equivalence from $(\Proj{\ZZ})^\mathrm{small}$ to the category $\Proj{\ZZ}$ of all finitely generated free $\ZZ$-modules (equipped with the cartesian symmetric monoidal structure, i.e., direct sum).
\end{remark}

\subsection{The Quillen spectral sequence}

Recall the category $\Proj{R}$ of finitely generated projective left $R$-modules over a ring $R$, considered as a Waldhausen category as in Example~\ref{ex:projectives}. Then 
\begin{equation} \label{eq:BK}
BK(R) := |N_\bullet S_\bullet(\Proj{R})|, \mbox{ \ \ and \ \ } K_i(R) := \pi_{i+1} BK(R).
\end{equation}
By \eqref{eq:BK} and its degree shift, $BK(R)$ is a de-looping of the $K$-theory space of $R$.

We now take $R=\ZZ$.  Consider the following filtrations on the objects of $\Proj{\ZZ}$ and $S_\bullet(\Proj{\ZZ})$. First, an object $P$ is in $\Fil_n$ if $\mathrm{rank}(P)\le n$.  This filtration induces a filtration on the objects of $S_p(\Proj{\ZZ})$, in which an object of $S_p (\Proj{\ZZ})$, which is a triangular diagram~\eqref{eq:triangle-of-projectives}, is in $\Fil_n$ if  the top right projective module is in $\Fil_n$, i.e.,  has rank at most $n$.
Since that rank is non-increasing under face and degeneracy maps, and is preserved under weak equivalences, it follows that the filtration on objects of $S_p(\Proj{\ZZ})$, for each $p\ge 0$ yields a filtered bisimplicial set $N_\bu S_\bu (\Proj{\ZZ})$.
Hence, we obtain an exhaustive filtration of based spaces
\begin{equation}\label{eq:fil-BKZ}
 \Fil_0 BK(\ZZ) \subset \cdots  \subset \Fil_n BK(\ZZ) \subset \Fil_{n+1} BK(\ZZ) \subset \cdots\subset BK(\ZZ),\end{equation}
where $\Fil_0 BK(\ZZ)$ is contractible.

The functor $\Proj{\ZZ} \times \Proj{\ZZ} \to \Proj{\ZZ}$ defined by direct sum of $\ZZ$-modules induces a map
\begin{equation*}
    BK(\ZZ) \times BK(\ZZ) \xrightarrow{m} BK(\ZZ)
\end{equation*}
which is filtered in the sense that $m(\Fil_a(BK(\ZZ)) \times \Fil_b(BK(\ZZ))) \subset \Fil_{a+b} BK(\ZZ)$, and hence induces a product on the spectral sequence associated to the filtration (see Proposition~\ref{prop:(4)} below for more details).

\defnow{Let $\sseq{Q}^*_{*,*}$ denote the homology spectral sequence
  associated to the rank filtration on
  $BK(\ZZ)$, called the {\em Quillen spectral sequence}.}

A filtration of $BK(\ZZ)$ equivalent to the one above, and more generally a filtration of $BK(\mathcal{O}_F)$ for a number field $F$, was used by Quillen in \cite[pp.\ 179--199]{quillen-higher} for his proof that $K_i(\mathcal{O}_F)$ is a finitely generated abelian group for all $i \in \NN$.  Quillen's construction of the filtration used the ``$Q$-construction'' model for $BK(\mathcal{O}_F)$, and he also identified the associated graded with homology of the Steinberg module.

\propnow{\label{prop:E1-formula-with-Steinberg}There is a canonical isomorphism 
\begin{equation}\label{sseq:Q}
\sseq{Q}^1_{s,t} \cong H_{t}(\GL_s(\ZZ); \mathrm{St}_s \otimes \QQ) \Rightarrow H_{\ast}(BK(\ZZ)).\end{equation}
With respect to this isomorphism, the product on the $E^1$ page is induced by the block sum homomorphism $\GL_s(\ZZ) \times \GL_{s'}(\ZZ) \to \GL_{s+s'}(\ZZ)$ and the $\GL_s(\ZZ) \times \GL_{s'}(\ZZ)$-equivariant map $\mathrm{St}_s \otimes \mathrm{St}_{s'} \to \mathrm{St}_{s + s'}$ induced by the ``block sum'' map of spaces $T_s(\QQ) \ast T_{s'}(\QQ) \to T_{s + s'}(\QQ)$.
}

The Tits building $T_{s + s'}(\QQ)$ is the (nerve of the) poset of non-zero proper linear subspaces of $\QQ^{s + s'} = \QQ^s \oplus \QQ^{s'}$.  Identifying $\QQ^s$ with $\QQ^s \oplus \{0\} \subset \QQ^{s + s'}$ and $\QQ^{s'}$ with $\{0\} \oplus \QQ^{s'} \subset \QQ^{s + s'}$ leads to a canonical embedding of simplicial complexes\begin{align*}
    T_s(\QQ) \ast T_{s'}(\QQ) & \hookrightarrow T_{s+s'}(\QQ)\\
    (V \in T_s(\QQ)) & \mapsto V \oplus \{0\}\\
    (V' \in T_{s'}(\QQ)) & \mapsto \{0\} \oplus V',
\end{align*}
and this is the ``block sum'' map referred to in the proposition.  It is evidently equivariant with respect to the usual block sum operation of matrices $$(A,B) \mapsto \begin{pmatrix}A & 0 \\ 0 & B\end{pmatrix}.$$
We remark that for $s = 0$, one should take $\mathrm{St}_0 = \QQ$ in the above formula for the $E^1$-page, and that the canonical isomorphism $\QQ = H_0(\GL_0(\ZZ);\mathrm{St}_0)$ gives a two-sided unit for the product.

\begin{proof}
  As mentioned above this is essentially due to Quillen, although his construction of the spectral sequence and identification of the $E^1$-page with the  homology of the Steinberg module used a different model of $BK(\ZZ)$.  His model, the $Q$-construction, compares to the Waldhausen construction by an explicit homotopy equivalence  explained in \cite[Section~1.9]{Waldhausen}. One may verify that this homotopy equivalence induces a homotopy equivalence in each filtration degree.
  
    In Section~\ref{sec: Loc-sym-Quillen}, we need to use what the isomorphism 
    $\sseq{Q}^1_{s,t} \cong H_t(\GL_s(\ZZ); \mathrm{St}_s\otimes \QQ)$ is, so we sketch a direct construction which makes no mention of the $Q$-construction. 
    Let us write
  \begin{equation*}
    S^n_p (\Proj{\ZZ}) = \Fil_n S_p(\Proj{\ZZ}) \setminus \Fil_{n-1} S_p(\Proj{\ZZ})
  \end{equation*}
  for the full subgroupoid on those triangular diagrams~(\ref{eq:triangle-of-projectives}) in which $P_{0,p} \cong \ZZ^n$.  Then the filtration quotients are described, in bidegree $(p,q)$, by the bijection of pointed sets
  \begin{equation}\label{eq:4}
    \frac{\Fil_n N_q S_p(\Proj{\ZZ})}{\Fil_{n-1} N_q S_p(\Proj{\ZZ})} \cong N_q S^n_p (\Proj{\ZZ}) \cup \{\infty\}.
  \end{equation}
   Here, the quotient $S/T$ of a set $S$ by a subset $T$, or more generally along a morphism $T\to S$, is understood as the pushout of the diagram $\{\ast\} \leftarrow T \rightarrow S$. This pushout $S/T$ is naturally regarded as a pointed set with distinguished element corresponding to the image of $\{\ast\}\to S/T$.  
  On the right hand side,
  $\infty$ denotes a disjoint basepoint corresponding to the collapsed subset.  There are well defined functors $d_i\col S_p^n (\Proj{\ZZ}) \to S^n_{p-1}(\Proj{\ZZ})$ for $0 < i < p$, defined by deleting the $i$th row and column of the triangular diagram, but the outer faces $d_0$ and $d_p$ will not always land in the subgroupoid $S^n_{p-1}(\Proj{\ZZ}) \subset \Fil_n S_p(\Proj{\ZZ})$ because the rank may drop (which happens unless $P_{0,1} \cong 0$ or $P_{p-1,p} \cong 0$, respectively).  In the bisimplicial set~(\ref{eq:4}), the face operators in the $p$-direction will send a simplex to the basepoint $\infty$ when this happens.  We deduce that $E^1_{n,*}$ is calculated as the homology of the total complex associated to the double complex with
  \begin{equation*}
    C_{p,q} = \QQ \langle N_q S^n_p (\Proj{\ZZ})\rangle = \frac{ \QQ \langle N_q S^n_p (\Proj{\ZZ}) \cup \{\infty\}\rangle}{\QQ\langle\{\infty\}\rangle},
  \end{equation*}
  where $\QQ\langle X\rangle$ denotes the free $\QQ$-module on a set $X$, and boundary maps defined as alternating sum of face maps as usual.

  The groupoid $S_p(\Proj{\ZZ})$ comes with a functor $F$ 
  to the groupoid of finitely generated projective $\ZZ$-modules and their isomorphisms, defined by sending a triangular diagram~(\ref{eq:triangle-of-projectives}) to $P_{0,p}$.  The categorical fiber 
  $(F\downarrow\ZZ^n)$ over the object $\ZZ^n$ is the groupoid whose objects are triangular diagrams equipped with a specified isomorphism $P_{0,p} \to \ZZ^n$, and where morphisms from one triangular diagram to another is an isomorphism of diagrams compatible with the two specified isomorphisms from upper right entries
  to $\ZZ^n$.  If we denote this groupoid $\widetilde{S}^n_p (\Proj{\ZZ})$, then there is a forgetful map
  \begin{equation}\label{eq:16}
    N_q \widetilde{S}^n_p (\Proj{\ZZ}) \to N_q S^n_p (\Proj{\ZZ})
  \end{equation}
  which identifies $N_q S^n_p (\Proj{\ZZ})$ with the set of orbits for the free $\GL_n(\ZZ)$-action on the set $N_q \widetilde{S}^n_p (\Proj{\ZZ})$, defined by 
  acting by post-composition on
  the specified isomorphisms $P_{0,p} \cong \ZZ^n$.  There is also a map of sets
  \begin{equation}
    \label{eq:15}
    N_q \widetilde{S}^n_p (\Proj{\ZZ}) \to \pi_0(N_\bullet \widetilde{S}^n_p (\Proj{\ZZ}))
  \end{equation}
  sending a $q$-simplex to the path component containing it.  Regarding $p$ as fixed, the maps~(\ref{eq:15}) assemble to a map of simplicial sets from $[q] \mapsto N_q \widetilde{S}^n_p (\Proj{\ZZ})$ to the constant simplicial set $[q] \mapsto \pi_0(N_\bullet \widetilde{S}^n_p (\Proj{\ZZ}))$. This map is a weak homotopy equivalence 
  because automorphisms in the groupoid $\widetilde{S}^n_p (\Proj{\ZZ})$ are uniquely determined by their action on the top-right module $P_{0,p} = \ZZ^n$ and hence all automorphism groups in this groupoid are trivial. 

  The set $\pi_0(N_\bullet \widetilde{S}^n_p (\Proj{\ZZ}))$ is in bijection with the set of flags $0 \subset P_{0,1} \subset \dots \subset P_{0,p-1} \subset \ZZ^n$ of saturated submodules.  Such a flag is non-degenerate as a $p$-simplex of $[p] \mapsto \pi_0(N_\bullet \widetilde{S}^n_p(\Proj{\ZZ}))$ if and only if none of the inclusions are equalities. For $p \geq 2$, this set of non-degenerate $p$-simplices is in bijection with the set of $(p-2)$-dimensional faces of the Tits building $T_n(\QQ)$.  Writing
  \begin{equation*}
    \widetilde{C}_{p,q} = \QQ \langle N_q \widetilde{S}^n_p (\Proj{\ZZ})\rangle = \frac{ \QQ \langle N_q \widetilde{S}^n_p (\Proj{\ZZ}) \cup \{\infty\}\rangle}{\QQ\langle\{\infty\}\rangle},
  \end{equation*}
  we deduce that the maps
  \begin{equation*}
    \Tot_p\widetilde C_{*,*} \to \widetilde{C}_{p-2} (T_n(\QQ))
  \end{equation*}
that are induced by~(\ref{eq:15}) on $\widetilde{C}_{p,0}$ and zero on $\widetilde{C}_{p',p-p'}$ for $p'\ne p$ assemble to a quasi-isomorphism of chain complexes,   and hence, for $n \geq 1$, an isomorphism
  \begin{equation*}
    H_p(\Tot\widetilde{C}_{*,*}) \cong
    \begin{cases}
      \mathrm{St}_n(\QQ) \otimes \QQ & \text{for $p = n$}\\
      0 & \text{otherwise}.
    \end{cases}
  \end{equation*}
  The maps~(\ref{eq:16}) induce an isomorphism of double complexes
  \begin{equation*}
    \QQ \otimes_{\QQ[\GL_n(\ZZ)]} \widetilde{C}_{*,*} \xrightarrow{\cong} C_{*,*}
  \end{equation*}
  and hence an isomorphism of associated total complexes.  Since $\Tot(\widetilde{C}_{*,*})$ is a complex of free $\QQ[\GL_n(\ZZ)]$-modules, we can use it to compute group homology as
  \begin{align*}
    H_p(\GL_n(\ZZ);\mathrm{St}_n(\QQ)) & = \mathrm{Tor}_p^{\QQ[\GL_n(\ZZ)]}(\QQ,\mathrm{St}_n(\QQ)) = H_{p-n}(\QQ \otimes_{\QQ[\GL_n(\ZZ)]} \Tot \widetilde{C}_{*,*}) \\ &= H_{p-n}(\Tot C_{*,*}) = H_{p-n}(\Fil_n(BK(\ZZ)),\Fil_{n-1}(BK(\ZZ));\QQ),
  \end{align*} as required. The assertion about the algebra structure follows by tracing isomorphisms, since direct sum of $\ZZ$-modules corresponds to block sum of matrices upon choosing an isomorphism to $\ZZ^n$ for some $n$.
  \end{proof}

\begin{remark}
    We have stated the theorem and proof above for the ring $R = \ZZ$, but the rank filtration evidently makes sense for any ring $R$, by considering the full subcategories $\Fil_g(\Proj{R})$ on those projective modules which arise as summands of $R^{\oplus g}$.

The formula for the $E^1$-page generalizes to any Dedekind domain $R$, for instance $R = \cO_E$ for a number field $E$, or if $R$ is any field.  If $K$ is the fraction field of a Dedekind domain $R$, then saturated subspaces of a finitely generated projective $R$-module $P$ of rank $s$
are in bijection with linear subspaces of $P \otimes_R K \cong K^{\oplus s}$, and the Tits building $T(P \otimes_R K) \approx T_s(K)$ is defined as the nerve of the partially ordered set of non-zero proper linear subspaces of $P \otimes_R K$.
The geometric realization of this partially ordered set has the homotopy type of a wedge of $(s-2)$-spheres by the Solomon--Tits theorem, and the Steinberg module can be defined as $\St(P \otimes_R K) = \widetilde{H}_{s-2}(|T(P \otimes_R K)|) \cong \St_s(K)$.  The proof above goes through in this generality with only minor modifications, giving a Quillen spectral sequence of the form
\begin{equation*}
    \sseq{Q}^1_{s,t} =  \bigoplus_P H_t(\GL(P);\St(P\otimes_R K)) \Rightarrow H_{s+t}(BK(R))
\end{equation*}
where $P$ ranges over a set of representatives for the isomorphism classes of projective modules $P$ over $R$ of rank $s$.  For a completely general ring we still have a spectral sequence associated to the rank filtration, but not a useful formula for its $E^1$-page.
\end{remark}

In Section~\ref{sec:k-theory-graphs}, we explain a construction of $K$-theory for graphs that is closely related to the $S_\bu$-construction for Waldhausen categories. Several variants are possible.  The variant we use, which seems most closely related to the Quillen spectral sequence, does not fit the definitions of a Waldhausen category, but is similar in spirit.

\subsection{Comparison of tropical and Quillen spectral sequences}
\label{sec: Loc-sym-Quillen}

Recall that the tropical spectral sequence and the Quillen spectral sequence have isomorphic $E^1$-pages
\begin{equation*}
    \sseq{T}^1_{s,t} \cong 
    H_t(\GL_s(\ZZ);\St_s \otimes \QQ)
    \cong \sseq{Q}^1_{s,t}.
\end{equation*}
The goal of this subsection is to make the implied isomorphism of $E^1$-pages more explicit.  Recall that Borel--Moore homology of a reasonable space agrees with the relative homology of its one-point compactification, for instance $H_s^\mathrm{BM}(A_g^\trop) = H_s (A_g^\trop \cup \{\infty\}, \infty)$.  Therefore, the tropical spectral sequence can be viewed as the spectral sequence associated (in ordinary singular homology) to the filtered space
\begin{equation*}
    A_\infty^\trop \cup \{\infty\} = \colim_{g \to \infty} A_g^\trop \cup \{\infty\},
\end{equation*}
the direct limit of the one-point compactifications of the locally compact space $A_g^\trop$.  The associated graded has
\begin{equation*}
    \Gr_g(A_\infty^\trop \cup \{\infty\}) \cong (P_g / \GL_g(\ZZ)) \cup \{\infty\},
\end{equation*}
the one-point compactification of $P_g/\GL_g(\ZZ)$.

In this subsection we will construct an explicit zig-zag of (rational) equivalences
\begin{equation}\label{eq:24}
    \Gr_g(BK(\ZZ)) \xleftarrow{\simeq} 
    |N_\bu T_\bu(\QQ^g)\cup\{\infty\}|
    \xrightarrow{\simeq_\QQ} \Gr_g(A_\infty^\trop \cup\{\infty\}))
\end{equation}
in which the middle term $|N_\bu T_\bu(\QQ^g)\cup\{\infty\}|$ will be defined in this section.
The reduced homology of the rightmost space agrees with the Borel--Moore homology of $P_g/\GL_g(\ZZ)$ and hence 
is the $E^1$-page of the tropical spectral sequence.  The reduced homology of the leftmost space in the zig-zag is manifestly the $E^1$-page of the Quillen spectral sequence, and for both spectral sequences we have identified the $E^1$-page as $E^1_{s,t} = H_t(\GL_s(\ZZ),\St_s\otimes\QQ)$.
The purpose of discussing this zig-zag is to evaluate certain pairings: given classes
\begin{align*}
  w \in \sseq{Q}^1_{s,t} & \cong H_t(\GL_s(\ZZ);\St_s \otimes \QQ)\\
  \omega \in \sseq{T}_1^{s,t} \otimes \RR& \cong \Hom(H_t(\GL_s(\ZZ);\St_s),\RR)
\end{align*}
we wish to make sense of the pairing $\langle w,\omega\rangle \in \RR$.  To do this we need to pin down the isomorphism $\sseq{Q}^1_{s,t} \to \sseq{T}^1_{s,t}$, which is what the zig-zag~\eqref{eq:24} is useful for.  In practice, the homology class $w$ will be represented by a somewhat explicit cycle in  $\Gr_g(BK(\ZZ))$ arising from graphs, and the cohomology class $\omega$ will be given by an explicit $\GL_g(\ZZ)$-equivariant differential form on $P_g$.  In this situation, the strategy will be to transfer $w$ along the zig-zag to get an interpretation as a  cycle in $(A_g^\trop \cup \{\infty\})/(A_{g-1}^\trop \cup \{\infty\})$ and evaluate the pairing by performing an integral.

Before constructing the zig-zag~\eqref{eq:24}, let us point out that this must happen at the level of associated gradeds  and cannot be lifted to the level of filtered spaces, since the tropical spectral sequence collapses to $E^2 = 0$ while the Quillen spectral sequence does not.

We first define the space in the middle.
\begin{definition}
    For a finite-dimensional $\QQ$-vector space $V$, let $N_0 T_p(V)$ be the set of pairs $(A_p \subset \dots \subset A_0, <)$ where
    \begin{equation*}
        A_p \subset \dots \subset A_0 \subset V^\vee \setminus\{0\}
    \end{equation*}
    is a flag of non-zero vectors in the dual vector space, $<$ is a total order on $A_0$, such that 
    \begin{enumerate}
        \item $A_i$ is finite for all $i$,
        \item $A_p = \emptyset$,
        \item $\mathrm{span}(A_0) = V^\vee$.
    \end{enumerate}
    Let $N_0 T_p(V) \cup\{\infty\}$ denote the pointed set obtained by adding a disjoint base point, which we denote $\infty$, and define maps of pointed sets
    \begin{equation*}
        d_i \colon N_0 T_p(V) \cup\{\infty\} \to N_0 T_{p-1}(V) \cup\{\infty\}
    \end{equation*} 
    for $0 \leq i \leq p$ by the formula
    \begin{equation*}
        d_i(A_p \subset \dots \subset A_0,<) =
        \begin{cases}
            \infty & \text{if $i = 0<p$ and $\mathrm{span}(A_1) \neq V^\vee$}\\
            \infty & \text{if $i = p > 0$ and $A_{p-1} \neq \emptyset$}\\
            (A_p \subset \dots \subset \widehat{A_i} \subset \dots \subset A_0,<) &\text{otherwise.}
        \end{cases}
    \end{equation*}    
    There are also degeneracy maps $s_i \colon N_0 T_{p-1}(V) \cup \{\infty\} \to N_0 T_p(V) \cup \{\infty\}$ defined by duplicating $A_i$, making $[p] \mapsto N_0 T_p(V) \cup \{\infty\}$ into a simplicial object in the category of pointed sets.

There is an evident action of $\GL_g(\ZZ)$ on $N_0 T_p(\QQ^g)$, which we can use to turn $N_0 T_p(\QQ^g)$ into the objects of a groupoid $T_p(\QQ^g)$: the set of morphisms from $(A_p \subset \dots \subset A_0)$ to $(A'_p \subset \dots \subset A'_0)$ is the set of matrices $X \in \GL_g(\ZZ)$ with the property that $X.(A_p \subset \dots \subset A_0) = (A'_p \subset \dots \subset A'_0)$.  (This action is of course the restriction of an action of $\GL_g(\QQ)$, but for our purposes we need the action by integer matrices only.)  Explicitly,
\begin{equation*}
  N_q T_p(\QQ^g) = (\GL_g(\ZZ))^q \times N_0 T_p(\QQ^g),
\end{equation*}
with face maps in the $q$-direction given by
\begin{equation*}
\resizebox{.95\hsize}{!}{
$\begin{aligned}
    d_i(X_1,\dots, X_q,(A_p \subset \dots \subset A_0,<)) =
    \begin{cases}
        (X_2, \dots, X_q, (A_p \subset \dots \subset A_0,<)) & \text{for $i = 0$}\\
        (X_1, \dots, X_i X_{i+1}, \dots, X_q, (A_p \subset \dots \subset A_0,<)) & \text{for $0 < i < p$}\\
        (X_1, \dots, X_{q-1}, X_q.(A_p \subset \dots \subset A_0,<))
    \end{cases}
    \end{aligned}$
    }
\end{equation*}
\end{definition}

The notation $N_\bu T_\bu(\QQ^g) \cup \{\infty\}$ should be read as the bisimplicial pointed set
\begin{equation*}
  ([p],[q]) \mapsto N_q T_p(\QQ^g) \sqcup \{\infty\},
\end{equation*}
with face and degeneracy operators as defined above.
\begin{remark}
  All automorphism groups in the groupoid $T_p(\QQ^g)$ are trivial, so the canonical map
  \begin{equation}\label{eq:pi-zero-T}
    |N_\bu T_p(\QQ^g)| \to \pi_0(N_\bu T_p(\QQ^g)) \cong N_0(T_p(\QQ^g))/\GL_g(\ZZ)
  \end{equation}
  is a weak equivalence for all $p$.  Therefore the middle space in the zig-zag~\eqref{eq:24} is also homotopy equivalent to the geometric realization of the simplicial set
  \begin{equation*}
    [p] \mapsto N_0T_p(\QQ^g)/\GL_g(\ZZ) \cup \{\infty\}.
  \end{equation*}
  We need both simplicial directions for the map to $\Gr_g(BK(\ZZ))$ though.
\end{remark}

We first explain how to construct the map $|N_\bu T_\bu(\QQ^g) \cup \{\infty\}| \to \Gr_g(BK(\ZZ))$ in the zig-zag~(\ref{eq:24}).  Here it is natural to look for a map of bisimplicial sets, with $(p,q)$-simplices
\begin{equation*}
  N_q T_p(\QQ^g) \cup \{\infty\} \to \Gr_g(N_q S_p(\Proj{\ZZ})).
\end{equation*}
We explain how this works for $q = 0$, starting with $(\emptyset = A_p \subset \dots \subset A_0,<) \in N_0 T_p(V)$.  

The subsets $A_i \subset V^\vee \setminus \{0\}
$ give rise to canonical maps
\begin{align*}
    V &\to \QQ^{A_i}\\
    v &\mapsto (\psi \mapsto \psi(v)),
\end{align*}
which by assumption is injective for $i = 0$.  For $V = \QQ^g$ we shall make use of the restrictions
\begin{equation*}
    \ZZ^g \hookrightarrow \QQ^g \to \QQ^{A_i}
\end{equation*}
in the following.  For $0 \leq i \leq j \leq p$, define a free $\ZZ$-module $P_{i,j}$ as the image of the composition
\begin{equation*}
  \mathrm{Ker}(\ZZ^g \to \QQ^{A_j}) \hookrightarrow \ZZ^g \to \QQ^{A_i}.
\end{equation*}
These modules fit into a triangular diagram of the form~\eqref{eq:triangle-of-projectives}, forming an object of $S_p(\Proj{\ZZ})$.  Let us also remark that for $j=p$, the submodule $P_{i,p} \subset \QQ^{A_i}$ agrees with the image of the second map $\ZZ^g \to \QQ^{A_i}$, which for $i = 0$ is injective.  Hence we obtain a canonical isomorphism $\ZZ^g \to P_{0,p}$.  We have constructed a map of sets
\begin{equation}\label{eq:25}
  N_0 T_p(\QQ^g) \to \Fil_g(N_0 S_p(\Proj{\ZZ})) \setminus \Fil_{g-1}(N_0 S_p(\Proj{\ZZ})),
\end{equation}
which is easily extended to a map of bisimplicial pointed sets
\begin{equation*}
  N_\bu T_\bu(\QQ^g) \cup \{\infty\} \to \Gr_g(N_\bu S_\bu(\Proj{\ZZ})),
\end{equation*}
which in turn realizes to the desired map of spaces
\begin{equation*}
  |N_\bu T_\bu(\QQ^g) \cup \{\infty\}| \to \Gr_g(BK(\ZZ)).
\end{equation*}
\begin{proposition}
  This map $|N_\bu T_\bu(\QQ^g) \cup \{\infty\}| \to \Gr_g(BK(\ZZ))$ induces an isomorphism in integral homology.
\end{proposition}

\begin{proof}
   The simplicial set $N_0T_\bu(V) \cup \{\infty\}$ has a cubical structure (similar in nature to a space $\MGr$ which we will define later), as follows.  For each pair $(A,<)$ consisting of a finite subset $A \subset V^\vee \setminus\{0\}$ and a total order $<$ on $A$, the set of elements $(A_p \subset \dots \subset A_0, <') \in N_0T_p(V) \cup \{\infty\}$ for which $A_0 \subset A$ and $<'$ is the restriction of $<$, assemble as $p$ varies to a map of simplicial sets
   \begin{equation*}
       (\Delta^1_\bu)^A \xrightarrow{(A,<)} N_0T_\bu(V) \cup \{\infty\}.
   \end{equation*}
   This map lands in the basepoint $\infty$ unless $A$ spans $V^\vee$, in which case the induced map of geometric realizations
   \begin{equation}\label{eq: Lee-Szczarba cell}
       |\Delta^1_\bu|^A \xrightarrow{|(A,<)|} |N_0 T_\bu(V)| \cup \{\infty\}
   \end{equation}
   is injective when restricted to $(\Delta^1 \setminus \partial \Delta^1)^A$.  As $(A,<)$ varies over all possible (totally ordered) finite spanning sets of $V^\vee$, the maps~\eqref{eq: Lee-Szczarba cell} together with the basepoint $\infty$ form a CW structure on the space $|N_0 T_\bu(V) \cup \{\infty\}|$.  It follows that the reduced homology of $|N_0T_\bu(V)\cup\{\infty\}|$ may be computed by a chain complex defined combinatorially by generators $[\phi_0,\dots, \phi_p] \in (V^\vee\setminus\{0\})^{p+1}$ for varying $p$, subject to the relation that $[\phi_0, \dots, \phi_p] = 0$ unless the $\phi_i$ span $V^\vee$.  (Notice that we do not impose any relations between generators that differ by the action of the symmetric group $S_{p+1}$ on the set of ordered $(p+1)$-tuples.)  The boundary map is given by
   \begin{equation*}
       \partial[\phi_0,\dots, \phi_p] = \sum_{i = 0}^p (-1)^i [\phi_0, \dots, \widehat{\phi_i}, \dots, \phi_p].
   \end{equation*}
   For $V = \QQ^g$, this chain complex is in fact a well known free resolution of $\St_g(\QQ)$ as a module over $\GL_g(\QQ)$ due to Lee and Szczarba \cite[Section 3]{lee-szczarba-homology}, and we deduce
   \begin{equation*}
       \widetilde{H}_s(|N_0T_\bu(\QQ^g) \cup\{\infty\}|) \cong
       \begin{cases}
           \St_g(\QQ) & \text{for $p = g$}\\
           0 & \text{otherwise.}
       \end{cases}
   \end{equation*}
   
   The Steinberg module $\St_g(\QQ)$ is also the (reduced) homology of  $N_\bu \widetilde S^g_\bu(\Proj{\ZZ}) \cup \{\infty\}$, the bisimplicial set from the proof of Proposition~\ref{prop:E1-formula-with-Steinberg} above. It is not hard to compare these two instances of the Steinberg module.  Indeed, the map~\eqref{eq:25} was constructed by a recipe in which the ``upper right'' module $P_{0,p}$ in the triangular diagram came with a preferred isomorphism $\ZZ^g \to P_{0,p}$, and this provides a lift to
   \begin{equation*}
        N_0 T_\bu(\QQ^g) \cup \{\infty\} \to N_0 \widetilde{S}^g_\bu(\Proj{\ZZ})  \cup \{\infty\}.
   \end{equation*}
   On reduced homology, the composition
   \begin{equation*}
        |N_0 T_\bu(\QQ^g) \cup \{\infty\}| \to |N_0 \widetilde{S}^g_\bu(\Proj{\ZZ}) \cup \{\infty\}| \to |N_\bu \widetilde{S}^g_\bu(\Proj{\ZZ}) \cup\{\infty\}|
   \end{equation*}
   therefore induces a map between two modules isomorphic to $\St_g(\QQ)$.  It is the same map as considered by Lee and Szczarba and hence an isomorphism.
   
   By a spectral sequence argument, it follows that the map
   \begin{equation*}
       |N_\bu T_\bu(\QQ^g) \cup \{\infty\}| \to |N_\bu S^g_\bu(\Proj{\ZZ})| \cup \{\infty\}
   \end{equation*}
   also induces an isomorphism on homology.   
\end{proof}
   
Next we define the map to $\Gr_g(A_\infty^\trop \cup \{\infty\}) = (P_g/\GL_g(\ZZ)) \cup \{\infty\}$ in~\eqref{eq:24} as a composition
\begin{equation*}
    |N_\bu T_\bu(\QQ^g) \cup \{\infty\}| \xrightarrow{\simeq} |N_0T_\bu(\QQ^g)/\GL_g(\ZZ) \cup \{\infty\}| \to \Gr_g(A_\infty^\trop \cup\{\infty\}),
\end{equation*}
where the first map is induced by~\eqref{eq:pi-zero-T} above, and we must define the second map.  To an object $(\emptyset = A_p \subset \dots\subset A_0,<) \in N_0 T_p(\QQ^g)$ we associate a map $\Delta^g \to (P_g/\GL_g(\ZZ)) \cup \{\infty\}$, constructed in the following way.  For $(\emptyset = A_p \subset \dots \subset A_0,<) \in N_0 T_p(\QQ^n)$ consider the map
\begin{equation}
    \begin{aligned}
\label{eq:map-simplex-to-g}
  (\Delta^p \setminus \partial \Delta^p) &\to P_g\\
  (0 < s_1 < \dots < s_p < 1) & \mapsto \sum_{i = 1}^{p} \ell(s_i) \sum_{\psi \in A_{i-1} \setminus A_i} \psi^2,
\end{aligned}
\end{equation}
where $\psi^2 \in P_g^\mathrm{rt}$ denotes the rank-1 quadratic form $v \mapsto (\psi(v))^2$, and \begin{equation}\label{eq:diffeo}\ell\col (0,1)\to (0,\infty)\end{equation} is any diffeomorphism, for example $s\mapsto -\log(s)$ or $s\mapsto s/(1-s)$.  

Letting $(A_p \subset \dots \subset A_0,<) \in N_0 T_p(\QQ^g)$ vary, these maps assemble to a $\GL_g(\QQ)$-equivariant map
\begin{equation*}
  \coprod_p (\Delta^p\setminus\partial \Delta^p) \times N_0 T_p(\QQ^g) \to P_g.
\end{equation*}
Passing to orbit sets for the action of the subgroup $\GL_g(\ZZ)$, we obtain a map which extends as in the following diagram
\begin{equation}\label{eq:20}
  \begin{tikzcd}
    \coprod_p (\Delta^p\setminus\partial \Delta^p) \times N_0 T_p(\QQ^g)/\GL_g(\ZZ) \rar\dar[hookrightarrow] & P_g/\GL_g(\ZZ) \dar\\
    {|N_0T_\bu(\QQ^g)/\GL_g(\ZZ) \cup \{\infty\}|} \rar & (P_g/\GL_g(\ZZ)) \cup \{\infty\}.
  \end{tikzcd}
\end{equation}
\begin{proposition}\label{prop:bottom-horizontal}
  The bottom horizontal map in~(\ref{eq:20}) induces an isomorphism in reduced rational homology.
\end{proposition}
We remark that a map similar to the bottom horizontal map in~\eqref{eq:20} appeared in~\cite[p.~333]{ash-unstable} where it is credited to L.~Rudolph.
\begin{proof}
  The homology of the domain is isomorphic to $\sseq{Q}^1_{g,*} = H_*(\GL_g(\ZZ);\St_g \otimes\QQ)$, as is the homology of the codomain by duality (see \S\ref{subsubsection:tropical}).  Therefore the rational homology groups of domain and codomain are abstractly isomorphic.  Since the homology groups are also finite-dimensional, it suffices to construct a one-sided inverse to the map.

  To produce a one-sided inverse (on the level of rational homology), the strategy will be to choose a suitable ``simplicial structure'' on $P_g/\GL_g(\ZZ) \cup \{\infty\}$, namely a simplicial pointed set $X_\bu$ with finitely many non-degenerate simplices, and a homotopy equivalence $|X_\bu| \to P_g/\GL_g(\ZZ) \cup \{\infty\}$, and then defining the one-sided inverse on the simplicial chains of $X_\bu$ by sending each non-degenerate non-basepoint simplex to some preferred lift in $N_0 T_\bu(\QQ^g)/\GL_g(\ZZ) \cup \{\infty\}$.  We will not quite succeed in doing exactly this, but can achieve that simplices of $X_\bu$ come with a finite set of specified lifts.  The one-sided inverse is given by averaging those lifts.

  To construct the simplicial set $X_\bu$, we first choose an admissible decomposition of $P_g^\mathrm{rt}$ whose rays are all rank $1$,
  for instance the perfect cone decomposition.  Each cone $\sigma$ is then the convex hull of its extremal rays $\RR_{\geq 0} \to P_g^\mathrm{rt}$ which are of the form $t \mapsto t \psi^2$ for some $\psi \in (\QQ^g)^\vee \setminus\{0\}$.  Up to scaling by a non-zero rational number we can arrange that $\psi$ is \emph{normalized} to satisfy $\psi(\ZZ^g) = \ZZ$, and for normalized $\psi$ the form $\psi^2 \in P_g^\mathrm{rt}$ is uniquely determined by the projective class $[\psi] \in \mathbb{P}((\QQ^g)^\vee)$, because $\psi$ itself is determined up to a sign by the normalization condition.
  
  Any cone in the cone decomposition is then the image of a map of the form
  \begin{align*}
      \RR_{\geq 0}^{p} &\to P_g^\mathrm{rt}\\
      (t_1, \dots, t_p) &\mapsto \sum_{i = 1}^p t_i \psi_i^2
  \end{align*}
  for some normalized $\psi_i \in (\QQ^g)^\vee \setminus \{0\}$.  The map need not be injective unless the cone is simplicial, but it will always be a proper homotopy equivalence onto the cone (the inverse image of any point in the cone is a compact and convex subset of the octant).  If the $\psi_i$ span $(\QQ^g)^\vee$, then the open cone has image in $P_g$ and is also the image of the map
  \begin{align*}
      (\Delta^1 \setminus \partial \Delta^1)^{p} &\to P_g\\
      (s_1, \dots, s_p) &\mapsto \sum_{i = 1}^p \ell(s_i) \psi_i^2,
  \end{align*}
  where $\ell$ is as in~\eqref{eq:diffeo},
  analogous to~\eqref{eq:map-simplex-to-g} but without inequalities among the $s_i$.
  The admissible decomposition gives finitely many $\GL_g(\ZZ)$-orbits of such maps, whose image in $P_g/\GL_g(\ZZ)$ forms a decomposition into open cones (which may or may not be simplicial). 
  After 1-point compactifying, we obtain finitely many maps of the form
  \begin{equation}\label{eq:cube-in-locally-symmetric-space}
      |\Delta^1_\bu|^{p} \to (P_g/\GL_g(\ZZ)) \cup \{\infty\},
  \end{equation}
  one for each $\GL_g(\ZZ)$-orbit of cones
  spanned by $p$ many rays $t \mapsto t \psi_i^2$, with the property that the $\psi_i$ span $(\QQ^g)^\vee$.  The image of each such map is the closure of the corresponding cone, but the map need not be injective.
  However, the inverse image of a point in the cone is contractible.  If we pre-compose with the geometric realization of any of the top-dimensional simplices $\Delta^p_\bu \to (\Delta^1_\bu)^p$ then we obtain $p!$ many maps of the form~\eqref{eq:map-simplex-to-g} for each cone.

The maps $\Delta^p \to P_g/\GL_g(\ZZ) \cup \{\infty\}$ arising this way, $p!$ many for each cone in the decomposition with $p$ many extremal rays, can be regarded as elements of the pointed set $\mathrm{Sin}_p((P_g/\GL_g(\ZZ) \cup \{\infty\})$.  Together with all iterated faces and degeneracies of these elements, and with the basepoint, they form a simplicial subset
  \begin{equation*}
      X_\bu \subset \mathrm{Sin}_\bu((P_g/\GL_g(\ZZ) \cup\{\infty\})
  \end{equation*}
  with the property that the induced map $$|X_\bu| \to P_g/\GL_g(\ZZ) \cup \{\infty\}$$ has contractible fibers 
  and that $X$ has only finitely many non-degenerate simplices.  It follows from the Vietoris--Begle theorem that this map induces an isomorphism on homology.  The simplicial set $X_\bu$ has only finitely many non-degenerate simplices, on each of which the map is of the form~\eqref{eq:map-simplex-to-g} and therefore agrees with the map $|N_0 T_\bu(\QQ^g) \cup\{\infty\}| \to P_g/\GL_g(\ZZ) \cup \{\infty\}$ restricted to some $p$-simplex.  

  If we write $\widetilde C_*(X) = C_*(X_\bu,\{\infty\};\QQ)$ for the normalized rational chains relative to the base point, then we obtain a quasi-isomorphism
  \begin{equation*}
      \widetilde C_*(X_\bu) \stackrel{\simeq}\hookrightarrow \widetilde C_*(\mathrm{Sin}_\bu(P_g/\GL_g(\ZZ) \cup \{\infty\})) \simeq \widetilde C_*^\mathrm{sing}(P_g/\GL_g(\ZZ) \cup \{\infty\};\QQ)
  \end{equation*}
  from a finite chain complex $\widetilde C_*(X_\bu)$ similar in spirit to the ``Voronoi complexes'' of \cite{elbaz-vincent-gangl-soule-perfect}.  The main difference is that each cone  in the decomposition, of geometric dimension $d$ and spanned by $p$ many extremal rays, gives rise to one generator of homological degree $d$ in the Voronoi complex while in our complex it gives rise to $p!$ many generators of homological degree $p$ and also some lower-dimensional generators.

  Unfortunately, this does not quite identify $X_\bu$ with a simplicial subset of $N_0T_\bu(\QQ^g) \cup \{\infty\}$, for two reasons.  Firstly, the extremal rays of the cones in a cone decomposition are of the form $t \mapsto t \psi^2$ for some covector $\psi \in (\QQ^g)^\vee$ but $\psi$ is not quite canonically determined by the ray, only up to a sign.  Secondly, elements of $N_0T_p(\QQ^g)$ involve a total order $<$ which we cannot canonically produce from a $p$-simplex of $X_\bu$.  Altogether, we have explained a recipe by which a non-degenerate non-basepoint $\sigma \in X_p$ lifts in $2^p p!$ many ways to an element $\widetilde \sigma \in N_0T_p(\QQ^g)/\GL_g(\ZZ)$.  This is sufficient for defining a one-sided inverse on the level of rational homology though: let the diagonal map in the diagram
  \begin{equation*}
      \begin{tikzcd}
          & \widetilde C_*^\mathrm{sing}(|N_0T_\bu(\QQ^g)/\GL_g(\ZZ) \cup \{\infty\}|;\QQ) \dar\\
          \widetilde C_*(X_\bu) \rar["\simeq"] 
          \arrow[ur] & \widetilde C_*^\mathrm{sing}(P_g/\GL_g(\ZZ) \cup \{\infty\};\QQ)
      \end{tikzcd}
  \end{equation*}
  be defined by sending a non-degenerate non-basepoint simplex of $X_\bu$ to the average of the $2^p p!$ many lifts to $N_0 T_p(\QQ^g)$ that we have explained.
\end{proof}

\begin{remark} \label{rem:two-comparisons}
    The choice of diffeomorphism $\ell$ in~\eqref{eq:diffeo}  may be chosen either to be an increasing or decreasing function.  This gives two different isomorphisms between $\sseq{T}^1_{s,t}$ and $\sseq{Q}^1_{s,t}$, differing by the automorphism of $\sseq{Q}^*$ induced by dualizing projective modules (and reflecting the triangular diagrams~\eqref{eq:triangle-of-projectives}).  The two comparison isomorphisms differ by a sign on $\sseq{Q}^1_{1,0} \cong \QQ \cong \sseq{T}^1_{1,0}$.
\end{remark}

\begin{remark}
    Orthogonal direct sum of symmetric forms induces a map $P_g \times P_h \to P_{g+h}$.  Passing to orbits leads to a map
    \begin{equation*}
        (P_g/\GL_g(\ZZ)) \times (P_h/\GL_h(\ZZ)) \to (P_{g+h}/\GL_{g+h}(\ZZ))
    \end{equation*}
    which is proper.  Therefore it induces a map of one-point compactifications, giving a product on $\sseq{T}^1_{*,*}$, the $E^1$-page of the tropical spectral sequence.  The isomorphism
    \begin{equation*}
        \sseq{T}^1 \xrightarrow{\cong} \sseq{Q}^1
    \end{equation*}
    constructed above is an isomorphism of bigraded algebras with respect to this product on $\sseq{T}^1_{*,*}$ and the product on $\sseq{Q}^1_{*,*}$ constructed in Section~\ref{sec:copr-filtr-waldh}.  There, we also produce a coproduct, of which it seems more difficult to give a simple interpretation in $\sseq{T}$.
\end{remark}

\subsection{Canonical forms for \texorpdfstring{$\GL_n$}{GLn}} \label{sect: backgroundcanforms}
The space $P_g$ of positive-definite real symmetric  matrices of rank $g$ is  equipped with a right action of $\GL_g(\RR)$ given by $X \mapsto g^T X g$.  Classical invariant theory  provides  a differential $n$-form for every $n\geq 1 $
\[  \omega^n_X= \tr ( (X^{-1} dX)^n)\]
which is invariant under the action of $\GL_g(\RR)$. One shows that $\omega^n$ vanishes unless $n\equiv 1 \pmod{4}$; the case $n=1$ plays very little role in what follows.
The $\omega^n$ have a number of useful properties, including compatibility with block direct sums of matrices:
\begin{equation} \label{eqn: omegablockdirectsum} \omega^n_{X_1 \oplus X_2} = \omega^n_{X_1} + \omega^n_{X_2}
    \end{equation}
    and, in the case when $h>1$ is odd, the form $\omega^{2h-1}$ has the   vanishing property
\begin{equation} \label{eqn: omega2g1vanishing}
\omega^{2h-1}_X  =0  \qquad \hbox{ if } X \hbox{ has rank } g<h \ .
 \end{equation}
 By invariance, the forms $\omega^{4k+1}_X$  for $k>1$ define differential forms on the locally symmetric space $P_g/\GL_g(\ZZ)$. 
Borel showed that the graded exterior algebra they generate is isomorphic to its stable cohomology:  
\[  \bigwedge   \bigoplus_{k\geq 1} \omega^{4k+1}\RR   \  \overset{\sim}{\longrightarrow} \   \varprojlim_g H^{*}(P_g/\GL_g(\ZZ);\RR)\ .\]
Both sides of this equation have a Hopf algebra structure such that the (indecomposable) generators   $\omega^{4k+1}$ are primitive. This follows from \eqref{eqn: omegablockdirectsum}.  By taking primitives, Borel deduced that $K_{n}(\ZZ)\otimes\QQ \cong \QQ$ for  all  $n=4k+1$ with $k>1$,  and vanishes for all other $n>0$. 

We will need much more precise results about the cohomology of $\GL_n(\ZZ)$ in the unstable range. For this, let $g>1$ be odd and 
let $\Omega^{*}(g)$ denote the graded exterior algebra generated by $\omega^5,\ldots, \omega^{2g-1}$, which are the non-vanishing forms on $P_g/\GL_g(\ZZ).$  It is a direct sum
\[  \Omega^*(g) = \Omega^*_{c}(g) \oplus \Omega^*_{nc}(g)\]
where $\Omega^*_c(g)$ is the graded vector space  of forms `of compact type' given by the ideal generated by $\omega^{2g-1}$, and $\Omega^*_{nc}(g)$ is the graded algebra generated by $\omega^5,\ldots, \omega^{2g-5}$. The Hodge star operator interchanges these two spaces.  In \cite{brown-bordifications} it was shown that every element in $\Omega_c(g)$ defines a unique compactly supported cohomology class, giving two injective maps
\begin{eqnarray}
\Omega^{*}_{nc}(g) & \longrightarrow &  H^*(P_h/\GL_h(\ZZ);\RR)  \qquad \hbox{ for all } h\geq g, \label{eqn: Omegancinjects} \\
\Omega^*_{c}(g)[-1] & \longrightarrow &  H_c^*(P_g/\GL_g(\ZZ);\RR)\ , \label{eqn: Omegacinjects}
\end{eqnarray}
the first of which is a map of graded algebras, the second only a map of graded vector spaces. 
By taking the limit of the first injection, one obtains a stronger version  of Borel's theorem (which states that  \eqref{eqn: Omegancinjects} is injective in  degrees  $\leq g/4$ when $h=g$.) However, it is the second injective map \eqref{eqn: Omegacinjects} which is involved in a key definition.  

\begin{defn} \label{def:omega-compact} Set $\Omega_c(g)=0$ for $g$ even.  Denote by 
\[ \Omega^{*}_{c} = \bigoplus_{g>1  } \Omega_c^{*}(g)\]
 the  $\QQ$ vector space spanned by the canonical forms of compact type, bigraded as follows.  The {\em degree} and {\em genus} are given on generators by 
\begin{eqnarray*}
\on{deg}(\omega^{4i_1 + 1} \wedge \cdots \wedge \omega^{4i_k+1}) &= & \textstyle \Sigma_{j=1}^k (4i_j + 1)\\ 
\on{g}(\omega^{4i_1 + 1} \wedge \cdots \wedge \omega^{4i_k+1}) & = & 2i_k+1
\end{eqnarray*}
for integers $0<i_1 < \cdots < i_k.$  Then a generator $\omega$ is in bidegree $(\on{g}(\omega), \on{deg}(\omega)\!-\!\on{g} ).$
\end{defn}

  Although   the bigraded vector space $\Omega^{*}_c$ is formally isomorphic to the space of  elements of positive degree in the graded exterior algebra generated by  the  $\{\omega^{4k+1} \mid k\ge 1\}$,  
  it is  advisable  not to confuse the two.  The former is a bigraded vector space, while the latter is the positive degree elements of a Hopf algebra. Indeed, we shall prove that the image of $\Omega^{*}_c$ is primitive with respect to the coproduct we shall define, which is \emph{not} the case for non-trivial products of $\omega^{4k+1}$ in the stable cohomology of the general linear group. 
  The precise relationship between these two coproducts is via the Quillen spectral sequence,  explained below. 

\subsubsection{Canonical  tensor algebra} 
Denote the tensor algebra on a bigraded  vector space $V$   by 
\[ T(V) = \bigoplus_{n \geq 0 } V^{\otimes n}\ .\]
It is a connected, bigraded Hopf algebra with non-commutative product given by the tensor product, and (graded) cocommutative coproduct  
$ \Delta \colon T(V) \rightarrow T(V) \otimes T(V)$ dual to the shuffle product, with respect to which the elements of $V$ are primitive.  We shall denote elements  
 $v_1\otimes \ldots\otimes  v_n \in V^{\otimes n} \subset T(V)$ 
using the bar notation  $[v_1|\ldots | v_n]$. 

Let $\Omega^\ast_c[-1]$ denote the bigraded $\QQ$-vector space in which degree is shifted by $1$. So $\omega^{4i+1}$ has genus $2i+1$ and degree $(4i+1)+1$.
The tensor algebra $T(\Omega^*_{c}[-1] )$ 
 is a non-commutative,  (graded) cocommutative Hopf algebra. It is again bigraded by genus and degree minus genus,  where $[v_1|\ldots |v_n]$ has degree $\sum \mathrm{deg}(v_i)$ and genus $\sum \mathrm{g}(v_i)$. 
 Note that it has an additional grading by length of tensors, but only the associated filtration will play a role.

\section{Coproducts and filtrations in the Waldhausen construction}\label{sec:copr-filtr-waldh}

Recall that the rational homology of an infinite loop space such as $BK(\ZZ)$ is naturally a graded commutative and cocommutative Hopf algebra, with coproduct given by push-forward under the diagonal. However, this coproduct does not respect the filtration on $BK(\ZZ)$ and hence does not give rise to a coproduct on the pages of the Quillen spectral sequence.

Our goal in this section is to construct a \emph{filtered coproduct} on $BK(\ZZ)$. We do so via the simplicial operation of edgewise subdivision from \cite[Lemma~1.1]{bokstedt-hsiang-madsen-cyclotomic}. The resulting  coproduct induces a bigraded associative (but not cocommutative) coproduct on each page of the Quillen spectral sequence. It is homotopic to the diagonal and hence induces the usual commutative coproduct on the abutment.

\subsection{A filtered coproduct from edgewise subdivision}
The goal now is to define a map
\[BK(\ZZ) \to BK(\ZZ) \times BK(\ZZ)\]
that respects the filtration~\eqref{eq:fil-BKZ}, i.e., with the property that the restriction to the $n$th filtration $\Fil_n = \Fil_n BK(\ZZ)$ factors as
\[\Fil_n \to \bigcup_{a+b=n} \Fil_a \times \Fil_b \]
so that it induces a coproduct on the associated relative homology spectral sequence.  We will use the \emph{edgewise subdivision} of \cite[Section 1]{bokstedt-hsiang-madsen-cyclotomic}, in turn inspired by \cite{Segal-configuration-spaces}.  We recall the definition.

\defnow{\label{def:es} For any category $\cS$, given a simplicial object $X\col \Delta^{\mathrm{op}}\to \cS$, the edgewise subdivision $\es(X)$ of $X$ is the simplicial object in $\cS$ 
\[\Delta\to\Delta\xrightarrow{X}\cS\]
where the functor $\Delta\to \Delta$ is \[[p]\mapsto [p]\sqcup [p].\] 
Here, the set $[p]\sqcup [p]$ is linearly ordered by concatenation. 
}

In other words, we identify  $[p]\sqcup[p]\cong [1]\times [p]$ with the lexicographic ordering.

\remnow{
Note, in particular, that $\es_p(X) = X_{2p+1}$ for each $p\ge 0.$}

For a simplicial set $X\col \Delta^{\mathrm{op}}\to \mathsf{Set}$, there is a natural homeomorphism
\begin{equation*}
  |\es(X)| \xrightarrow{\phi_X} |X|.
\end{equation*}
It is the unique natural transformation that is affine on simplices and sends each vertex $x \in \es_0(X) = X_1$ to the mid-point of the corresponding edge  or vertex (degenerate 1-simplex) in $|X|$.  To check that this is a homeomorphism it suffices to consider the case where $X = \Delta(-,[n])$ is a simplex, since both sides preserve colimits.  See  \cite[Lemma 1.1]{bokstedt-hsiang-madsen-cyclotomic}.  

We now consider how the (inverse of the) homeomorphism $\phi_X$ interacts with filtrations. 
\begin{lemma}\label{lemma:subdividing-and-filtering}
  Let $X_\bullet$ be a filtered simplicial set, i.e., a simplicial set $X$ and a sequence of subsets $\Fil_0 X_p \subset \Fil_1 X_p \subset \dots$ for each $p$ such that each $\Fil_n X_\bullet$ forms a simplicial subset, i.e.\ $d_i \Fil_nX_{p+1} \subset \Fil_n X_p$ and $s_i \Fil_n X_p \subset \Fil_n X_{p+1}$ for all $p$ and all $i$.  For any filtered simplicial space, give the geometric realization $|X|$ the induced filtration: $\Fil_n|X|$ is the image of the realization of the inclusion $|\Fil_nX| \to |X|$.  Finally, give the simplicial set $\es_\bullet(X)$ the filtration with $\Fil_n \es_p(X) = \es_p(\Fil_n X_\bullet) = \Fil_n X_{2p+1}$.  
  
  Then the inverse homeomorphism 
  \begin{equation*}
    |X| \xrightarrow{\phi_X^{-1}} |\es(X)|
  \end{equation*}
  is a filtered map: it satisfies $\phi_X^{-1}(\Fil_n|X|) \subset \Fil_n|\es(X)|$.
\end{lemma}
\begin{proof}
  Let $f \colon X_p \to \NN$ be the function such that $f(\sigma) = n$ if $\sigma \in \Fil_n X_p \setminus \Fil_{n-1}X_p$.  Then each $\Fil_n X_\bullet$ being a simplicial subset implies that $f(\theta^* \sigma) \leq f(\sigma)$ for any morphism $\theta \colon [p] \to [q]$ in $\Delta$ and any $\sigma \in X_q$, but then the simplicial identity $d_i \circ s_i (\sigma) = \sigma$ implies $f(s_i \sigma) = f(\sigma)$.  In other words, the function $f$ is determined by its values on non-degenerate simplices.  

  Now a point $x \in |X|$ will be in the image of the canonical injection
  \begin{equation*}
    \{\sigma_x\} \times (\Delta^p \setminus \partial \Delta^p) \hookrightarrow |X|
  \end{equation*}
  for a unique $p \in \NN$ and a unique non-degenerate $\sigma_x \in X_p$.  Then $\phi_X^{-1}\colon |X| \xrightarrow{\sim} |\es(X)|$ sends $x$ to a point in the image of the canonical injection
  \begin{equation*}
    \{\tau_x\} \times (\Delta^q \setminus \partial \Delta^q) \hookrightarrow |\es(X)|
  \end{equation*}
  for a unique $q \in \{0, \dots, p\}$ and a unique $\tau_x \in \es_q(X) = X_{2q+1}$.

  By naturality of the homeomorphism, there must exist a morphism $\theta_x \colon [p] \to [2q+1]$ in $\Delta$ such that $\theta_x^*(\sigma_x) = \tau_x$.  This implies $f(\tau_x) \leq f(\sigma_x)$, and hence
  \begin{equation*}
    f(\phi_X^{-1}(x)) \leq f(x)
  \end{equation*}
  as desired.
\end{proof}

Let us briefly discuss edgewise subdivision of bisimplicial sets $X\col \Delta^{\mathrm{op}} \times \Delta^{\mathrm{op}} \to \mathsf{Set}$, for which we write $X_{p,q} = X([p],[q])$.  The notion itself is symmetric in $p$ and $q$, but for the applications we have in mind the two simplicial directions will play quite different roles.  In particular, when $\cC$ is a Waldhausen category we will consider the bisimplicial set with 
\begin{equation*}
    X_{p,q} = N_q S_p(\cC)
\end{equation*}
and use the \emph{first} simplicial direction $p$ as the ``$S_\bullet$-direction.''  In this and other examples we consider, the second simplicial direction will play a more auxiliary role, mainly as a means to encode the simplicial topological space $[p] \mapsto |X_{p,\bu}|$, while in the $p$-direction we will make more use of the simplicial structure.  In particular, we will write
\begin{equation*}
    \es(X) = \es([p] \mapsto X_{p,\bu})
\end{equation*}
for edgewise subdivision in the $p$-direction.  In other words, on objects $[p]$ and $[q]$ of $\Delta$ we have
\begin{equation*}
    \es(X)_{p,q} = X_{2p+1,q} = X([p] \sqcup [p],[q]).
\end{equation*}
\propnow{Let $X \col \Delta^{\mathrm{op}} \times \Delta^{\mathrm{op}} \to \mathsf{Set}$ be a bisimplicial set. Then $|\es(X)|$ is naturally homeomorphic to $|X|$.}
\proofnow{ We have
\begin{eqnarray*}
    |X|&\cong & |[q]\mapsto([p]\mapsto X_{p,q})| \cong |[q]\mapsto ([p]\mapsto X_{2p+1,q})| 
    \cong |[p]\mapsto ([q]\mapsto X_{2p+1,q})| \cong |\es(X)|,
\end{eqnarray*}
using the fact that edgewise subdivision is a natural homeomorphism for simplicial sets, and that the two different geometric realizations of bisimplicial sets explained above are naturally homeomorphic.  All of the other intermediate homeomorphisms are natural as well.
}

Here is the main definition.
\defnow{\label{def:the-nt} Let $\cS$ be a category with finite products.  Let \[X\col \Delta^{\mathrm{op}}\to \cS\]
be a simplicial object in $\cS$, and let $\es(X)$ be its edgewise subdivision.  
Define a natural transformation 
\begin{equation}\label{eq:the-nt}\es(X) \Rightarrow X\times X\end{equation}
whose component at $[p]$, for each $p\ge 0$,
\[\es(X)_p = X_{2p+1} \to X_p \times X_p,\]
is given by the two maps $X_{2p+1}\to X_p$ induced by the two order preserving inclusions \[[p]\to [p]\sqcup [p] \cong [2p+1]\] onto the first $p+1$, respectively the last $p+1$, elements.
}

\begin{remark}
When $\cS$ is the category of simplicial sets, $X \colon \Delta^\mathrm{op} \to \cS$ is a bisimplicial set.  Using that geometric realization commutes with finite products, the natural transformation $\es(X) \Rightarrow X \times X$ defined above leads to a map 
\begin{equation}\label{eq:Alexander-Whitney}
    |X| \xrightarrow{\approx} |\es(X)| \to |X| \times |X|.
\end{equation}
As a map of spaces, this map is not especially interesting by itself---as we will see momentarily, it is naturally homotopic to the diagonal map $x \mapsto (x,x)$ of the space $|X|$.  What we will see is that in interesting examples, the bisimplicial set $X$ comes with a filtration making~\eqref{eq:Alexander-Whitney} into a map of filtered spaces and~\eqref{eq:Alexander-Whitney} will induce an interesting map of associated gradeds, while the actual diagonal map of $|X|$ will not be a filtered map.
\end{remark}

\begin{example}
Our main example comes from the simplicial object $X=S_\bu(\Proj{\ZZ})\colon [p] \mapsto S_p(\Proj{\ZZ})$ in small groupoids, for the Waldhausen category $\Proj{\ZZ}$ of finitely generated projective $\ZZ$-modules (Example~\ref{ex:projectives}).  In this case the functor
\[\es_p (S_\bullet (\Proj{\ZZ})) \to S_p(\Proj{\ZZ}) \times S_p(\Proj{\ZZ})\]
is given on objects by
\[(0\to V_1\to\cdots\to V_{2p+1})\mapsto ((0\to V_1\to\cdots\to V_p),(0\to V_{p+2}/V_{p+1} \to\cdots\to V_{2p+1}/V_{p+1})),\]
omitting the chosen quotients $V_i/V_j$ from the notation.   

We recall the filtration on $X=S_\bu(\Proj{\ZZ})$ described previously in~\eqref{eq:fil-BKZ}.  For integers $p,q$, an element of $X_{p,q}$ is a chain of $q$ morphisms of triangular diagrams of projective modules of the form~\eqref{eq:triangle-of-projectives}.  The morphisms of triangular diagrams are required to be isomorphisms in each component.  Define the rank of such an element to be the rank of the top-right projective $\ZZ$-module in any of the $q+1$ triangular diagrams; this rank is well-defined since the $q+1$ modules involved are related by isomorphisms.  
For any $n$, the property of having rank at most $n$ 
is preserved by all face and degeneracy maps of $X$, and thus rank induces a filtered bisimplicial set in which
\[(\Fil_n X)_{p,q} = X_{p,q} \cap \mathrm{rank}^{-1}((-\infty, n]).\]

The rank function on simplices of $X$ induces a rank function on the simplices of $\es(X)$, as well as a rank function on the simplices of $X\times X$ by additivity: the rank of $(x,y) \in (X\times X)_{p,q} = X_{p,q}\times X_{p,q}$ is $\mathrm{rank}(x) + \mathrm{rank}(y)$.  These rank functions are compatible with face and degeneracy maps, inducing filtrations on both $\es(X)$ and $X\times X$.  We note that~\eqref{eq:the-nt} respects the filtration, since 
\begin{equation}\label{eq:rank-ineq}
\mathrm{rank}(V_{2p+1}) \ge \mathrm{rank}(V_p) + \mathrm{rank}(V_{2p+1}/V_{p+1}).
\end{equation}

Therefore, the natural transformation $\es(X)\Rightarrow X\times X$ is a map of filtered bisimplicial sets, and geometric realization yields a map
\[|X| \xrightarrow{\approx} |\es(X)|\longrightarrow |X\times X|\cong |X|\times|X|\]
of filtered spaces.

In Proposition~\ref{prop:(4)}, we shall study a product map $m\col X\times X\to X$, given by direct sums of triangular diagrams.  The observation that the rank of a direct sum of projective modules is the sum of the ranks will imply that $m$ is compatible with filtrations; see Remark~\ref{rem:product-compatible-with-filtration} below.

\end{example}

Returning briefly to the general case of an arbitrary (unfiltered) simplicial set $X$, let us justify the claim above that~\eqref{eq:Alexander-Whitney} is homotopic to the diagonal.
\propnow{
\label{prop:go-to-id} Suppose we are given maps $\Phi_X\col |X|\to |X|$ for every simplicial set $X$ that are natural in $X$, i.e., they assemble into a natural transformation $\Phi$.  Then 
$\Phi$ is naturally homotopic to the identity.  Precisely,  there is a homotopy
\[|X|\times[0,1]\to|X|\]
from $\Phi_X$ to $\mathrm{id}_X$ for each $X$, and these homotopies are natural in $X$.
}
\proofnow{First, for $X = \Delta^n = \Hom_{\Delta}(-,[n])$ an $n$-simplex, one may take a straight line homotopy 
\[|\Delta^n| \times [0,1]\to |\Delta^n|\]
from $\Phi_X$ to $\mathrm{id_X}$, since $\Delta^n$ is a convex subset of $\RR^n$.  These straight line homotopies are natural in all morphisms of simplicial sets, in particular face maps and degeneracy maps.  An arbitrary simplicial set is a colimit of simplices, so the result follows.}

\cornow{\label{cor:homotopic-to-diagonal} For simplicial sets $X$, the geometric realization 
\begin{equation}\label{eq:homotopic-to-diagonal}
    |X|\cong |\es(X)|\to |X|\times |X|
\end{equation} of the morphism in~\eqref{eq:the-nt} is naturally homotopic to the diagonal map.}
\proofnow{The natural transformation $\es(X)\Rightarrow X\times X$ is assembled from two natural transformations $\es(X)\to X$, with the property that the corresponding maps $|\es(X)|\to |X|$ are homotopic to the identity by~Proposition~\ref{prop:go-to-id}. Therefore the map $|\es(X)|\to |X|\times |X|$ is naturally homotopic to the diagonal.}

\remnow{\label{rem:(6)}Suppose as in Proposition~\ref{prop:go-to-id}, we are given maps $\Phi_X\col |X|\to |X|$ for every simplicial set $X$ that are natural in $X$.  Now if $X$ is a {\em bisimplicial} set, for each $p\ge 0$, let $X_p$ denote the simplicial set with $(X_p)_q = X_{p,q}$. Then we obtain a map $\Phi_X\col |X|\to |X|$ by gluing the maps $ \Phi_X\col |X_p|\to |X_p|$ 
for each $p$.  Then Proposition~\ref{prop:go-to-id} implies that $\Phi_X$ is also naturally homotopic to the identity. 

In particular, for a {\em bisimplicial} set $X$, the map $|X|\cong|\es(X)|\to |X|\times |X|$  is again homotopic to the diagonal, naturally in $X$.  
}

\begin{proposition}\label{prop:(3)}
Suppose 
\[F_0 X \subset F_1 X \subset \cdots \subset X \col \Delta^{\mathrm{op}}\times\Delta^{\mathrm{op}}\to\mathsf{Set}\]
is a filtered bisimplicial set.  Define a filtration on the bisimplicial set $X\times X$ by \[F_s(X\times X)_{p,q} = \bigcup_{u+v \le s} (F_{u}X \times F_{v}X)_{p,q},\]
and similarly for $X\times X\times X$:
\[F_s(X\times X \times X)_{p,q} = \bigcup_{t+u+v \le s} (F_{t}X \times F_{u}X \times F_{v}X)_{p,q}.\]
Let $\Phi\col \es(X)\Rightarrow X\times X$ be the natural transformation~\eqref{eq:the-nt}. 
Suppose that $\Phi$ respects the filtrations, i.e. for each $s$, the natural transformation $\es(F_s X) \Rightarrow F_s X \times F_s X$ factors as  \[\es(F_s X) \Rightarrow   F_s (X\times X) \Rightarrow F_s X \times F_s X.\]

Then the diagram
\[\squarediagramlabel{|X|}{|X|\!\times\!|X|}{|X|\!\times\!|X|}{|X|\!\times\!|X|\!\times\! |X|}{|\Phi|}{|\Phi|}{|\Phi|\times\mathrm{id}}{\mathrm{id}\times|\Phi|}\]
commutes up to a homotopy
\[[0,1]\times |X| \to |X|\times|X|\times|X|.\]
Moreover, the homotopy may be chosen so that, for each $s\ge 0$, it restricts to a map
\[[0,1]\times |F_s X| \to |F_s(X\times X\times X)|.\]

\end{proposition}
\begin{proof}

We shall construct a map 
\begin{equation*} G_X\col [0,1]\times |X|\to |X|,\end{equation*}
for every bisimplicial set $X$, which is natural in $X$.     The map $G_X$ will have the property that $G_X(0,-)\col |X|\to |X|$ is the identity map, and, moreover, the composition
\begin{equation}\label{eq:G_X-composed}|X|\times [0,1]\xra{G_X}|X|\xra{(\mathrm{id}\times \Phi)\circ \Phi} |X|\times|X|\times|X| \end{equation}
is a natural homotopy from $(\mathrm{id}\times \Phi)\circ \Phi$ to  $(\Phi\times \mathrm{id})\circ \Phi$.  

Suppose we construct such maps $G_X$, natural in bisimplicial sets $X$. We now explain how this is enough to prove the Proposition. Suppose that $X$ is a filtered bisimplicial set and suppose $\Phi\col \es(X)\Rightarrow X\times X$ respects the filtration.  Then for each $s\ge 0$, the map $\Phi$ restricts to a map \[|F_s(X)|\to |F_s(X\times X)|,\] and similarly, the map $(\mathrm{id}\times \Phi)\circ \Phi$, and also the map $(\Phi\times \mathrm{id})\circ \Phi$, restricts to a map
\begin{equation}\label{eq:restricted-triple}|F_s(X)| \to |F_s(X\times X\times X)|.\end{equation}
for each $s\ge 0$.  
Then applying~\eqref{eq:G_X-composed} to $F_s(X)$ and combining with~\eqref{eq:restricted-triple} yields a homotopy 
\[|F_s(X)| \times [0,1] \longrightarrow |F_s(X\times X\times X)|\]
from $(\mathrm{id}\times \Phi)\circ \Phi$ to  $(\Phi\times \mathrm{id})\circ \Phi$, as desired.

Now we construct the maps $G_X$.  They will be defined first on the level of individual simplices, shown to be natural in face and degeneracy maps, and hence defined for all simplicial sets.  Since they are natural in morphisms of simplicial sets, they then define maps on bisimplicial sets.

Consider Cartesian coordinates on the standard $n$-simplex for each $n$: 
\begin{equation}\label{eq:cartesian-coords}\Delta^n \cong \{(t_1,\ldots,t_n)\in \RR^n \mid 0 \le t_1 \le t_2 \le \cdots \le t_n\le 1\}.\end{equation}

First let $g\col [0,1]\to [0,1]$ be any nondecreasing continuous function with $g(0) = 0$ and $g(1) = 1$.  Later on we will specialize $g$ to a particular piecewise linear function. 
By applying $g$ coordinatewise to~\eqref{eq:cartesian-coords}, $g$ defines self-maps on each simplex $g_n\col \Delta^n \to \Delta^n$ given by \[(t_1,\ldots,t_n) \mapsto (g(t_1),\ldots,g(t_n)).\]  For each $n$, the maps $g_n$ and $g_{n-1}$ are compatible with the $n+1$ face maps $\Delta^{n-1}\to \Delta^n$, which send $(t_1,\ldots,t_{n-1})$ to 
\[(0,t_1,\ldots,t_{n-1}), \quad (t_1,\ldots,t_i,t_i,\ldots,t_{n-1})\text{ for $i=1,\ldots,n-1$, and }(t_1,\ldots,t_{n-1},1)\]
respectively, as well as the $n$ degeneracy maps $\Delta^n\to \Delta^{n-1}$ that send $(t_1,\ldots,t_n)$ to
\[(t_1,\ldots,\widehat{t_i},\ldots,t_n)\]
for each $i=1,\ldots,n$, respectively.  Therefore, $g$ determines maps $g_X\col |X|\to |X|$ for any simplicial set $X$, naturally in $X$.  If $X$ is a bisimplicial set, then $g$ also determines a map $|X_{p,\bullet}| \to |X_{p,\bullet}|$ for each $p$, and by gluing these, we get a map $g_X\col |X|\to |X|$ as well, which is natural in morphisms of bisimplicial sets.  Finally, let
\begin{equation}\label{eq:def-G_X} G_X \col [0,1]\times |X| \to |X|\end{equation}
be the straight line homotopy from the identity map to $g_X$; this is again natural in morphisms of bisimplicial sets.

Now we specialize to a particular $g$.  Let $g \col [0,1] \to [0,1]$ be the increasing, piecewise linear map that sends the intervals 
\[ \big[0,\tfrac{1}{4}\big], \quad \big[\tfrac{1}{4},\tfrac{1}{2}\big], \quad \big[\tfrac{1}{2}, 1\big]\]
linearly to the intervals
\[ \big[0,\tfrac{1}{2}\big], \quad \big[\tfrac{1}{2},\tfrac{3}{4}\big], \quad \big[\tfrac{3}{4}, 1\big]\]
respectively.  What remains only to show is that the composition
\[|X|\xra{g_X}|X|\xra{(\textrm{id}\times\Phi)\circ \Phi} |X|\times|X|\times|X|\]
is equal to $(\Phi\times\textrm{id})\circ \Phi.$  It suffices to perform this calculation in the case of a simplex \[X = \Delta^n = \Hom_{\Delta}(-,[n])\col \Delta^{\mathrm{op}}\to \mathsf{Set}.\]
We first state the following lemma, without proof:

\lemnow{\label{lem:the-guts} Let $X$ be an $n$-simplex.  In Cartesian coordinates, the map \[\Phi\col |X| \to |X|\times |X|\] is given by
\[(t_1,\ldots,t_n) \mapsto ((t_1',\ldots, t_i',1,\ldots,1),(0,\ldots,0,t_{i+1}',\ldots,t_n'))\]
for 
\[0\le t_1\le \cdots \le t_i \le \tfrac{1}{2} \le t_{i+1}\le \cdots\le t_{n}\le 1\]
and where $t'_\ell$ denotes the result of applying to $t_\ell$ the appropriate linear rescaling
\[\big[0,\tfrac{1}{2}\big]\to[0,1]\quad\text{ or }\quad \big[\tfrac{1}{2},1\big]\to[0,1].\]
}
Continue to let $X$ be an $n$-simplex.  Applying Lemma~\ref{lem:the-guts}, in Cartesian coordinates, the map
$(\mathrm{id}\times \Phi)\circ \Phi$ is given by
\[(t_1,\ldots,t_n) \mapsto ((t_1',\ldots,t_i',1,\ldots,1),(0,\ldots,0,t_{i+1}',\ldots,t_j',1,\ldots,1),(0,\ldots,0,t_{j+1}',\ldots,t_n'))\]
for 
\[0\le t_1\le\cdots\le t_i \le \tfrac{1}{2} \le t_{i+1}\le \cdots \le t_j \le \tfrac{3}{4} \le t_{j+1} \le \cdots \le t_n\le 1\]
and $t_\ell'$ denotes the result of applying to $t_\ell$ the appropriate linear rescaling
\[\big[0,\tfrac{1}{2}\big]\to [0,1], \quad \big[\tfrac{1}{2}, \tfrac{3}{4}\big]\to [0,1], \quad \text{ or}\quad \big[\tfrac{3}{4},1\big]\to[0,1]. \]
And the map $(\Phi \times \mathrm{id}) \circ \Phi$ has the same description as above, replacing the numbers $\tfrac{1}{2}$ and $\tfrac{3}{4}$ with $\tfrac{1}{4}$ and $\tfrac{1}{2}$ respectively.  The equality
\[(\mathrm{id}\times\Phi)\circ \Phi \circ g_X = (\Phi\times\mathrm{id})\circ \Phi\]
is now clear. This proves Proposition~\ref{prop:(3)}.
\end{proof}

For $X = N_\bu S_\bu (\Proj{\ZZ})$ with $\Proj{\ZZ}$ the Waldhausen category of finitely generated projective $\ZZ$-modules, $|X| = BK(\ZZ)$ and we will use the filtered map $|X| \approx |\es(X)| \to |X| \times |X|$ to induce a \emph{coproduct} on the Quillen spectral sequence.  The \emph{product} will be induced by a filtered map $|X| \times |X| \to |X|$ which is in some ways easier to comprehend, but depends on extra structure on the category $\Proj{\ZZ}$, namely the symmetric monoidal structure given by direct sum of free $\ZZ$-modules.  Strictly speaking, the direct sum operation $(P,P') \mapsto P \oplus P'$ is a choice (at least of the set underlying $P \oplus P'$), because of this, the definition of the product map $|X| \times |X| \to |X|$ looks a bit lengthy when spelled out, but hopefully will not be surprising.

\begin{proposition}\label{prop:(4)} 
Let \[X = N_\bu S_\bu (\Proj{\ZZ}) \col \Delta^{\mathrm{op}}\times \Delta^{\mathrm{op}} \to \mathsf{Set}\] be the bisimplicial set with \[X_{p,q} = N_q S_p(\Proj{\ZZ})\] as before.
Define a product \begin{equation}\label{eq:product-definition} m\col X\times X\to X\end{equation}
given by maps \[(X\times X)_{p,q} = X_{p,q}\times X_{p,q} \to X_{p,q}\]
which are direct sums of composable chains of morphisms of triangular diagrams.  
 Precisely, we first choose for each pair of objects $V, V'$ of $\Proj{\ZZ}$ an object $m(V,V')$ and morphisms \[V \to m(V,V') \leftarrow V'\] satisfying the universal property of coproducts.  In other words, $m(V,V')$ is a chosen model for the direct sum $V \oplus V'$, and there is a canonical associator and symmetry making $(\Proj{\ZZ},m)$ into a symmetric monoidal category.  As usual, any two choices of $m$'s will be canonically isomorphic, although the functions $N_0 \Proj{\ZZ} \times N_0\Proj{\ZZ} \to N_0\Proj{\ZZ}$ need not be equal.  Relatedly, the product map $m \colon |X| \times |X| \to |X|$ that we will define depends on the choice of $m$, although its homotopy class as a filtered map will not.  From now on we will fix such a choice and write simply $V \oplus V'$ for the chosen object $m(V,V')$.

Suppose now  we are given two elements of $X_{p,q}$, given (after suppressing chosen cokernels from the notation, in other words suppressing from the notation all but the first row of each triangular diagram~\eqref{eq:triangle-of-projectives}) by rectangular diagrams
\[\xymatrix@R=4mm@C=6mm{0 \ar[r] & V_1^0\ar@{>->}[r]\ar[d]_{\cong}  & \cdots\ar@{>->}[r] & V_p^0\ar[d]_{\cong}\\ & \vdots\ar[d]_{\cong} &\cdots &\vdots\ar[d]_{\cong} \\0 \ar[r] & V_1^q\ar@{>->}[r]  & \cdots\ar@{>->}[r] & V_p^q
},\qquad \xymatrix@R=4mm@C=6mm{0 \ar[r] & W_1^0\ar@{>->}[r]\ar[d]_{\cong}  & \cdots\ar@{>->}[r] & W_p^0\ar[d]_{\cong}\\ & \vdots\ar[d]_{\cong} &\cdots&\vdots\ar[d]_{\cong} \\0 \ar[r] & W_1^q\ar@{>->}[r]  & \cdots\ar@{>->}[r] & W_p^q.
} \]
The notation that has been suppressed includes a choice of cokernel $V^k_{i,j}$ for each map $V^k_i \to V^k_j$.  For ease of notation, we shall abbreviate rectangular diagrams of this form further, so that the diagrams above are abbreviated as 
\[0\to V_1^\bu \to \cdots \to V_p^\bu,\qquad 0\to W_1^\bu \to \cdots \to W_p^\bu.\]
Then their product is defined to be the diagram
\[V_1^\bu \oplus W_1^\bu\to \cdots\to V_p^\bu\oplus W_p^\bu.\]
The chosen cokernels, which still do not appear in the notation, are the direct sums of the existing choices: namely, the chosen cokernel for 
\[V_i^k\oplus W_i^k \to V_j^k \oplus W_j^k\]
is $V_{i,j}^k \oplus W_{i,j}^k$.
Write $m\col |X|\times |X| \to |X|$ and $\Phi\col |X|\to |X|\times |X|$ also for the maps on spaces.
Then we claim that product and coproduct are compatible:
\[\xymatrix@C=36mm{|X|\times |X|\ar[r]^m \ar[d]^{\Phi\times \Phi} & |X|\ar[d]^{\Phi} \\ |X|\!\times\! |X|\!\times\!|X|\!\times\!|X| \ar[r]^{(m\times m) \circ (\mathrm{id}\times s \times \mathrm{id})} & |X|\!\times\! |X|}\]
where $s\col X\times X \to X\times X$ exchanges first and second coordinate.  
The lower horizontal map is thus \[(x_1,x_2,x_3,x_4)\mapsto (m(x_1,x_3), m(x_2, x_4)).\]
\end{proposition}
\proofnow{It suffices to show that the compositions 
\[\es(X) \times \es(X) \xra{m} \es(X) \xra{\Phi} X\times X\]
and
\[\es(X)\times \es(X) \xra{\Phi\times \Phi} X\times X\times X\times X \xra{(m\times m)\circ (\mathrm{id}\times s \times \mathrm{id})} X\times X\]
are equal as morphisms of bisimplicial sets.  This is direct from the definitions: both maps send a pair of $(p,q)$-simplices of $\es(X)$
\[((0\to V_1^\bu\to \cdots \to V_{2p+1}^\bu), (0\to W_1^\bu\to\cdots\to W_{2p+1}^\bu))\]
to
\begin{equation*} \label{equ:eulerchar}
\resizebox{.95\hsize}{!}{
$\begin{aligned}
   ((0\to V_1^\bu\oplus W_1^\bu \to \cdots \to V_{2p+1}^\bu\oplus W_{2p+1}^\bu),
(0\to V_{p+1,p+2}^\bu\oplus W_{p+1,p+2}^\bu\to \cdots\to V_{p+1,2p+1}^\bu\oplus W_{p+1,2p+1}^\bu)),
    \end{aligned}$
    }
\end{equation*}
and the proposition follows.
}
\remnow{\label{rem:product-compatible-with-filtration} We also record that the product $m$ is evidently compatible with the filtration on $X$.  Precisely, $m$ restricts to a map
\begin{equation}\label{eq:product-compatible-with-filtration}m\col |F_s X| \times |F_t X| \to |F_{s+t} X|
\end{equation}
for all $s,t\ge 0$.
}

\remnow{\label{rem:mult-structures} We also need the following generalities on multiplicative spectral sequences.

Given a filtered space and a product that respects filtrations, we obtain (at this level of generality) a spectral sequence on relative singular homology with products on each page, and all differentials on each page are derivations of the product (Leibniz rule), and the isomorphism from $H_*(E_r) \to E_{r+1}$ is a multiplicative isomorphism with respect to the product on $H_*(E_r)$ induced from that on $E_r$. 
}

\begin{proposition}\label{prop:(5)}
Let \[X = N_\bu S_\bu (\Proj{\ZZ})\col \Delta^{\mathrm{op}}\times \Delta^{\mathrm{op}}\to \mathsf{Set},\]
with its rank filtration.  Then the product $m\col X\times X\to X$ is commutative up to homotopy, respecting the filtration. In other words, there is a homotopy
\[H\col |X|\times |X| \times [0,1]\to |X|\]
from $m$ to $m\circ s$, where $s\col X\times X\to X\times X$ switches coordinates. Furthermore, $H$ respects the  filtration, in that it restricts to a map
\[H\col |F_t X| \times |F_u X| \times [0,1]\to |F_{t+u}X|\]
for all $t,u\ge 0$.  
\end{proposition}
\begin{proof}
    We shall construct a morphism of bisimplicial sets
    \[X\times X \times \Delta^{1,0} \to X\]
whose geometric realization is the desired $H$.  Here \[\Delta^{1,0} = \Hom_{\Delta\times \Delta}(-,([1], [0])),\]
whose geometric realization is an interval $[0,1]$.
For ease of notation, abbreviate
\[0\to V_1^\bu \to \cdots \to V_p^\bu\]
for a $(p,q)$-simplex of $X$, as before. The notation stands for a chain of $q$ morphisms of triangular diagrams of projective $\ZZ$-modules of size $p$.

For $i=0,\ldots,p+1$, let $f_i\in \Delta([p],[1])$ be the morphism with $f_i (x) = 0$ if $x<i$ and $f_i(x) = 1$ if $x\ge i$.
An element of $(X\times X \times \Delta^{1,0})_{p,q}$ is a triple
\[(0\to V_1^\bu \to \cdots \to V_p^\bu, 0\to W_1^\bu \to \cdots \to W_p^\bu, f_i)\]
for some $i\in \{0,\ldots p+1\}$. It is sent by $H$ to
\[0 \to V_1^\bu \oplus W_1^\bu\to \cdots \to V_{i-1}^\bu \oplus W_{i-1}^\bu \to W_i^\bu\oplus V_i^\bu \to \cdots\to W_p^\bu\oplus V_p^\bu.\]
These maps \[(X\times X \times \Delta^{1,0})_{p,q}\to X_{p,q}\]
are compatible with face and degeneracy maps, so we have produced a map of bisimplicial sets, which restricts to $m$ and $m\circ s$ at $t=0$ and $t=1$, respectively.
Finally, $H$ obviously respects filtration: it restricts to a map of bisimplicial sets, for any $t,u\ge 0$,
\[F_t X \times F_u X \times \Delta^{1,0} \to F_{t+u} X,\]
since the rank of a direct sum of modules of rank $t$ and $u$ is $t+u$.
\end{proof}

\subsection{Monoidality for the spectral sequence of a filtered space}  

We recall a certain monoidality of the spectral sequence associated to a filtered space.  This seems well known in the algebraic topology literature, so we list the precise statements we need and give some references: \cite{cartan-eilenberg-homological, douady-suite-spectrale-adams, douady-suite-spectrale-multiplicative}. See also \cite[Chapter 6]{rognes-spectral} and \cite{goette-mathoverflow}.  

\subsubsection{Filtered spaces}

In this section we write $H_*$ for singular homology with coefficients in some field.  Recall that for a topological space $X$ filtered by subcomplexes $\Fil_t X \subset X$, there is a spectral sequence 
\begin{equation*}
  E^1_{s,t} = E^1_{s,t}(X) = H_{s+t}(\Fil_t X,\Fil_{t-1} X) \Rightarrow H_*(X),
\end{equation*}
with convergence assuming the filtration is bounded below, for instance $\Fil_{-1} X = \emptyset$, and exhaustive: $\colim_t H_*(\Fil_t X) \to H_*(X)$ is an isomorphism.  For definiteness, let us work with the construction of this spectral sequence given in \cite[Chap.\ XV, \S7]{cartan-eilenberg-homological}, see especially Example 3 on page 335.  This agrees with the construction in \cite[Section II.C]{douady-suite-spectrale-adams} apart from notation (in particular, the latter writes $\pi(p,q)$ for what the former denotes $H(p,q)$).  This spectral sequence is natural with respect to all maps of filtered spaces: that is, continuous maps $X \to X'$ sending $\Fil_tX$ into $\Fil_tX'$.  

\subsubsection{Products of filtered spaces}

If $X'$ and $X''$ are filtered spaces, then the Cartesian product $X = X' \times X''$ inherits a filtration, namely
\begin{equation*}
  \Fil_tX = \mathrm{Im}\Big( \coprod_{u + v \leq t} \Fil_u X' \times \Fil_v X'' \to X' \times X'' = X\Big).
\end{equation*}
In this situation the inclusion $\Fil_u X' \times \Fil_v X'' \hookrightarrow \Fil_{u+v}X$ induces a chain map \[C_*(\Fil_u X') \otimes C_*(\Fil_vX'') \to C_*(\Fil_{u+v} X)\] defined by the chain-level cross product.  This inclusion sends both subspaces $\Fil_u X' \times \Fil_{v-1} X''$ and $\Fil_{u-1} X' \times \Fil_v X''$ into $\Fil_{u+v-1} X$, so the cross product factors over a chain map \begin{equation}\label{eq:cross-product-on-relative-chains}
C_*(\Fil_uX',\Fil_{u-1}X') \otimes C_*(\Fil_v X'',\Fil_{v-1}X'') \to C_*(\Fil_{u+v} X,\Fil_{u+v-1} X).    
\end{equation}  Passing to homology then gives a homomorphism
\begin{equation}\label{eq:5}
  \phi^1 \colon E^1_{p,q}(X') \otimes E^1_{p',q'}(X'') \to E^1_{p+p',q+q'}(X)
\end{equation}
which we will call the {\em exterior product}.  In brief, these make $X \mapsto E^1_{*,*}(X)$ into a lax monoidal functor from filtered spaces to bigraded vector spaces.

Similarly, filtering $H_*(X)$ by the images of the $H_*(\Fil_t X)$, and similarly for $H_*(X')$ and $H_*(X'')$, the cross product $H_*(X') \otimes H_*(X'') \to H_*(X)$ descends to a pairing
\begin{equation}\label{eq:6}
  \mathrm{Gr}_p H_{p+q}(X') \otimes  \mathrm{Gr}_{p'} H_{p'+q'}(X'') \to \mathrm{Gr}_{p+p'}H_{p+p'+q+q'}(X).
\end{equation}

\begin{proposition}\label{prop:douady}
  In the above setting there are pairings
  \begin{equation*}
    \phi^r \colon E^r_{p,q}(X') \otimes E^r_{p',q'}(X'') \to E^r_{p+p',q+q'}(X)
  \end{equation*}
  for all $r \geq 1$, agreeing with~(\ref{eq:5}) for $r = 1$, satisfying the Leibniz rule
  \begin{equation*}
    d^r\phi^r(x' \otimes x'') = \phi^r((d^r x') \otimes x'' + (-1)^{p+q} x' \otimes (d^r x''))
  \end{equation*}
  and such that the induced pairing on homology of $r$th pages is identified with $\phi^{r+1}$.  (Such pairings are of course unique if they exist since each determines the next---the main content is the Leibniz rule so that $\phi^r$ descends to a pairing on \emph{homology} of $r$th pages.)

  Moreover, the induced pairing $\phi^\infty$ of $E^\infty$-pages is identified with (\ref{eq:6}).  Finally, if $X'$ and $X''$ are CW complexes filtered by subcomplexes, then the homomorphisms
  \begin{equation*}
    \phi^r \colon E^r_{*,*}(X') \otimes E^r_{*,*}(X'') \to E^r_{*,*}(X),
  \end{equation*}
  obtained by taking direct sum over all $p,q,p',q'$, are isomorphisms for all $r \geq 1$.
\end{proposition}
\begin{proof}[Proof sketch]
  This is mostly contained in \cite[Th\'eor\`eme II.A]{douady-suite-spectrale-multiplicative}.  In the notation from there, we should set
  \begin{equation*}
    \pi(q,r) = \bigoplus_n H_n(\Fil_{-q}X,\Fil_{-r}X)
  \end{equation*}
  for $-\infty \leq q \leq r \leq \infty$ (where we write $\Fil_{-\infty}X = \emptyset$ and $\Fil_\infty X = X$), let
  \begin{equation*}
    \eta\colon \pi(q,r) \to \pi(q',r')
  \end{equation*}
  be defined by functoriality of homology for $q' \leq q$ and $r' \leq r$, and let
  \begin{equation*}
    \partial\colon \pi(q,r) \to \pi(r,s)
  \end{equation*}
  be the connecting homomorphism for the triple $(\Fil_q X, \Fil_r X , \Fil_s X)$ for $-\infty \leq q \leq r \leq s \leq \infty$.  This data is called a \emph{syst\`eme spectraux} in \cite{douady-suite-spectrale-adams} and a \emph{Cartan--Eilenberg system} in many other places, and satisfies the axioms (SP.1)--(SP.5) of \cite{cartan-eilenberg-homological}.
  
  Defining \emph{syst\`emes spectraux} $\pi'$ and $\pi''$ associated to the filtered spaces $X'$ and $X''$ in the same manner, the inclusions $\Fil_{-n}X' \times \Fil_{-q} X'' \hookrightarrow \Fil_{-n-q} X$ then induce homomorphisms
  \begin{equation*}
    \pi'(n,n+r) \otimes \pi''(q,q+r) \xrightarrow{\phi_r} \pi(n+q,n+q+r) = H_*(\Fil_{-n-q}X,\Fil_{-n-q-r} X)
  \end{equation*}
  for all $r \geq 1$, in the same way as~(\ref{eq:5}) which is the special case $r=1$. These homomorphisms satisfy the assumptions of Douady's theorem, which then gives the stated result, apart from the claim that the $\phi^r$ define isomorphisms after taking direct sum over all bidegrees.  The original reference \cite{douady-suite-spectrale-multiplicative} in fact omits the proof, but details can be found elsewhere, for instance \cite{goette-mathoverflow} or \cite{rognes-spectral}.  The verification of axioms (SPP.1) and (SPP.2) in Douady's theorem is as in \cite[Proposition 6.3.12]{rognes-spectral}.

  It remains to see that the exterior products $\phi^r$ become isomorphisms after taking direct sum over all bidegrees, when $X'$ and $X''$ are CW complexes filtered by subcomplexes.  In that case we can identify relative homology with reduced homology of the quotient space: $E^1_{p,q}(X') = \widetilde H_{p+q}(\Fil_p X'/\Fil_{p-1} X')$, and similarly for $X''$ and $X$.  Then $\phi^1$ is induced by maps $(\Fil_u X'/\Fil_{u-1}X')\wedge (\Fil_vX''/\Fil_{v-1}X'') \to \Fil_{u+v}X/\Fil_{u+v-1} X$ arising from the inclusions $\Fil_u X' \times \Fil_v X'' \hookrightarrow \Fil_{u+v} X$.  Taking wedge sum over all filtrations, we see that $\phi^1$ is induced from a single map of pointed spaces
  \begin{equation*}
    \Big(\bigvee_{s \in \ZZ} \Fil_s X'/\Fil_{s-1} X'\Big) \wedge   \Big(\bigvee_{s \in \ZZ} \Fil_{s} X''/\Fil_{s-1} X''\Big) \longrightarrow \bigvee_{s \in \ZZ} \Fil_s X/\Fil_{s-1} X,
  \end{equation*}
  and the claim about isomorphism follows for $r = 1$ from the K\"unneth theorem and the observation that this map is in fact a homeomorphism when $X = X' \times X''$ is given the CW topology.  But then inductively $\phi^r$ is also an isomorphism for higher $r$, again by the K\"unneth theorem.
\end{proof}

\subsection{Hopf algebra structure on the Quillen spectral sequence}

The results earlier in this section combine to yield the following extra structures on $\sseq{Q}$, the homological Quillen spectral sequence.

\begin{theorem}\label{thm:QSS}\mbox{}
\begin{enumerate}
    \item 
    The homological Quillen spectral sequence
    \[\sseq{Q}^1_{s,t} = H_{t}(\GL_s(\ZZ), \mathrm{St}_s \otimes \QQ) \Rightarrow H_{\ast}(BK(\ZZ))\]
    is a spectral sequence of Hopf algebras.  That is, for each $r\ge 0$, there are maps
    \[m^r \col E^{r}_{s,t} \otimes E^r_{s',t'} \to E^r_{s+s',t+t'},\qquad 
    \Delta^r \col E^r_{s,t} \to \bigoplus_{\substack{s'+s'' = s \\ t'+t'' = t}} E^r_{s',t'} \otimes E^r_{s'',t''},\]
    induced from 
    \[m \col BK(\ZZ)\times BK(\ZZ) \to BK(\ZZ)  \quad\text{and}\quad \Phi \col BK(\ZZ) \to BK(\ZZ) \times BK(\ZZ) \]
    respectively,  
    making $E^r$ a bigraded Hopf algebra. Moreover, the differential \[d^r\col E^r_{s,t} \to E^r_{s-r,t+r-1} \]
    is compatible with product and coproduct in the sense that 
  \[d^r(x_1\cdot x_2) = d^r(x_1)\cdot x_2 + (-1)^{p_1+q_1} x \cdot d^r(x_2)\]
    for $x_i \in E^{r}_{p_i,q_i}$, where $x\cdot y$ denotes the product $m^r$; and
    \[\Delta^r \circ d^r = (d^r\otimes 1 \pm 1\otimes d^r) \circ \Delta^r. \]
    More precisely, the sign ``$\pm$'' can be expressed in terms of the symmetry 
    \begin{align*}
        T \colon E^r_{p,q} \otimes E^r_{p',q'} & \xrightarrow{\cong} E^r_{p',q'} \otimes E^r_{p,q}\\
        x \otimes y \quad & \mapsto (-1)^{(p+q)(p'+q')} y \otimes x;
    \end{align*}
    the Leibniz rule for the coproduct should then read $\Delta^r \circ d^r = (d^r \otimes 1 + T \circ (d^r \otimes 1) \circ T) \circ \Delta^r$.
  \item There are isomorphisms $\QQ \to E^r_{0,0} \to \QQ$ acting as unit and counit respectively, making each page $E^r$ into a bigraded Hopf algebra.
  \item Giving the homology of the $r$th page the Hopf algebra structure induced by $m^r$ and $\Delta^r$, the isomorphism $E^{r+1} = H(E^r,d^r)$ is a Hopf algebra isomorphism.
  \item The filtration of $H_*(BK(\ZZ);\QQ)$ induced by convergence of the spectral sequence is multiplicative and comultiplicative, where the product on $H_*(BK(\ZZ);\QQ)$ is induced by $m \colon BK(\ZZ)\times BK(\ZZ) \to BK(\ZZ)$ (in turn constructed from the $\oplus$ operation on the Waldhausen construction; this also agrees with the product induced by the loop space structure) and the coproduct on $H_*(BK(\ZZ);\QQ)$ is induced by the diagonal map (dual to cup product).  With respect to the induced Hopf algebra structure on $\mathrm{Gr} H_*(BK(\ZZ);\QQ)$, the isomorphisms  $E^\infty_{p,q} = \mathrm{Gr}_p H_{p+q}(BK(\ZZ);\QQ)$ form an isomorphism of bigraded Hopf algebras.

\item The product on $E^r$ is graded-commutative for all $r \geq 1$: 
\[x_1\cdot x_2 = (-1)^{(p_1+q_1)(p_2 + q_2)} x_2\cdot x_1,\]
for $x_i$ in $E^r_{p_i,q_i}$.
\item 
    The coproduct $\Delta^r$ is graded co-commutative for $r = \infty$ but not necessarily for finite $r$.  (We are of course asserting that it is co-associative for all $r$, which is part of the axioms for Hopf algebras.) 
\end{enumerate}
\end{theorem}

\begin{proof}
  Setting $X' = X'' = BK(\ZZ)$ in Proposition~\ref{prop:douady}, we obtain isomorphisms
\begin{equation}\label{eq:7}
  E^r_{*,*}(BK(\ZZ)) \otimes E^r_{*,*}(BK(\ZZ)) \xrightarrow{\phi^r} E^r_{*,*}(BK(\ZZ) \times BK(\ZZ))
\end{equation}
satisfying the Leibniz rule on each page, as well as compatibility between pages and with the abutments.

Now we combine with the space-level products and coproducts and functoriality of the spectral sequence with respect to filtered maps, using Corollary~\ref{cor:homotopic-to-diagonal}, Propositions~\ref{prop:(3)},~\ref{prop:(4)},~\ref{prop:(5)}, and Remark~\ref{rem:product-compatible-with-filtration}.  For instance, the space-level coproduct $\Phi \colon BK(\ZZ) \to BK(\ZZ) \times BK(\ZZ)$ is a map of filtered spaces and hence induces a map of spectral sequences
\begin{equation*}
  E^r_{*,*}(BK(\ZZ)) 
  \xrightarrow{\Phi_*}  
  E^r_{*,*}(BK(\ZZ) \times BK(\ZZ)) 
\end{equation*}
which we combine with the inverse of~(\ref{eq:7}) to obtain a coproduct on the $r$th pages of the Quillen spectral sequence.  Similarly for the  product, while the unit and counit come from space-level maps $\{\text{point}\} \to BK(\ZZ) \to \{\text{point}\}$, where the one-point space is filtered as $\Fil_{-1} = \emptyset \subset \{\text{point}\} = \Fil_0$.

To see that the coproduct on the spectral sequence is coassociative, we write $I = [0,1]$ filtered as $\emptyset = \Fil_{-1}I \subset \Fil_0 I = I$ and first observe that the two injections $BK(\ZZ) \hookrightarrow I \times BK(\ZZ)$, given by $x \mapsto (0,x)$ and $x \mapsto (1,x)$, induce equal maps of spectral sequences, since both are one-sided inverses to the isomorphism of spectral sequences induced by the projection $I \times BK(\ZZ) \to BK(\ZZ)$.  Coassociativity then follows from the space-level homotopy
\begin{equation*}
  I \times BK(\ZZ) \to BK(\ZZ) \times BK(\ZZ) \times BK(\ZZ),
\end{equation*}
observing that this is in fact a filtered map. 

To see that the coproduct is co-commutative on the $E^\infty$-page, we use that the space-level map $\Phi$ is homotopic to the diagonal map.  Therefore the induced coproduct $\Phi_* \colon H_*(BK(\ZZ)) \to H_*(BK(\ZZ)) \otimes H_*(BK(\ZZ))$ is co-commutative, but then this also holds for the induced map of associated gradeds, which is identified with $\phi^\infty \colon E^\infty_{*,*} \to E^\infty_{*,*} \otimes E^\infty_{*,*}$.  

All other properties follow from the corresponding space-level properties in a similar way.
\end{proof}

\section{Proof of Theorems~\ref{thm:HopfAg} and \ref{thm:canonical-inj} and Corollary~\ref{cor:expSL}} \label{sec:proof12}

\subsection{A product on tropical moduli spaces}

Towards a proof of Theorem~\ref{thm:HopfAg}, we shall first study a graded-commutative product on $W_0H^*_c(\cA)$, which can be interpreted both in terms of products of abelian varieties and of tropical abelian varieties.  
As in the proof of Proposition~\ref{prop:E2-0}, we have short exact sequences
\begin{equation}\label{eq:a-ses}0 \to H^{\mathrm{BM}}_{k}(A_{g}^{\mathrm{trop}}) \xra{\iota} H^{\mathrm{BM}}_{k}(A_g^{\mathrm{trop}},A_{g-1}^{\mathrm{trop}}) \xra{\partial} H_{k-1}^{\mathrm{BM}}(A_{g-1}^{\mathrm{trop}})\to 0\end{equation}
for all $k$ and $g$.  These follow from the fact that $H_*^{\mathrm{BM}}(A_{g-1}^{\mathrm{trop}})\to H_*^{\mathrm{BM}}(A_{g}^{\mathrm{trop}})$ is zero, established in~\eqref{eqn: InflationZeroMap}.

Now let $(P^{(g)}[-1],d)$ denote the degree-shifted {\em perfect cone complex}, whose definition and properties we shall now recall from \cite{bbcmmw-top}.  The complex $P^{(g)}[-1]$ is a rational chain complex with differential of degree $-1$, with generators $[\sigma,\omega]$ in degree $\dim(\sigma)$, where $\sigma$ is a perfect cone and $\omega$ is an orientation of the linear span of $\sigma$.  Relations are given by $[\sigma,\omega] = \pm[\sigma',\omega']$ if $\sigma$ and $\sigma'$ are in the same $\GL_g(\ZZ)$-orbit,
with a plus sign if the induced action of $\GL_g(\ZZ)$ on the orientation $\omega$  equals $\omega'$, and a minus sign if not.    
The boundary of $[\sigma,\omega]$ is a sum of codimension $1$ faces of $\sigma$ with induced orientation.  Then from \cite{bbcmmw-top} we have
$H_k^{\mathrm{BM}}(A_g^{\mathrm{trop}}) \cong H_k(P^{(g)}[-1])$.
Moreover, there are natural inclusions $P^{(g-1)}\to P^{(g)}$ such that
\[\squarediagramlabel{H_k(P^{(g-1)}[-1])}{H_k(P^{(g)}[-1])}{H_k^{\mathrm{BM}}(A_{g-1}^{\mathrm{trop}})}{H_k^{\mathrm{BM}}(A_g^{\mathrm{trop}})}{}{\cong}{\cong}{}\]
commutes. Therefore from~\eqref{eq:a-ses} we obtain short exact sequences
\[0 \to H_k(P^{(g)}[-1]) \xra{\iota} H_k((P^{(g)}/P^{(g-1)})[-1]) \xra{\partial} H_{k-1}(P^{(g-1)}[-1])\to 0\]
for all $k$ and $g$.

Now we construct a product map and prove a Leibniz rule.  Consider the natural continuous maps
\begin{equation}\label{eq:mult-trop}A_g^{\mathrm{trop}}\times A_h^{\mathrm{trop}}\to A_{g+h}^{\mathrm{trop}}\end{equation}
induced by block sum of positive semidefinite forms.  These maps extend to
\begin{equation*}\label{eq:mult-trop-compactified}(A_g^{\mathrm{trop}}\cup\{\infty\},\infty)\wedge (A_h^{\mathrm{trop}}\cup\{\infty\},\infty) \to (A_{g+h}^{\mathrm{trop}}\cup\{\infty\},\infty),\end{equation*}
the one-point compactifications with added point $\infty$.  Hence we obtain
\begin{equation}\label{eq:mult-on-bm}
H_k^{\mathrm{BM}}(A_g^{\mathrm{trop}}) \otimes H_\ell^{\mathrm{BM}}(A_h^{\mathrm{trop}})  \to H_{k+\ell}^{\mathrm{BM}}(A_{g+h}^{\mathrm{trop}}).
\end{equation}
Similarly, we have
\begin{equation}\label{eq:mult-on-rel-bm}
H_k^{\mathrm{BM}}(A_g^{\mathrm{trop}}, A_{g-1}^{\mathrm{trop}}) \otimes H_\ell^{\mathrm{BM}}(A_h^{\mathrm{trop}}, A_{h-1}^{\mathrm{trop}})  \to H_{k+\ell}^{\mathrm{BM}}(A_{g+h}^{\mathrm{trop}}, A_{g+h-1}^{\mathrm{trop}})
\end{equation}
Both~\eqref{eq:mult-on-bm} and~\eqref{eq:mult-on-rel-bm} are induced from
\begin{equation}\label{eq:mult-on-Pg} m\col P^{(g)}[-1]\otimes P^{(g')}[-1]\to P^{(g+g')}[-1], \quad [\sigma,\omega]\otimes [\sigma',\omega']\mapsto [\sigma\times \sigma',\omega\ast \omega']\end{equation}
where, for cones $\sigma \subset \Sym^2((\RR^g)^\vee)$ and $\sigma'\subset\Sym^2((\RR^{g'})^\vee)$, we take $\sigma\times \sigma'\in \Sym^2((\RR^{g'})^\vee) \times \Sym^2((\RR^{g})^\vee) \subset \Sym^2((\RR^{g+g'})^\vee)$ as the cone of block sums of symmetric bilinear forms in $\sigma$ and $\sigma'$.  The product $m(\sigma,\sigma')$ shall be denoted by $\sigma.\sigma'$ henceforth, and similarly for the products~\eqref{eq:mult-on-bm} and~\eqref{eq:mult-on-rel-bm}.  Note that $m$ is graded-commutative.

Consider the bigraded vector  space $\mathcal{A}^{BM}$ defined as
\begin{equation*}
    \mathcal{A}^{BM}_{s,t} = H^{BM}_{s+t}(A^\mathrm{trop}_s;\QQ).
\end{equation*}
It is the (bi)graded dual of $W_0H^*_c(\cA)$ and vanishes for $t < 0$.  Then we have the following conclusion.

\propnow{The maps~\eqref{eq:mult-on-bm} equip $\cA^{\mathrm{BM}}$ with the structure of a $\QQ$-algebra, which is graded-commutative with respect to its total grading.}

\proofnow{The graded-commutativity of the product on $\cA^{\mathrm{BM}}$ follows from graded-commutativity of the product~\eqref{eq:mult-on-Pg} on perfect cone complexes.}

\begin{remark}
The algebraic moduli space $\cA$ has a natural space-level commutative product, given on connected components by
\begin{equation} \label{eq:algebraic-product-ag}
\cA_g \times \cA_h \to \cA_{g+h}; \ \ \ (A_1, A_2) \mapsto A_1 \times A_2.
\end{equation}
This product morphism $\cA \times \cA \to \cA$ is proper, and the induced pullback map $$W_0H^*_c(\cA) \to  W_0H^*_c(\cA)\otimes W_0H^*_c(\cA)$$ agrees with~\eqref{eq:mult-on-bm} via the standard comparison isomorphisms and dualities. 

To see this agreement, one may first note that the tropicalization of the algebraic product \eqref{eq:algebraic-product-ag} is the tropical product defined in~\eqref{eq:mult-trop},
induced by block sum of positive semidefinite quadratic forms. Then, use the fact that the Berkovich analytic skeleton of a product of abelian varieties over a valued field is the product of the skeletons, as principally polarized tropical abelian varieties. This is because the skeleton can be read off via non-archimedean analytic uniformization from the Raynaud cross diagram, as in \cite[Section~3.2]{foster-rabinoff-shokrieh-soto}, and the Raynaud cross of the product is the product of the Raynaud crosses of the factors. 

We also note that the algebraic product \eqref{eq:algebraic-product-ag} extends to additive families of toroidal compactifications, such as the perfect cone compactifications. The fact that 
\eqref{eq:mult-trop}
is the tropicalization of \eqref{eq:algebraic-product-ag} is well-known in this context. See \cite[Proposition~9]{grushevsky-hulek-tommasi-stable-betti}.
\end{remark}

The following version of the Leibniz rule will be used in the next subsection.

\propnow{\label{prop:leibniz}
 Let $\alpha \in H_k^{\mathrm{BM}}(A_g^{\mathrm{trop}}, A_{g-1}^{\mathrm{trop}})$ and $\beta \in H_\ell^{\mathrm{BM}}(A_h^{\mathrm{trop}}, A_{h-1}^{\mathrm{trop}})$.  Then, with $\iota$ and $\partial$ as in~\eqref{eq:a-ses}, we have
\[\iota\partial(\alpha.\beta) = (\iota\partial\alpha).\beta + (-1)^{\mathrm{deg}(\alpha)}\alpha.(\iota\partial\beta).\]
Furthermore, the map $\iota$ is a homomorphism for the multiplication maps \eqref{eq:mult-on-bm} and \eqref{eq:mult-on-rel-bm}.
}

\proofnow{The key point is that the product~\eqref{eq:mult-on-Pg} satisfies the graded Leibniz rule. More precisely,~\eqref{eq:mult-on-Pg} is a morphism of chain complexes: \[d([\sigma,\omega].[\sigma',\omega']) = d[\sigma,\omega].[\sigma',\omega'] + (-1)^{\mathrm{dim}(\sigma)}[\sigma,\omega].d[\sigma',\omega']\]
by properties of faces of products of polyhedral cones.  Then the Proposition follows since the product~\eqref{eq:mult-on-rel-bm} is induced from this product map on perfect cone complexes.
}

\subsection{Hopf algebra structure on \texorpdfstring{$W_0 H^*_c(\cA)$}{W0H(A)}}
Let $\sseq{Q}^1$ denote the $E^1$-page of the homological Quillen spectral sequence. It is a graded-commutative  bigraded Hopf algebra by Theorem \ref{thm:QSS}. The graded-commutativity is with respect to the total degree.  

The vector space $\sseq{Q}^1_{1,0}$ is 1-dimensional and plays an important role in the Hopf algebra structure on $\sseq{Q}^1$, so we pause briefly to establish a canonical choice of generator.
The fundamental group $\pi_1(BK(\ZZ)) = K_0(\ZZ) \cong \ZZ$ is infinite cyclic, and is in fact generated by the loop corresponding to any 1-simplex $(0 \subset P_{0,1}) \in N_0 S_1 (\Proj{\ZZ})$ for $P_{0,1}$ of rank 1 (that is, $P_{0,1}$ is isomorphic as an abelian group to $\ZZ$).  Any two choices of $P_{0,1}$ will be isomorphic, and lead to homotopic loops in the based space $BK(\ZZ)$. Let us write
\begin{equation*}
    e \col S^1 \to BK(\ZZ)
\end{equation*}
for any based loop in this homotopy class.  The loop $e$ can also be interpreted as the generator of the fundamental group of $\mathrm{gr}_1 BK(\ZZ) = (F_1 BK(\ZZ))/(F_0 BK(\ZZ))$. Hence we obtain a canonical generator
\begin{equation*}
    e_*([S^1]) \in H_1(F_1 BK(\ZZ),F_0 BK(\ZZ);\QQ) = \sseq{Q}^1_{1,0} \cong \QQ,
\end{equation*}
where $[S^1]$ denotes the fundamental class of the circle $S^1 = \Delta^1/(\partial \Delta^1)$.  This class is a permanent cycle in the spectral sequence, and we will henceforth use the notation $e \in \sseq{Q}^1_{1,0}$ for this homology class, as well as for its image $e \in \sseq{Q}^r_{1,0}$ in all subsequent pages  $r \leq \infty$.

\begin{proposition}   \label{prop:quillen-to-A}
With $e\in \sseq{Q}^1_{1,0} $ as above,   there   is a canonical isomorphism of graded-commutative   algebras
\[    \mathcal{A}^{BM} \otimes_{\QQ} \QQ[\eps]/\eps^2 \overset{\sim}{\to} \sseq{Q}^1 \ , \]
sending $\epsilon$ to $e$, where $\QQ[\eps]/\eps^2 \cong  \bigwedge \QQ \eps$ is the graded exterior algebra generated by  $\eps$. 
\end{proposition}

\begin{proof} 
 
Let $e \in \sseq{Q}^1_{1,0}\cong \QQ$ be 
the generator chosen above. (Any non-zero multiple thereof would work equally well for the following argument.)  Since the coproduct  $\Delta$ on $\sseq{Q}^1$ respects the bigrading, we must have
\begin{equation} \label{egeneratorisprimitive} \Delta\,  e = 1 \otimes e + e \otimes 1 \ .\end{equation}
By graded-commutativity, or using the fact that $ \sseq{Q}^1_{2,0} $ vanishes, we have $e^2=0$. 

Recall from~\eqref{eq:a-ses} the short exact sequences
\begin{equation} \label{longexactQE1st} 0 \to H^{\mathrm{BM}}_{s+t}(A_{s}^{\mathrm{trop}}) \xra{\iota} H^{\mathrm{BM}}_{s+t}(A_s^{\mathrm{trop}},A_{s-1}^{\mathrm{trop}}) \xra{\partial} H_{s+t-1}^{\mathrm{BM}}(A_{s-1}^{\mathrm{trop}})\to 0,\end{equation}
where the middle term is isomorphic to $\sseq{Q}^1_{s,t}$. 
 The multiplication on $\sseq{Q}^1$ is induced by block sum of matrices, as in Proposition~\ref{prop:E1-formula-with-Steinberg}, and hence coincides with~\eqref{eq:mult-on-rel-bm}.    
Therefore, for any element   $\alpha \in  H_{s+t-1}^{BM}(A_{s-1}^{\trop};\QQ)$, Proposition~\ref{prop:leibniz} implies the following formula in $\sseq{Q}^1$:
\begin{equation}\label{eq:boundary-of-inflation} \iota \partial (e  .\iota  \alpha) =  \iota  \partial(e) .   \iota  \alpha\end{equation}
since the second term in the Leibniz rule is $e.(\iota\partial\iota \alpha)$ which 
vanishes by \eqref{longexactQE1st}. 
Since $H_1^{\mathrm{BM}}(A_1^{\trop})=0$, the sequence \eqref{longexactQE1st} reduces for $(s,t)=(1,0)$ to an isomorphism 
\[ \sseq{Q}^1_{1,0}  \overset{\partial}{\cong}  H_{0}^{BM}(A_{0}^{\trop};\QQ)  \ , \]
and hence by injectivity of $\iota$, we can write $\iota \partial (e) = \lambda 1  \in \sseq{Q}^1_{0,0} $ for $\lambda \in \QQ^\times$ and $1 \in \sseq{Q}^1_{0,0}$ 
the unit element in the algebra structure.  In fact, by tracing through definitions, we can see $\lambda=\pm 1$.  The sign is not important: as in Remark~\ref{rem:two-comparisons}, it depends on whether the chosen diffeomorphism $\ell\col (0,1)\to (0,\infty)$ in~\eqref{eq:diffeo}, appearing in the comparison of tropical and Quillen spectral sequences in subsection~\ref{sec: Loc-sym-Quillen}, is increasing or decreasing.  Then $\partial(e.\iota\alpha) = \lambda\alpha$, so that $\alpha \mapsto \lambda^{-1} e.\iota\alpha$ is a right-inverse to $\partial$, splitting the short exact sequence \eqref{longexactQE1st}.

Let $\cA^{\mathrm{BM}}[-1]$ denote the shift by $(-1,0)$ in bidegree, so $\cA^{\mathrm{BM}}[-1]_{s,t} = \cA^{\mathrm{BM}}_{s-1,t}.$ We now define the following diagram, which, we shall then argue, is commutative.
\[
\begin{array}{cccccc}
    0  \longrightarrow & \mathcal{A}^{BM} & \longrightarrow & \mathcal{A}^{BM} \oplus \mathcal{A}^{BM}[-1] & \longrightarrow & \mathcal{A}^{BM}[-1]  \longrightarrow 0  \\
      & \downarrow  &  & \downarrow \,\,  &   & \downarrow  \qquad  \\
      0  \longrightarrow & \mathcal{A}^{BM} & \longrightarrow & \sseq{Q}^1_{*,*}  &  \overset{\partial}{\longrightarrow} & \mathcal{A}^{BM}[-1]  \longrightarrow 0  
\end{array}
\]
Here the top row is the canonical split short exact sequence, and the bottom row is~\eqref{longexactQE1st} in bidegree $(s,t).$
The vertical map in the middle is \[(\alpha, \beta) \mapsto \iota \alpha + \lambda^{-1} e. \iota \beta\] 
where $.$ denotes multiplication in $\sseq{Q}^1$,
and the left and right vertical maps are the identity.  By~\eqref{eq:boundary-of-inflation}, $\partial (\iota \alpha + \lambda^{-1} e. \iota \beta) = \beta$ and therefore the diagram commutes, and all vertical maps must be isomorphisms. In particular, we deduce a canonical isomorphism  $\mathcal{A}^{BM} \otimes_{\QQ} \QQ[\epsilon]/\epsilon^2 \cong \sseq{Q}^1$  on the level of bigraded vector spaces. 

Finally, the natural morphism $\cA^{\mathrm{BM}}\to \sseq{Q}^1$ arising from the above isomorphism is a morphism of graded-commutative algebras, i.e., respects the products.   Indeed, the products on domain and codomain are~\eqref{eq:mult-on-bm} and~\eqref{eq:mult-on-rel-bm}, respectively, and are both induced from the same product map~\eqref{eq:mult-on-Pg} on perfect complexes. It follows from this, together with the definition of the middle vertical map in the above diagram, that 
$\mathcal{A}^{BM} \otimes_{\QQ} \QQ[\epsilon]/\epsilon^2 \cong \sseq{Q}^1$ is an isomorphism of algebras.
\end{proof}

\begin{proof}[Proof of Theorem~\ref{thm:HopfAg}] 
Any generator $e \in \sseq{Q}^1_{1,0}\cong \QQ$ is primitive 
since the coproduct $\Delta$ on $\sseq{Q}^1$ respects the bigrading \eqref{egeneratorisprimitive}.  Therefore the ideal generated by $e$ is also a co-ideal, 
and the quotient Hopf algebra $\sseq{Q}^1/(e)$ is isomorphic to $\cA^{\mathrm{BM}}$ by Proposition~\ref{prop:quillen-to-A}.  This yields a Hopf algebra structure on $\cA^{\mathrm{BM}}$, and dualizing yields a Hopf algebra structure on $W_0H^*_c(\cA)$.
\end{proof}

Next, we show that there is a natural injection from the bigraded vector space $\Omega_c^*[-1]$ into the subspace of primitives $\Prim(W_0H^*_c(\cA))$.  We recall the canonical differential forms for $\GL_g$ and their basic properties from \S\ref{sect: backgroundcanforms}.

\begin{proof}[Proof of Theorem~\ref{thm:canonical-inj}]
By \cite{brown-bordifications} there is an injective map of graded $\RR$-vector spaces 
\begin{equation}  \label{eqn: Omegamapstocompactsupport}
 \Omega^{*}_{c}(g)[-1] \otimes \RR  \to  H^{*}_c( P_g /\GL_g(\ZZ);\RR) \cong \sseq{Q}_1^{g,\ast} \otimes \RR 
\end{equation}
for all $g>1$ odd (see Section~\ref{sect: backgroundcanforms} for further detail).  

It remains to show that the image of $\Omega_c^{*}(g)[-1]$ in $\sseq{Q}_1$ consists of primitive elements.
Let  $\omega \in \Omega^*_c(g)$ be homogeneous of compact type where $g>1$ is odd. We may assume that it is of the form 
\[ \omega^I = \omega^{4i_1+1} \wedge \omega^{4i_2+1} \wedge \ldots \wedge \omega^{4i_k+1} \]
where $I= \{i_1,\ldots, i_k\}$ are distinct and $4i_k+1=2g-1$. From  \eqref{eqn: omegablockdirectsum}  we deduce that 
\[  \omega^I_{X \oplus Y }  = \sum_{I= J \cup K}  \omega^J_X \otimes \omega^K_Y\]
where the sum is over all decompositions of $I$ into a disjoint union of (possibly empty) sets $J, K$, and $X,Y$ are positive definite symmetric matrices. Suppose that $X$ and $Y$ both have positive rank. Since one of $\omega^J$ and $\omega^K$ must necessarily have $\omega^{2g-1}$ as a factor, it follows from  \eqref{eqn: omega2g1vanishing}  that $\omega^J_X \otimes \omega^K_Y=0$ for all $J,K$ and we deduce that $\omega^I_{X\oplus Y}$ is identically zero. If  $\Delta' = \sum_{m,n \geq 1} \Delta_{m,n}$  denotes the  reduced coproduct, defined by  $\Delta  =\mathrm{id} \otimes 1 + 1 \otimes \mathrm{id} + \Delta'$, then we have shown that $\Delta' \omega=0$ for all $\omega \in \Omega_c(g)$ since the restriction of $\omega$ to the image of $\mathcal{A}_m^{\trop}\times\mathcal{A}_n^{\trop}$ for $m+n=g$ and $m,n>0$ is  zero.    It follows that  the $(m,n)$-component of the coproduct $\Delta \omega$  is zero, and hence the image of  $\Omega_c(g)[-1]$ in $\sseq{Q}_1$ is primitive. The theorem follows by quotienting  by $e.$
\end{proof}

\begin{proof}[Proof of Corollary~\ref{cor:expSL}]

Recall that \[H_*(\GL_n(\ZZ);\St_n\otimes \QQ)\cong H^{\binom{n}{2}-*}(\GL_n(\ZZ); \QQ_\mathrm{or})\] where $\QQ_\mathrm{or}$ denotes the $\GL_n(\ZZ)$-module given by orientations on $\SL_n(\RR)/\SO_n(\RR)$, the symmetric space of positive definite symmetric bilinear forms on $\RR^n$ of determinant $1$.
To relate to $\SL_n(\ZZ)$, Shapiro's Lemma
\cite[6.3.2, p.~171]{weibel-introduction} gives 
\[H^*(\SL_n(\ZZ);\QQ) \cong H^*(\GL_n(\ZZ);\mathrm{Ind}\substack{\GL_n(\ZZ)\\\SL_n(\ZZ)}\,\QQ)\]
and $\mathrm{Ind} \substack{\GL_n(\ZZ)\\\SL_n(\ZZ)}\,\QQ \cong \QQ \oplus \widetilde{\QQ}$
where $\widetilde{\QQ}$ denotes the determinantal representation of $\GL_n(\ZZ)$.  Now by \cite[Lemma 7.2]{elbaz-vincent-gangl-soule-perfect}, 
the orientation module $\QQ_{\mathrm{or}}$ is isomorphic to $\QQ$ if $n$ is odd and to $\widetilde{\QQ}$ if $n$ is even.  
Thus, if $n$ is odd then $\GL_n(\ZZ) \cong \SL_n(\ZZ) \times \ZZ/2\ZZ$, and 
$$H^*(\SL_n(\ZZ);\QQ) \cong H^*(\GL_n(\ZZ);\QQ) \cong H^*(\GL_n(\ZZ);\QQ_{\mathrm{or}}).$$  If $n$ is even, then
\[H^*(\SL_n(\ZZ);\QQ) \cong H^*(\GL_n(\ZZ);\QQ) \oplus H^*(\GL_n(\ZZ);\widetilde{\QQ})\cong H^*(\GL_n(\ZZ);\QQ) \oplus H^*(\GL_n(\ZZ);\QQ_{\mathrm{or}}).\]
In both cases, $H^*(\GL_n(\ZZ);\QQ_{\mathrm{or}})$ is a summand of $H^*(\SL_n(\ZZ);\QQ)$.
Now we have
\[
(W_0 H^*_c(\cA))^\vee \otimes_\QQ \QQ[x] / x^2 \cong \bigoplus_g H_{*} (\GL_g(\ZZ), \St_g \otimes \QQ),
\]
where $x$ has genus $1$ and degree $1$, from dualizing Proposition~\ref{prop:quillen-to-A}.
Now Corollary~\ref{cor:expSL} follows from Corollary~\ref{cor:expAg}.
\end{proof}
See subsection~\ref{rem:kSL} for a refinement of Corollary~\ref{cor:expSL}.

\subsection{Consequences of Theorem~\ref{thm:canonical-inj} and a  conjecture}
 Theorem~\ref{thm:canonical-inj}  can be rephrased as follows: 

\begin{thm}  \label{thm: Tensoralgebramaps} There is a canonical morphism of bigraded Hopf algebras
\begin{equation} \label{TOmegamap} T(\Omega^*_c[-1]) \otimes \RR \longrightarrow W_0 H^*_c(\cA;\RR),   \end{equation}
    whose domain is given the coproduct induced by declaring $\Omega^*_c[-1] \subset T(\Omega^*_c[-1])$ to be primitive.  The restriction of this Hopf algebra map to $\Omega^*_c[-1]$ is injective.
    \end{thm}
The map \eqref{TOmegamap} restricts to a map of graded Lie algebras
    \begin{equation}   \label{LieTOmegamap} \mathbb{L} ( \Omega^*_c[-1]) \longrightarrow \Prim ( W_0 H^*_c(\cA;\RR) ) \ , 
    \end{equation}
    where $\mathbb{L}$ denotes the free (graded-commutative) Lie algebra on $\Omega^*_c[-1]$.  Conjecture~\ref{conj: Tinjects} predicts that the classes from $\Omega^*_c[-1]$ satisfy no relations in $W_0 H^*_c(\cA;\RR)$. We restate this as follows:
\begin{conj} \label{conj: Liealginjects} The  map \eqref{TOmegamap}, or equivalently, the map~\eqref{LieTOmegamap}, is injective. 
\end{conj} 
The later sections of this  paper provide various kinds of evidence for this conjecture. 
\begin{remark}
    A similar conjecture in \cite{brown-invariant} states that  the free Lie algebra on a space isomorphic to $\Omega^*_c[-1]$ injects into the cohomology of the commutative graph complex (or equivalently, in the cohomology of  the moduli spaces of tropical curves). There are some key differences between these two conjectures: the map in Conjecture \ref{conj: Liealginjects} is given explicitly and respects the bigrading; the one in \cite{brown-invariant} is not explicit and does not respect the grading by genus in general. These two conjectures are not related by the tropical Torelli map: one  can show that tropical Torelli map is zero on the image of almost all canonical forms which lie above the diagonal  where degree equals twice the genus \cite[\S14.6]{brown-bordifications}. It is, however,  expected to be an isomorphism on the Lie subalgebra generated by the $\omega^{4k+1}$ which lies on the diagonal.  
\end{remark}

\section{Graph spectral sequences} \label{sec:Graphss}

Our next aim will be to construct a graph spectral sequence with a Hopf algebra structure analogous to that on the Quillen spectral sequence produced in Theorem~\ref{thm:QSS}, with coproduct induced by a variant of the Connes--Kreimer coproduct on the Hopf algebra of graphs constructed in \cite{connes-kreimer-hopf}.  The spectral sequence that we produce most naturally will only be a spectral sequence of bialgebras, whose pages do not admit an antipode.  For our purposes this is a very minor defect.  In subsection~\ref{sec:inverting-the-class-of-a-point}, we will explain how to obtain a spectral sequence of connected Hopf algebras, as the localization obtained by inverting a certain element of bidegree $(0,0)$.

The graph spectral sequence arises from the construction of a ``$K$-theory of graphs,'' a bisimplicial space $\BKGr$ with the same structure and formal properties as $BK(\ZZ) = |N_\bu S_\bu (\Proj{\ZZ})|$.  The bisimplicial space we construct is rationally equivalent to a space $\MGr$ that may be interpreted as a space of metric graphs, which has a filtration by first Betti number.  
The $E^1$-page of the associated spectral sequence is the homology of a suitable graph complex and comes with a map of spectral sequences of bialgebras to the Quillen spectral sequence.  In Section~\ref{sec:freeness}, we use this map to prove that the image of $\{\omega^5, \omega^9, \ldots, \omega^{45} \} \subset \Omega_c^*[-1] \otimes \RR$ inside $\sseq{Q}_1^{*,*}\otimes \RR$ 
generates a free Lie subalgebra of the primitives of $W_0H^*_c(\cA;\RR)$.

\begin{remark} J.~Steinebrunner in recent talks has outlined a different construction of a spectral sequence whose $E^2$-page is given by the homology of the spaces $\Delta_g$ considered in \cite{cgp-graph-homology} and hence also related to the homology of Kontsevich's graph complexes.  His construction uses a filtration of the terminal modular $\infty$-operad.
\end{remark}

\begin{remark}
    A related structure of a Hopf algebra on graphs appears in the work \cite{KW17,GKT20a,GKT20b}.  In particular, the graph complex as given by the bar construction \cite{KW23} produces a co-operad structure with a multiplication which is filtered by the first Betti number.  We thank Ralph Kaufmann for helpful correspondence regarding the relationship to his work with coauthors.
\end{remark}

\begin{remark} \label{rem:cubical-spaces}  These $K$-theory spaces of graphs, and their corresponding chain complexes, naturally have a cubical structure reminiscent of earlier constructions of moduli spaces of metric graphs.  We refer to Culler--Vogtmann's construction of Outer Space $\mathrm{CV}_n$ \cite{culler-vogtmann} and its marked-point variants $\mathrm{CV}_{n,s}$. These are used to study homology of the groups $\mathrm{Out}(F_n)$, $\mathrm{Aut}(F_n)$, and a more general family of groups denoted $\Gamma_{n,s}$ \cite{hatcher-homological, hatcher-vogtmann-homology-stability, conant-hatcher-kassabov-vogtmann-assembling}.  
The space $\mathrm{CV}_n$ is simplicial, but it has a deformation retraction to a subcomplex $\mathbb{S}\mathrm{CV}_n$ of its barycentric subdivision called the {\em spine}, which is a cubical space parametrizing marked metric graphs with a chosen subforest and has an associated graph complex \cite[\S3]{hatcher-vogtmann-rational}.  The space $\MGr$ that we shall consider, and its associated graph complex, have a similar structure. 
\end{remark}

\begin{remark}
  There is a Waldhausen category $\cG$, whose objects are all connected graphs $G$ together with chosen basepoint $\ast \in V(G)$, whose  morphisms  are all maps of graphs preserving base point, cofibrations the injective maps, and weak equivalences the ones whose corresponding map of topological spaces is a homotopy equivalence.  There is also a symmetric monoidal structure given by wedge sum.  Filtering by $b_1(G) \in \NN$ induces a filtration of $BK(\cG)$ with the exact same formal properties as the filtration on $BK(\ZZ)$, leading to a spectral sequence of Hopf algebras
  \begin{equation*}
    E^1_{s,t} \Rightarrow H_{s+t}(BK(\cG)).
  \end{equation*}
  Moreover, the functor $H_1\col \cG\to \Proj{\ZZ}$ taking $G$ to $H_1(G;\ZZ)$ preserves all the structure and induces a map of spectral sequences of Hopf algebras. 

  At present we do not know how to make good use of this specific spectral sequence.  Instead we proceed below with a variant involving disconnected graphs without basepoint.  It does not seem to literally fit into the axioms of a Waldhausen category, but we will explain an explicit construction which seems similar in spirit, and which we shall therefore denote $\BKGr$.
\end{remark}

\subsection{\texorpdfstring{$K$}{K}-theory of graphs} \label{sec:k-theory-graphs}

We use the same definition as in \cite{cgp-graph-homology}: a \emph{graph} is a finite set $X$, together with functions $i, r\col X\to X$ that satisfy $i^2 = 1_X$ and $r^2=r$, and such that
\[\{x\in X \mid r(x) = x\} = \{x\in X \mid i(x)=x\}.\]  
We write $V = \{x \in X \mid i(x) = x\}$, $H = X \setminus V$, and $E = H/(x \sim i(x))$, for the sets of vertices, half-edges, and edges of $G$, respectively.  The most general notion of morphism from $G = (X,i,r)$ to $G' = (X',i',r')$ is just a map of sets $\phi\col X \to X'$ such that $i' \circ \phi = i$ and $r' \circ \phi = r$; we do not require that it send $H$ to $H'$.  When we impose extra conditions on graphs (non-empty, connected, no zero-valent vertices, etc.) or graph morphisms, we say so explicitly. A morphism of graphs is injective (resp.~surjective) if the underlying map on sets is injective (resp.~surjective).

While officially there are no extra conditions on morphisms, in practice only a restricted class of morphisms will appear in the diagrams of graphs below.  For example, surjective morphisms of graphs can in principle increase first Betti number by identifying vertices (for example, one may map the two endpoints of an edge to a single vertex, to create a self-edge). But such a morphism will never arise in the diagrams of graphs considered below.

Since we allow vertices of any valence including 0, the category of finite sets may be identified via $S\mapsto (S, \mathrm{id},\mathrm{id})$ with a full subcategory of the category of graphs, namely those graphs $G$ for which $E(G) = \emptyset$.  Any graph $G$ comes with universal maps to and from a set (more precisely, a map of graphs from $G$ to a graph with no edges, initial among such, as well as a map of graphs  from a graph with no edges to $G$, terminal among such) namely
\begin{equation*}
  V(G) \hookrightarrow G \twoheadrightarrow \pi_0(G),
\end{equation*}
where $V(G)$ is the set of vertices and $\pi_0(G)$ is the set of path components of $G$.

We now construct a $K$-theory space of graphs, corresponding to a simplicial category $\cF$ that we now define.
For an object $[p] \in \Delta$, $\cF_p$ is a groupoid whose object set is the set of diagrams of the shape
\begin{equation}\label{eq:10}
  \begin{aligned}
    \xymatrix{G_{0,0} \ar@{>->}[r] & G_{0,1} \ar[d] \ar@{>->}[r] & G_{0,2} \ar[d] \ar@{>->}[r] & \cdots\ar[d]  \ar@{>->}[r] & G_{0,p}\ar[d]\\  & G_{1,1} \ar@{>->}[r] &G_{1,2} \ar[d]
      \ar@{>->}[r]  & \cdots \ar[d] \ar@{>->}[r] & G_{1,p} \ar[d] \\ &&   
      &\vdots \ar[d]  
      &\vdots \ar[d] \\ &&&G_{p-1,p-1} \ar@{>->}[r] & G_{p-1,p} \ar[d] \\ &&&& G_{p,p}}
\end{aligned}
\end{equation}
of graphs $G_{i,j}$ and morphisms between them, subject to the conditions
\begin{itemize}
\item $E(G_{i,i}) = \emptyset$ for $i = 0, \dots, p$,
\item the maps $V(G_{i,j}) \to V(G_{i,j+1})$ induced by the horizontal maps in the diagram are bijections, and $E(G_{i,j}) \to E(G_{i,j+1})$ are injections, $0 \leq i \leq j < p$,
\item the maps $\pi_0(G_{i-1,j}) \to \pi_0(G_{i,j})$ induced by the vertical maps are bijections,
\item all squares in the diagrams are pushout diagrams.
\end{itemize}

The arrows denoted $\rightarrowtail$ in the diagram~\eqref{eq:10}  are the ones that are required to be bijections on vertices and injections on edges. The conditions ensure that every vertical map is induced by contraction of edges only. 
The morphisms in $\cF_p$ are isomorphisms of such diagrams.  Deleting the $i$th row and column define functors $d_i\col \cF_p \to \cF_{p-1}$ for $0 \leq i \leq p$, and there are
also functors $s_i\col \cF_p \to \cF_{p+1}$ defined by inserting identities, making $[p]\mapsto \cF_p$ a simplicial object in the category of small groupoids (see Remark~\ref{rem:Fp-small}).

\begin{remark}\label{remark:graph-space-interpretations}\mbox{}
\begin{enumerate}
    \item The diagonal maps $G_{i-1,i-1} \to G_{i,i}$ are surjections of finite sets.  The diagram provides a factorization   
   \[G_{i-1,i-1} = V(G_{i-1,i-1})  \cong V(G_{i-1,i}) \twoheadrightarrow \pi_0(G_{i-1,i}) \cong \pi_0(G_{i,i}) \cong G_{i,i}\]
    identifying $G_{i-1,i-1} \twoheadrightarrow G_{i,i}$ with the quotient by the equivalence relation generated by the edge set of $G_{i-1,i}$.
    \item
      The diagram \eqref{eq:10} is determined up to isomorphism by its top row; since each rectangle
      \[\squarediagram{G_{i-1,i}}{G_{i-1,j}}{G_{i,i}}{G_{i,j}}\]
      is a pushout diagram,
      each vertical
  map $G_{i-1,j} \to G_{i,j}$ must be induced by collapsing each edge in the image of $G_{i-1,i} \hookrightarrow G_{i-1,j}$.  The whole diagram can therefore be reconstructed up to isomorphism from the graph $G_{0,p}$ together with the flag of subsets of $E(G_{0,p})$ given by the images of $E(G_{0,j}) \hookrightarrow E(G_{0,p})$.  As in the Waldhausen construction, it is, for set-theoretic reasons, convenient to include the data of chosen subquotients.
    \item  Similarly, the diagram is  determined up to isomorphism by its rightmost column.
\end{enumerate}
\end{remark}
\smallskip

We define
\begin{equation}\label{eq:bkgr-is-born}
    \BKGr = |N_\bullet \cF_\bullet|.
\end{equation}
There is a functor
\begin{equation*}
  \cF_p \stackrel{H_1}\longrightarrow S_p(\Proj{\ZZ}), \qquad G_{i,j} \mapsto H_1(G_{i,j};\ZZ),
\end{equation*}
which induces a map of bisimplicial sets and in turn of topological spaces
\begin{align}
  N_\bu\cF_\bu & \stackrel{H_1} \longrightarrow N_\bu S_\bu (\Proj{\ZZ}),\\
  \BKGr = |N_\bullet \cF_\bullet| & \longrightarrow |N_\bullet S_\bullet(\Proj{\ZZ})| = BK(\ZZ).
  \label{eq:9}
\end{align}

\begin{remark} \label{rem:Fp-small}
  For set-theoretic reasons, the above definitions should be augmented with a choice of small category $\mathcal{G}$ equivalent to all graphs and all morphisms (for instance by insisting $G = (X,i,r)$ where $X \subset \Omega$ is a subset of some fixed infinite set, or by choosing a graph of each isomorphism type).  Then we get a small category
  $\cF_p$ whose object set is the set of diagrams in $\cG$ of the form~(\ref{eq:10}), subject to the stated requirements.
\end{remark}

\subsection{Filtrations and graph spectral sequences}

\begin{definition} For $g \in \NN$, let $F_g \cF_p \subset \cF_p$ be the full subcategory containing those objects~(\ref{eq:10}) with $b_1(G_{0,p}) \leq g$. 
  Similarly, let $F_g \BKGr \subset \BKGr$ be the image of the map induced by $N_\bu (F_g \cF_\bu) \to N_\bu \cF_\bu$, the nerve of the inclusion of the subcategory.
\end{definition}
We have now defined filtrations on both spaces in~(\ref{eq:9}), and it is clear that the map preserves the filtration.  Therefore there is an induced map of spectral sequences
\begin{equation}\label{eq:21}
  \sseq{G}^r_{*,*} \to \sseq{Q}^r_{*,*},
\end{equation}
converging to the rational homology of the spaces~(\ref{eq:9}).  We will call $\sseq{G}^r_{*,*}$ the ``graph spectral sequence.''

\begin{proposition}\label{prop:graph-spectral sequence}
  The constructions in Section~\ref{sec:copr-filtr-waldh} apply to the bisimplicial set $X_{p,q} = N_q \cF_p$ filtered as $F_g X_{p,q} = N_q (F_g \cF_p)$, leading to a space-level filtered coproduct
  \begin{equation*}
    |X| \xrightarrow{\approx} |\es(X)| \to |X| \times |X|.
  \end{equation*}

  Choosing a disjoint union operation $(G,G') \mapsto m(G,G') \cong G \sqcup G'$ on the chosen small category equivalent to all finite graphs and promoting it to a symmetric monoidal functor leads to a map of bisimplicial sets $X \times X \to X$ inducing a product
  \begin{equation*}
      m\col |X| \times |X| \xrightarrow{\approx} |X \times X| \to |X|.
  \end{equation*}
  The product and coproduct are both filtered maps, the product is associative and commutative up to a filtration-preserving homotopy, and the coproduct is associative up to a filtration-preserving homotopy.  Forgetting filtrations, the coproduct is homotopic to the diagonal map of $|X|$.
\end{proposition}

\begin{corollary} \label{cor:graph-spectral-sequence} 
  Let $
    \sseq{G}^1_{*,*} \Rightarrow H_*(\BKGr)$ 
  be the homological spectral sequence associated to the filtration on $\BKGr$ defined above.  Then $\sseq{G}^*_{*,*}$ admits the structure of a spectral sequence of bialgebras, with graded commutative product on all pages and graded co-commutative co-product on $E^\infty$.  
\end{corollary}
As we will see later, the graph spectral sequence will have $E^1_{0,0} \cong \QQ[x]/(x^2 - x)$, with unit $1$ the class of the empty graph and $x \in E^1_{0,0}$ the class of a graph consisting of a single zero-valent vertex.  In this bidegree the coproduct is given by $\Delta(x) = x \otimes x$ and the augmentation by $x \mapsto 1$.  This bigraded bialgebra is not connected and the existence of an antipode is not automatic, hence ``bialgebra'' instead of ``Hopf algebra'' in the above statement.  (And indeed it does not admit an antipode: recall that $\mathrm{Spec}$ of a commutative bialgebra is a monoid scheme, and is a group scheme if and only if the bialgebra is part of a Hopf algebra structure.  See, e.g., \cite[Theorem 5.1]{milneAGS}.  The bialgebra $E^1_{0,0}$ represents the functor sending a commutative ring $R$ to its monoid of idempotent elements---this defines a monoid scheme $\Spec(\QQ[x]/(x^2-x))$ which is not a group scheme.)  
\begin{proof}[Proof sketch for Proposition~\ref{prop:graph-spectral sequence}]
  The coproduct is constructed just as for $BK(\ZZ)$, using the filtration and the bisimplicial structure.  
  The coassociativity of the coproduct, up to a homotopy respecting the filtration, follows from  Proposition~\ref{prop:(3)} and the observation that given
  \[G_0 \rightarrowtail \cdots \rightarrowtail G_{2p+1},\]
  where the arrows $\rightarrowtail$ are bijections on vertices and injections on edges, the inequality 
  \[b_1(G_{2p+1}) \ge b_1(G_{p}) + b_1(G_{2p+1}/G_{p+1})\]
  holds.  From this observation, it follows that
  \[\es(F_s (N_\bu F_\bu)) \Rightarrow F_s (N_\bu F_\bu \times  N_\bu F_\bu) \Rightarrow F_s (N_\bu F_\bu) \times F_s (N_\bu F_\bu),\]
  so that the hypotheses of   Proposition~\ref{prop:(3)} are satisfied.
  
  The product  on $S_p(\Proj{\ZZ})$ came from the symmetric monoidal structure induced by direct sum, but we can define a symmetric monoidal
  structure on $\cF_p$ by disjoint union of graphs, and clearly $H_1\col \cF_p \to S_p(\Proj{\ZZ})$ promotes to a symmetric monoidal functor.
\end{proof}

In the remainder of this section, we explain why a filtered space rationally equivalent to $\BKGr$ may be interpreted as a moduli space of possibly-disconnected metric graphs, and explain how the $E^1$-page may be identified with homology of a certain graph complex.

\defnow{
  Let $\overline{\cF}_p$ denote the set of isomorphism classes in the groupoid $\cF_p$, in other words the coequalizer of $d_0, d_1\col N_1 \cF_p \to N_0 \cF_p$ in the category of sets, or the set $\pi_0(|N_\bullet \cF_p|)$ of path components of the geometric realization of the nerve of $\cF_p$.
}
The sets $\ov \cF_p$ assemble into a simplicial set $\ov \cF_\bu$. We set the notation
\begin{equation}\label{eq:mgr-is-born}
    \MGr = |\ov\cF_\bu|,
\end{equation}
since, as explained in Section~\ref{sec:cell-decomposition-of-graph-spaces} below, this space has an interpretation as a coarse moduli space of graphs.
Regarding $\ov \cF_p$ as a space with discrete topology, 
the canonical maps
\begin{equation}
  \label{eq:12}
  |N_\bullet \cF_p| \to \overline{\cF}_p
\end{equation}
for all $p$ assemble to a map of simplicial spaces, with geometric realization
\begin{equation}
  \label{eq:13}
  \BKGr = |N_\bullet \cF_\bullet| \to |\overline\cF_\bullet| = \MGr.
\end{equation}
The space $\MGr$ does not map to $BK(\ZZ)$, but is a simpler object than $\BKGr$ in that it arises from a simplicial rather than  bisimplicial set.  Furthermore, it is a good model up to rational equivalence:
\begin{lemma}\label{lemma:rational-iso}
  The maps~(\ref{eq:12}) and~(\ref{eq:13}) induce isomorphisms in rational homology.  Filtering $\overline{\cF}_p$ by the images of the filtration in $N_0\cF_p$, the map~(\ref{eq:13}) also induces an isomorphism in rational homology of associated gradeds.
\end{lemma}
\begin{proof}
  The homotopy type of  $|N_\bullet \cF_p|$ is a disjoint union of $K(\pi,1)$-spaces for the automorphism groups of diagrams of the form~(\ref{eq:10}), one for each isomorphism class of such diagrams.  Since any finite group has the rational homology of a point, the map~(\ref{eq:12}) is a rational equivalence for each $p$.  Filtering the realization in the $p$-direction by skeletons shows that~(\ref{eq:13}) is too.
\end{proof}

\subsection{Rational cell decomposition of graph spaces}\label{sec:cell-decomposition-of-graph-spaces}

The geometric realization $\MGr$ can be described rather explicitly, as an iterated ``rational cell attachment'' in which each cell is the cone over the quotient of a sphere by a finite group, i.e., a symmetric CW-complex as defined in \cite{allcock-corey-payne-tropical}.  The cells here most naturally have cubical shape, attaching the cube $(\Delta^1)^n$ to a space $X$ along an attaching map $e \colon \partial (\Delta^1)^n \to X$, and more generally as a pushout of the form
\begin{equation*}
  (\Delta^1)^n/H \leftarrow \partial (\Delta^1)^n/H \xrightarrow{e} X
\end{equation*}
for a subgroup $H < S_n$ acting by permuting coordinates.  When $H \leq A_n$ then $ \partial (\Delta^1)^n/H$ is a rational homology $S^{n-1}$ and the attachment has the same effect in rational homology as attaching an ordinary $n$-cell in the usual sense.  When $H$ contains an odd permutation, then the attachment does not change the rational homology.  As mentioned in Remark~\ref{rem:cubical-spaces}, this description of $\MGr$ as a cubical space of graphs, and the graph complex we shall associate to it in Section~\ref{sec:graph-compl-model}, is a variant of the cubical spaces of graphs introduced by Hatcher--Vogtmann \cite[\S3]{hatcher-vogtmann-rational} in their work computing rational homology groups of $\on{Aut}(F_n)$.  It is also 
similar in spirit, although not isomorphic, to the notion of ``symmetric $\Delta$-complex'' in \cite{cgp-graph-homology}: in that paper an iterated attachment of the pair $\partial \Delta^{n-1}/H \subset \Delta^{n-1}/H$ was used, a rational homology $(n-1)$-disk, instead of the pair $\partial (\Delta^1)^n/H \subset (\Delta^1)^n/H$, a rational homology $n$-disk. 

Observe that any graph $G$ gives rise to a map of simplicial sets
\begin{equation*}
  (\Delta^1_\bullet)^{E(G)} \to N_0(\cF_\bullet) \to \overline{\cF}_\bullet,
\end{equation*}
where $\Delta^1_\bullet$ is the usual representable simplicial set $[p] \mapsto \Delta([p],[1])$ with geometric realization $\Delta^1$.  
We now explain this map precisely.  Given a $p$-simplex $(f_e)_{e\in E(G)}$ of $(\Delta^1_\bu)^{E(G)}$, where $f_e\in \Delta([p],[1])$, we associate the sequence of subsets of $E(G)$
\begin{equation}\label{eq:graphs-sequence}
  \emptyset \subset E_0 \subset \dots \subset E_p \subset E(G)
\end{equation}
in which
\begin{equation}\label{eq:edge-subset}
  E_j = \{e\in E(G) \mid f_e(j) = 1\}.
\end{equation}
We may now form a diagram of graphs of the form~\eqref{eq:10}, by first writing $G_{-1,j}$ for the graph with $V(G_{-1,j}) = V(G)$ and $E(G_{-1,j}) = E_j$, and then for $0 \leq i \leq j \leq p$ defining $G_{i,j}$ as the quotient of $G_{-1,j}$ obtained by collapsing those edges $e \in E(G_{-1,j}) = E_j$ which are elements of the subset $E_i \subset E_j$.  This recipe gives a diagram of the form~(\ref{eq:10}) satisfying all requirements, except that the chosen quotient graphs $G_{i,j}$ may not literally be in the object set of $\mathcal{G}$.  (Recall that we chose a small category in order for $N_\bullet \cF_\bullet$ to be a bisimplicial \emph{set}.)  Each such diagram is certainly \emph{isomorphic} to a diagram in $\mathcal{G}$, and choosing an isomorphic diagram in $\mathcal{G}$ for each $p$ and each $f \in \Delta^1_p$ produces a map of simplicial sets $(\Delta^1_\bullet)^{E(G)} \to N_0 \cF_\bullet$.
The map $(\Delta^1_\bu)^{E(G)} \to N_0(\cF_\bu)$ constructed in this way depends on choices, but the composition $(\Delta^1_\bu)^{E(G)} \to N_0(\cF_\bu) \to \ov \cF_\bu$ does not, and factors over the quotient $(\Delta^1_\bu)^{E(G)}/\mathrm{Aut}(G)$. This gives commutative diagrams
\begin{equation*}
  \begin{tikzcd}
    (\Delta^1_\bullet)^{E(G)} \rar\dar & N_0(\cF_\bullet) \dar\\
    (\Delta^1_\bullet)^{E(G)}/\mathrm{Aut}(G) \rar &\overline{\cF}_\bullet.
  \end{tikzcd}
\end{equation*}
Recall that the geometric realization functor preserves small colimits and finite products, and that $|\Delta^1_\bu| = \Delta^1$,
so the induced maps of geometric realizations may be written as
\begin{equation}\label{eq:14}
  (\Delta^1)^{E(G)}/\mathrm{Aut}(G) \xleftarrow{\approx} |(\Delta^1_\bu)^{E(G)}/\mathrm{Aut}(G)|  \longrightarrow \MGr.
\end{equation}
The following lemma summarizes the sense in which these canonical maps $(\Delta^1)^{E(G)}/\mathrm{Aut}(G) \to \MGr$ behave (rationally) like cells of a CW structure.  It also defines an analogue of the filtration by skeletons (not to be confused with the more interesting filtration by first Betti number, which leads to the graph spectral sequence).
\begin{lemma}\label{lem:pushout}
  Let $\overline{\cF}^{(n)}_\bullet \subset \overline{\cF}_\bullet$ denote the simplicial subset whose $p$-simplices are cut out by the condition that the cardinality of $E(G_{0,p})$ is at most $n$.  Then the maps~(\ref{eq:14}) assemble to pushout diagrams of topological spaces
  \begin{equation*}
    \begin{tikzcd}
      \coprod_G \partial (\Delta^1)^{E(G)}/\mathrm{Aut}(G) \rar \dar & {\big|\overline{\cF}_\bullet^{(n-1)}\big| } \dar\\
      \coprod_G (\Delta^1)^{E(G)}/\mathrm{Aut}(G) \rar & { \big|\overline{\cF}_\bullet^{(n)}\big|},
    \end{tikzcd}
  \end{equation*}
  where the coproduct is indexed by graphs $G$ with $|E(G)| = n$, one in each isomorphism class.
\end{lemma}
The entry $\partial (\Delta^1)^{E(G)}/\mathrm{Aut}(G)$ in the diagram should be parsed as $\big(\partial \big((\Delta^1)^{E(G)}\big)\big)/\mathrm{Aut}(G)$: the orbit space of the action of $\mathrm{Aut}(G)$ on the boundary of the cube $(\Delta^1)^{E(G)}$.

Using this presentation as an iterated pushout, we deduce descriptions of the associated graded with respect to this filtration of $\MGr$ by number of edges, as well as a description of the underlying point set.
\begin{cor}\label{cor:consequences-of-pushout}\mbox{}
  \begin{enumerate}
  \item 
    The maps~(\ref{eq:14}) induce homeomorphisms of pointed spaces
    \begin{equation*}
      \bigvee_{|E(G)| = n} (S^1)^{\wedge E(G)}/\mathrm{Aut}(G) \longrightarrow \frac{|\ov \cF^{(n)}_\bullet|}{|\ov \cF^{(n-1)}_\bullet|}
    \end{equation*}
    where the coproduct of pointed spaces is indexed by graphs with precisely $n$ edges, one in each isomorphism class, and $S^n \approx (S^1)^{\wedge E(G)}$ is the $E(G)$-fold smash product of the circle with itself, on which $\mathrm{Aut}(G)$ acts by permuting factors.
  \item The maps~(\ref{eq:14}) induce a continuous bijection
    \begin{equation*}
      \coprod_G (\Delta^1 \setminus \partial \Delta^1)^{E(G)}/\mathrm{Aut}(G) \longrightarrow \MGr,
    \end{equation*}
    where the disjoint union is over all graphs, one in each isomorphism class.  
  \end{enumerate}
\end{cor}
\begin{remark}\label{rem:interpretation-as-moduli-of-metric-graphs}
  The second part of the corollary shows that the set underlying $\MGr$ may be interpreted as the set of isomorphism classes of pairs $(G,\ell)$ where $\ell \colon E(G) \to \Delta^1 \setminus \partial \Delta^1$ is any function, and $[G,\ell] = [G',\ell']$ if there exists an isomorphism $\phi \colon G \to G'$ such that $\ell'(\phi(e)) = \ell(e)$ for all $e \in E(G)$.  For psychological reasons it may be preferable to rescale using a homeomorphism $\Delta^1 \setminus \partial \Delta^1 \to (0,\infty)$ such as $(t_0,t_1) \mapsto -\log(t_1)$.  This gives an interpretation of $\MGr$ as a coarse moduli space of metric graphs, where edges have assigned positive real lengths.  In this picture, the topology may described informally as ``edges are collapsed as their lengths go to zero, and edges are deleted as their lengths go to infinity.''
\end{remark}
\begin{proof}[Proof of Lemma~\ref{lem:pushout}]
  Inspecting the definition of the canonical map~(\ref{eq:14}), we see that it factors through the subspace $|\ov\cF^{(n)}_\bullet|$ when $|E(G)| \leq n$ and that for each $e \in E(G)$, its composition with
  \begin{equation*}
    (\partial \Delta^1)^{\{e\}} \times (\Delta^1)^{E(G) \setminus \{e\}} \hookrightarrow (\Delta^1)^{E(G)} \twoheadrightarrow (\Delta^1)^{E(G)}/\mathrm{Aut}(G)
  \end{equation*}
  will factor through $|\ov\cF^{(n-1)}_\bu|$ because the only graphs that will appear are subquotients of $G \setminus e$ and $G/e$, so they will have strictly fewer edges than $G$ has.  This observation explains why~(\ref{eq:14}) induce horizontal maps in the square diagram in the lemma, and it is then clear that the diagram commutes.

  To prove that the diagram is a pushout, it suffices to consider the corresponding diagram of simplicial sets, again because geometric realization preserves small colimits.  Pushouts in simplicial sets are computed degree-wise, so it suffices to prove that the diagram of $p$-simplices
  \begin{equation*}
    \begin{tikzcd}
      \coprod_G (\partial (\Delta^1_\bullet)^{E(G)})_p/\mathrm{Aut}(G) \rar \dar & {\overline{\cF}_p^{(n-1)} } \dar\\
      \coprod_G (\Delta^1_p)^{E(G)}/\mathrm{Aut}(G) \rar & { \overline{\cF}_p^{(n)}}
    \end{tikzcd}
  \end{equation*}
  is a pushout in sets.  Both vertical maps are injective. We must show that the horizontal map restricts  to a bijection between the complements of the  images of the vertical maps.

  Before passing to $\mathrm{Aut}(G)$-orbits, the subset $(\partial (\Delta^1_\bu)^{E(G)})_p \subset (\Delta^1_p)^{E(G)} = \Delta([p],[1])^{E(G)}$ corresponds to the set of poset maps $[p] \to [1]^{E(G)}$ where at least one coordinate is constant.  The complement is therefore precisely the set
  \begin{equation*}
    (\Delta^1_p)^{E(G)} \setminus (\partial (\Delta^1_\bu)^{E(G)})_p = \Delta_\mathrm{surj}([p],[1])^{E(G)}
  \end{equation*}
  consisting of $E(G)$-tuples of morphisms $[p] \to [1]$ in the simplex category $\Delta$ whose underlying map of sets is surjective.

  The proof is concluded by firstly observing that the set map  $(\Delta^1_p)^{E(G)} \to \ov\cF_p$ defined above will send $\Delta_\mathrm{surj}([p],[1])^{E(G)}$ into $\ov\cF^{(n)}_p \setminus \ov \cF^{(n-1)}_p$ if $|E(G)| = n$: indeed, in this case the flag of subsets~(\ref{eq:graphs-sequence}) of $E(G)$ will have $E_0 = \emptyset$ and $E_p = E(G)$ and hence $G_{0,p} = G$.  Secondly, the resulting map of sets
  \begin{equation*}
    \coprod_{|E(G)| = n} \Delta_\mathrm{surj}([p],[1])^{E(G)}/\mathrm{Aut}(G) \to \ov\cF^{(n)}_p \setminus \ov\cF^{(n-1)}_p
  \end{equation*}
  is surjective because any diagram of the form~(\ref{eq:10}) with $|E(G_{0,p})| = n$ arises from $G = G_{0,p}$ and the flag of subsets $\emptyset = E_0 \subset \dots \subset E_p = E(G)$ such that $E(G_{0,j}) = E_j$, corresponding to some unique $f = (f_e)_{e \in E(G)} \in \Delta_\mathrm{surj}([p],[1])^{E(G)}$.  Conversely the map is injective because this flag of subsets of $E(G)$ is uniquely determined up to automorphisms of $G$.
\end{proof}

\subsection{A graph complex modeling \texorpdfstring{$C_*(\MGr)$}{C(M_Gr)}}
\label{sec:graph-compl-model}

If $G$ is a graph with $|E(G)| = n$ then any choice of bijection $\omega\col \{0, \dots, n-1\} \to E(G)$ induces a homeomorphism $S^n = (S^1)^{\wedge n} \approx (S^1)^{\wedge E(G)}$ and induces a group homomorphism $\mathrm{Aut}(G) \to S_n$.  Let us write $H = H_{G,\omega} < S_n$ for the image of this group homomorphism.  Then $\omega$ also induces a homeomorphism
\begin{equation*}
  (S^1)^{\wedge E(G)}/\mathrm{Aut}(G) \xrightarrow{\cong} S^n/H
\end{equation*}
and the quotient map $S^n \to S^n/H$ is a rational homology isomorphism if $H < A_n$, whereas $S^n/H$ has the rational homology of a point when $H$ contains an odd permutation.  In the eyes of rational homology, the filtration of $\MGr$ by number of edges therefore behaves like a CW structure whose $n$-cells are in bijection with the set of isomorphism classes of graph $G$ with $|E(G)| = n$ and such that $G$ does not admit any automorphisms inducing an odd permutation of $E(G)$.  This leads to a ``cellular chain complex'' quasi-isomorphic to $C_*^\mathrm{sing}(\MGr;\QQ)$.  This cellular chain complex has a combinatorial description which we spell out, as well as the filtration corresponding to first Betti number.  The latter gives a ``graph complex'' point of view on the spectral sequence of bialgebras mapping to the Quillen spectral sequence.

Let $C = (C_*, \partial)$ be the rational chain complex generated by symbols $[G,\omega]$, where $G$ is any graph and $\omega$ is a total order on $E(G)$, subject to the usual relation $[G,\omega] = \mathrm{sgn}(\sigma) [G',\omega']$ for each isomorphism $G \to G'$ with respect to which $\omega$ and $\omega'$ are related by permutation $\sigma$.  We do not impose any conditions whatsoever on $G = (X,i,r)$:  in particular, we allow disconnected graphs, bridges, loops, parallel edges, and vertices of any valence including zero.

The degree of the generator $[G,\omega]$ is declared to be $|E(G)|$, and we write $C_p \subset C$ for the span of those $[G,\omega]$ for which $|E(G)| = p$.  We will return to the boundary map shortly, but first we define the comparison map to $C_*^\mathrm{sing}(\MGr;\QQ).$  To this end, recall that the identity map of $\Delta^1$, regarded as a chain $\iota_1 \in C_1(\Delta^1;\ZZ)$, represents the fundamental class in $H_1(\Delta^1,\partial\Delta^1;\ZZ)$.  Then we use the cross product (\cite[p.277--278]{hatcher}) natural transformation $C_p(X) \otimes C_q(Y) \to C_{p+q}(X\times Y)$ to define a system of fundamental chains
\begin{equation}\label{eq:cube}
  \iota_n \in C_n((\Delta^1)^n;\ZZ),
\end{equation}
compatible in the sense that $\iota_n \times \iota_m = \iota_{n+m}$.  There is also an explicit formula for $\iota_n$ as a signed sum of $n!$ many maps $\Delta^n \to (\Delta^1)^n$, corresponding to all the non-degenerate $n$-simplices of $(\Delta^1_\bu)^n$ with appropriate signs, but we will not need this explicit formula.
\begin{definition}\label{def:five-fourteen}
  For a graph $G$ with $|E(G)| = n$ and a total ordering $\omega \colon \{0, \dots, n-1\} \to E(G)$, write $f^G\colon (\Delta^1)^E \to \MGr$ for the canonical map defined as in~(\ref{eq:14}), and $f^{G,\omega} \colon (\Delta^1)^n \to \MGr$ for the composition
  \begin{equation*}
    (\Delta^1)^n \approx (\Delta^1)^{E(G)} \xrightarrow{f^G} \MGr
  \end{equation*}
  with the homeomorphism induced by the bijection $\omega$.  Then define the linear map 
  \begin{equation}\label{eq:to-singular-chain}
    \begin{aligned}
      C_p & \to 
      C_p^{\mathrm{sing}}(\MGr;\QQ), \\ 
      [G,\omega] & \mapsto f^{G,\omega}_*(\iota_n).
    \end{aligned}
  \end{equation}
\end{definition}
Filtering by the number of edges, we deduce the following from Corollary~\ref{cor:consequences-of-pushout}.
\begin{lemma}
  The relative homology
  \begin{equation*}
    H_*\big(\big|\ov\cF^{(p)}_\bu\big|,\big|\ov\cF^{(p-1)}_\bu\big|;\QQ\big)
  \end{equation*}
  vanishes for $* \neq p$.    If $|E(G)| = p$, then the homomorphism~(\ref{eq:to-singular-chain}) sends $[G,\omega]$ into the subspace $C_p(|\ov\cF^{(p)}_\bu|;\QQ)$.  Its image in relative chains
  \begin{equation*}
    f^{G,\omega}_*(\iota_n) \in C_p\big(\big|\ov\cF^{(p)}_\bu\big|,\big|\ov\cF^{(p-1)}_\bu\big|;\QQ\big)
  \end{equation*}
  is a cycle, and the resulting map of rational vector spaces
  \begin{align*}
    C_p & \to H_p\big(\big|\ov\cF^{(p)}_\bu\big|,\big|\ov\cF^{(p-1)}_\bu\big|;\QQ\big)\\
    [G,\omega] & \mapsto 
                 \big[f^{G,\omega}_*(\iota_n)\big]
  \end{align*}
  is an isomorphism.
\end{lemma}
\begin{proof}[Proof sketch]
  This follows from the description of $\big|\ov \cF^{(p)}_\bu\big|/\big|\ov \cF^{(p-1)}_\bu\big|$ in Corollary~\ref{cor:consequences-of-pushout}, using the fact that $[\iota_n] \in H_n((\Delta^1)^n/H,\partial (\Delta^1)^n/H;\QQ)$ is a generator, for any subgroup $H \leq A_n$.
\end{proof}

We have not yet defined the boundary operator $\partial \colon C_p \to C_{p-1}$, but let us do that now.  
It is the difference $\partial = \partial_c - \partial_d$ between two operators,
one which deletes an edge in all possible ways, and one which contracts an edge in all possible ways.  Precisely, for $G$ a graph with $p+1$ edges and $\omega = e_0< \cdots < e_p$ a total ordering on $E(G)$, define
  \[\partial_d [G,\omega] = \sum_{i=0}^p (-1)^i \big[G\backslash e_i, \omega|_{E(G)\setminus\{e_i\}}\big],\quad \mbox{ and } \quad \partial_c [G,\omega] = \sum_{i=0}^p (-1)^i \big[G/ e_i, \omega|_{E(G)\setminus\{e_i\}}\big],\]
  where $\omega|_{E(G)\setminus\{e_i\}}$ is the total order induced from $\omega$ by omitting $e_i$.

\begin{lemma}\label{lemma: the map is a chain map}
  The map
  \begin{equation*}
    \begin{aligned}
      C & \to C_*^{\mathrm{sing}}(\MGr;\QQ)\\
      [G,\omega] & \mapsto 
      f^{G,\omega}_*(\iota_n)
    \end{aligned}
  \end{equation*}
  is a chain map.
\end{lemma}
\begin{proof}
  The boundary of $(\Delta^1)^n$ is the union of the $2n$ 
  many embeddings
  \begin{equation}\label{eq:17}
    (\Delta^1)^{n-1} \xleftarrow{\approx} (\Delta^1)^j \times \Delta^0 \times (\Delta^1)^{n-j-1} \xrightarrow{1 \times d^i \times 1} (\Delta^1)^n
  \end{equation}
  for $i \in \{0,1\}$ and $j \in \{0, \dots, n-1\}$.  Writing $\partial_{i,j} \colon (\Delta^1)^{n-1} \to (\Delta^1)^n$ for this embedding and writing $\iota_n \in C_n^{\mathrm{sing}}((\Delta^1)^n;\QQ)$ for the fundamental cycle~\eqref{eq:cube}, the Leibniz rule for cross product leads to the convenient formula 
  \begin{equation}\label{eq:18}
  \begin{aligned}
    \partial \iota_n & = \sum_{i = 1}^n (-1)^{i-1}\iota_{i-1} \times (\partial \iota_1) \times \iota_{n-i} \in C_{n-1}(\partial (\Delta^1)^n;\ZZ)\\
                     & =
                       \sum_{i = 1}^{n} (-1)^{i-1} (\partial_{0,j})_* \iota_{n-1} - \sum_{i = 1}^{n} (-1)^{i-1} (\partial_{1,j})_* \iota_{n-1}.
  \end{aligned}
\end{equation}
  The lemma then follows because the composition
  \begin{equation*}
    (\Delta^1)^{n-1} \xrightarrow{\partial_{i,j}} (\Delta^1)^n \xrightarrow{f^{G,\omega}} \MGr
  \end{equation*}
  is $f^{G',\omega'}$, where $G'$ is the graph obtained from $G$ by either collapsing (for $i = 0$) or deleting (for $i = 1$) the $j$th edge in $G$, and $\omega'$ is the restriction of the total order $\omega$ to $E(G') = E(G) \setminus \{e_j\}$.
\end{proof}

In order to obtain the graph spectral sequence, we must consider $\MGr$ with its filtration by the subspaces 
\begin{equation*}
  F_n \MGr = |F_n \ov\cF_\bu| \subset \MGr
\end{equation*}
where $F_n \ov\cF_p$ corresponds to isomorphism classes of triangular diagrams in which $b_1(G_{0,p}) \leq n$.  Fortunately the entire argument for the quasi-isomorphism $C = (C_*,\partial) \simeq C_*^\mathrm{sing}(\MGr;\QQ)$ applies essentially verbatim to the simplicial subset $F_n\ov\cF_\bu \subset \ov\cF_\bu$.  We state the conclusion.
\begin{thm}
  Let $F_n C \subset C$ be the span of those generators $[G,\omega]$ for which $b_1(G) \leq n$.  This defines a subcomplex for all $n$, and the formula~(\ref{eq:to-singular-chain}) restricts to quasi-isomorphisms
  \begin{equation*}
    F_n C \xrightarrow{\simeq} C_*^\mathrm{sing}(F_n\MGr;\QQ)
  \end{equation*}
  for all $n$.  In other words, formula~(\ref{eq:to-singular-chain}) defines a filtered quasi-isomorphism $C \to C_*^\mathrm{sing}(\MGr;\QQ)$.\qed
\end{thm}
By this theorem, the graph spectral sequence is isomorphic to the spectral sequence associated to the filtered chain complex $(C, F)$ starting on $E^1$, at least as a spectral sequence of rational vector spaces.  It is rather easy to upgrade this isomorphism to one of spectral sequences of algebras, where $(C,F)$ is upgraded to a filtered DGA by declaring $[G,\omega][G',\omega'] = [G \sqcup G',\omega\ast \omega']$, where $\omega \ast \omega'$ denotes the ``concatenation'' order of $E(G \sqcup G') = E(G) \sqcup E(G')$ where all edges of $G$ come before all edges of $G'$.  The coproduct is rather more involved. 
\begin{thm} \label{thm:  graphcoproduct}
  Let $C = (C_*,\partial)$ have the coproduct induced by
  \begin{equation} \label{eqn: graphcoproduct}
    \Delta([G,\omega]) = \sum_{\gamma \subset E(G)} \pm [G_{\vert \gamma},\omega_{\vert \gamma}] \otimes [G/\gamma,\omega_{\vert E(G)\setminus \gamma}]
  \end{equation}
  where $V(G_{\vert \gamma}) = V(G)$ and $E(G_{\vert \gamma}) = \gamma$ while $G/\gamma$ is the graph obtained by collapsing each element of $\gamma$ to a point.  The edge sets of $G_{\vert \gamma}$ and $G/\gamma$ are in canonical bijection with $\gamma$ and $E(G)\setminus \gamma$ respectively, and $\omega_{\vert \gamma}$ and  $\omega_{\vert E(G)\setminus \gamma}$ denote the restrictions of the total order $\omega$.  Finally, the sign $\pm$ in each term is the sign of the shuffle $S_\gamma \in S_{E(G)}$ which shuffles the elements of $\gamma \subset E(G)$ before the elements of $E(G) \setminus\gamma$. 
  \begin{enumerate}
  \item This coproduct makes $(C,F)$ into a filtered coalgebra.
  \item The map~(\ref{eq:to-singular-chain}) is compatible with the map induced by the space level coproduct $\MGr = |\ov \cF_\bu| \approx |\es\ov\cF_\bu| \to |\ov\cF_\bu| \times |\ov\cF_\bu| = \MGr \times \MGr$: the diagram
    \begin{equation}\label{eq:19}
      \begin{tikzcd}
        C \dar["\Delta"] \rar & C_*(\MGr;\QQ) \dar\\
        C \otimes C \rar & C_*(\MGr \times \MGr;\QQ),
      \end{tikzcd}
    \end{equation}
    commutes up to a filtration-preserving chain homotopy,
    where the lower horizontal map is the composition of~(\ref{eq:to-singular-chain})$\otimes$(\ref{eq:to-singular-chain}) and the cross product.
  \end{enumerate}
\end{thm}

\begin{proof}[Proof sketch]
  It is easy to see that $b_1(G_{\vert \gamma}) + b_1(G/\gamma) = b_1(G)$.

  The second statement is more involved, and the diagram~\eqref{eq:19} does not commute on the nose.  The starting point for the homotopy is the chain $\iota_1' = \iota_1 + \partial c \in C_1(\Delta^1;\ZZ)$
  where $c \in C_2(\Delta^1;\ZZ)$ is the simplex
  \begin{align*}
    \Delta^2 & \to \Delta^1\\
    (t_0,t_1,t_2) & \mapsto t_0(1,0) + t_1\big(\tfrac12,\tfrac12\big) + t_2 (0,1).
  \end{align*}
  In these formulas we use barycentric coordinates $(t_0, \dots, t_n) \in \Delta^n$, in which the simplex is defined by $t_i \geq 0$ and $\sum t_i = 1$. 
  In other words, $\iota_1'$ is the sum (in the abelian group $C_1(\Delta^1;\ZZ)$)  of  the two maps
  \begin{align*}
    \Delta^1 & \to \Delta^1\\
    (t_0,t_1) & \mapsto \big(\tfrac12 t_0, 1 - \tfrac12 t_0\big)\\
    (t_0,t_1) & \mapsto \big(1 - \tfrac12 t_1,\tfrac12 t_1\big),
  \end{align*}
  precisely the (geometric realization of the) two non-degenerate 1-simplices in the 
  edgewise subdivision of $\Delta^1_\bu$.  It follows that the composition
  \begin{equation*}
    C_1(|\Delta^1_\bu|) \xrightarrow{\cong} C_1(|\es(\Delta^1_\bu)|) \to C_1(|\Delta^1_\bu| \times |\Delta^1_\bu|)
  \end{equation*}
  sends 
  $\iota_1' \mapsto d^1 \times \iota_1 + \iota_1 \times d^0$.
  
  Then
  \begin{equation*}
    \iota_n' = \iota'_1 \times \dots \times \iota'_1 \in C_n((\Delta^1)^n;\ZZ)
  \end{equation*}
  gives a different set of chain level representatives of the fundamental class of the cube $(\Delta^1)^n$, with the exact same formal properties as the $\iota_n$.  This includes the formula~(\ref{eq:18}), which holds with $\iota'_n$ and $\iota'_{n-1}$ in place of $\iota_n$ and $\iota_{n-1}$.

  The formula $\iota_n' = (\partial c) \times \iota_{n-1}' + \iota \times \iota_{n-1}'$ implies inductively that $\iota'_n = \iota_n + \partial c_n$ for
  \begin{equation*}
    c_n = \sum_{j = 1}^n \iota_{j-1} \times c \times \iota'_{n-j}.
  \end{equation*}
  Then the two chain maps
  \begin{align*}
    C & \to C_*^\mathrm{sing}(\MGr;\QQ)\\
    [G,\omega] & \mapsto f^{G,\omega}_* (\iota_n)\\
    [G,\omega] & \mapsto f^{G,\omega}_* (\iota'_n)
  \end{align*}
  are chain homotopic via a filtration-preserving chain homotopy.  One then checks that the diagram~(\ref{eq:19}) commutes strictly, if the top horizontal map is replaced by the chain homotopic map $[G,\omega]  \mapsto f^{G,\omega}_* (\iota'_n)$. 

  To see this, we study the space-level coproduct $\Delta_X\colon |X| \approx |\mathrm{es}(X)| \to |X| \times |X|$ from Definition~\ref{def:the-nt} and especially the induced chain map
  chain map
  \begin{equation*}
        C_n(|\Delta^1_\bu|^E)  \xrightarrow{(\Delta_{(\Delta^1_\bu)^E})_*} C_n(|\Delta^1_\bu|^E \times |\Delta^1_\bu|^E)
  \end{equation*}
  in the case $X = (\Delta^1_\bu)^E$ for a finite set $E$.  Given also a bijection $\omega\colon \{0, \dots, n-1\} \to E$ we have defined the fundamental chains $\iota_n$ and $\iota'_n$ in the domain of this chain map, and claim the formula
  \begin{equation}\label{eq:coproduct of iota-n-prime}
    (\Delta_{(\Delta^1_\bu)^E})_*(\iota'_n) = \sum_{\gamma \subset E} \pm (|L_\gamma|,|R_\gamma|)_* \iota_n,
  \end{equation}
  where $L_\gamma = \prod_{e \in E} L_{\gamma,e}$ and $R_\gamma = \prod_{e \in E} R_{\gamma,e}$ for maps of simplicial sets $L_{\gamma,e}, R_{\gamma,e}\colon \Delta^1_\bu \to \Delta^1_\bu$ given as
  \begin{align*}
    L_{\gamma,e} & =
    \begin{cases}
      \mathrm{id}  &\text{if $e \in \gamma$}\\
      d^1s^0 & \text{if $e \not\in \gamma$}
    \end{cases}
    &&
    R_{\gamma,e} =
    \begin{cases}
      d^0 s^0 &\text{if $e \in \gamma$}\\
      \mathrm{id} & \text{if $e \not\in \gamma$}.
    \end{cases}
  \end{align*}
  We have already seen formula~(\ref{eq:coproduct of iota-n-prime}) in the case $n=1$, and the general case follows by expanding the $n$-fold cross product $(d^1 \times \iota_1 + \iota_1 \times d^0)^{\times n} \in C_n((\Delta^1 \times \Delta^1)^n)$ as a sum of $2^n$ terms and using graded-commutativity of the cross product (i.e., that $\times\colon C_*(X) \otimes C_*(Y) \to C_*(X \times Y)$ turns singular chains into a lax symmetric monoidal functor from spaces to chain complexes, in the usual symmetric monoidal structure on chain complexes).  The sign arising when permuting factors is precisely the sign of the shuffle $S_\gamma \in S_{E(G)}$ mentioned in the statement of the theorem. 

  For each $\gamma \subset E$, the map $(L_\gamma,R_\gamma)\colon (\Delta^1)^E \to (\Delta^1)^E \times (\Delta^1)^E$ has $2n$ coordinates, $n$ of which are the projections $(\Delta^1)^E \to \Delta^1$ and the remaining $n$ coordinates are constant (either $d^0$ or $d^1$).  
  Setting $E = E(G)$, we have canonical maps $f^G\colon (\Delta^1)^E \to \MGr$ as in Definition~\ref{def:five-fourteen}.  Using, as in the proof of Lemma~\ref{lemma: the map is a chain map}, that precomposing with $d^1\colon \Delta^0 \to \Delta^1$ in one of the factors of $(\Delta^1)^E$ corresponds to deleting the corresponding edge, we see that $f^G \circ L_\gamma = f^{G \vert \gamma} \circ \pi_\gamma$, where $\pi_\gamma: (\Delta^1)^E \to (\Delta^1)^\gamma$ denotes the projection, and similarly $f^G \circ R_\gamma = f^{G/\gamma} \circ \pi_{E \setminus \gamma}$, using that precomposing with $d^0$ corresponds to collapsing an edge.  We can summarize this situation by a commutative diagram of spaces
  \begin{equation*}
    \begin{tikzcd}
      (\Delta^1)^E \arrow{rr}{(\pi_\gamma,\pi_{E \setminus \gamma})}[swap]{\approx}
      \dar[equal] &&
      (\Delta^1)^\gamma \times (\Delta^1)^{E \setminus \gamma} \dar[hookrightarrow] \ar[rr,"{(f^{G\vert\gamma} , f^{G/\gamma})}"] 
      & &
      {\MGr \times \MGr} \dar[equal]\\
      (\Delta^1)^E \ar[rr,hookrightarrow,"{(L_\gamma,R_\gamma)}"] & &
      (\Delta^1)^E \times (\Delta^1)^E \ar[rr,"{(f^G , f^G)}"] 
      & &
      {\MGr \times \MGr}.
    \end{tikzcd}
  \end{equation*}  
  Using the same notations $\iota_n, \iota'_n \in C_n((\Delta^1)^E)$ for the images of $\iota_n, \iota'_n \in C_n((\Delta^1)^n)$ under the homeomorphism $(\Delta^1)^n \approx (\Delta^1)^E$ induced by $\omega$, and similarly $\iota_{|\gamma|} \in C_*((\Delta^1)^\gamma)$ and $\iota_{n-|\gamma|} \in C_*((\Delta^1)^{E \setminus \gamma})$ using the restriction of the total order to the subsets $\gamma\subset E$ and $E \setminus \gamma \subset E$, naturality of $\Delta_X\colon |X| \to |X| \times |X|$ then shows that
  \begin{align*}
      (\Delta_{\ov\cF_\bu})_*f^{G}_*(\iota'_n) = \sum_{\gamma \subset E(G)} \pm (f^G,f^G)_*(L_\gamma,R_\gamma)_*\iota_n 
      & = \sum_{\gamma \subset E(G)} \pm (f^{G \vert \gamma},f^{G/\gamma})_*(\iota_{|\gamma|} \times \iota_{n-|\gamma|}) \\
      & = \sum_{\gamma \subset E(G)}\pm (f^{G \vert \gamma}_* \iota_{|\gamma|}) \times (f^{G / \gamma}_* \iota_{n-|\gamma|}).\qedhere
  \end{align*}
\end{proof}

\begin{corollary}
    Let $\Gr(C)$ denote the associated graded of the chain complex $C = (C_*,\partial)$, filtered by the first Betti numbers of the graphs, 
    with bigrading in which a generator $[G,\omega]$ has bidegree $(b_1(G), |E(G)|-b_1(G))$. The boundary map $\partial$ has bidegree $(0,-1)$.
    
    The map~\eqref{eq:to-singular-chain} induces an isomorphism of bigraded bialgebras $H_*(\Gr(C)) \to \sseq{G}^1_{*,*}$ where the coproduct on $H_*(\Gr(C))$ is induced by the Connes--Kreimer formula~\eqref{eqn: graphcoproduct}.  
\end{corollary}

Each generator $[G,\omega] \in C$ represents a non-zero element of $\Gr_g(C)$ for $g = b_1(G)$ which we denote by the same notation $[G,\omega]$.  Then $\Gr(C)$ inherits a product and coproduct from $C$, making it a bigraded bialgebra (isomorphic to $C$ as a bigraded bialgebra).  The boundary map  
$\partial = \partial_c - \partial_d\col C \to C$
induces a boundary map on $\Gr(C)$ given by a similar formula on generators $[G,\omega] \in \Gr(C)$ but omitting the terms involving $[G',\omega']$ where $b_1(G') < b_1(G)$.  Precisely, this amounts to dropping the terms in $\partial_c([G,\omega])$ in which a loop is collapsed, and omitting the terms in $\partial_d([G,\omega])$ in which a non-bridge is deleted. 
\begin{proof}
    We have already shown that~\eqref{eq:to-singular-chain} is a filtered quasi-isomorphism, so it induces an isomorphism on homology.  The filtered homotopy in the theorem implies that the isomorphism on homology is a map of coalgebras.  That it is also a map of algebras is similar but easier, we omit the details.
\end{proof}

\subsection{Comparison with the full graph complex}   
\label{sec:comparison-with-fGC}
  In the chain complex $\Gr(C)$, some remnant of $\partial_d$ will remain in the boundary map, namely a (signed) sum of bridge deletion.  Here, a {\em bridge} in a not necessarily connected graph is any edge whose deletion increases the number of connected components.  Note that deleting a bridge results in a graph that is not connected, and the result is typically zero in $\mathrm{Indec}(\Gr(C))$.  
  
Motivated by this observation, we pass to indecomposables of $\Gr(C)$ to establish a precise relationship between the $E^1$-page of the graph spectral sequence and the cochain complex $\mathsf{GC}_2$ considered by Kontsevich, Willwacher, and others, in which all graphs are connected and the differential involves only edge contraction, not deletion.  We refer to \cite[\S3]{willwacher-kontsevich} for the precise definition of $\mathsf{GC}_2$ and its variants considered below.  However, from now on we take the following bigrading on $\mathsf{GC}_2$ and all of its variants, which differs from the grading used by Willwacher: the generators in bidegree $(s,t)$ are graphs with first Betti number $s$ and $s+t$ edges.  
(This graph complex, with this bigrading, is called $\mathsf{GC}_0$ in some other sources, such as \cite{brown-invariant, willwacher-kontsevich}, but we will stick with the notation $\mathsf{GC}_2$.)

\medskip
  
Recall the full graph complex  denoted
$\mathsf{fGC}_2^\circlearrowleft$ by Willwacher in \cite[\S3]{willwacher-kontsevich} and denoted $\mathsf{fGC}_2$ in \cite{khoroshkin-willwacher-zivkovic-differentials}. Like $C^\vee$, it is also built out of all graphs with no conditions on connectivity or the valency of vertices, and where ``tadpoles'' are allowed.  Reading the definition closely though, one notices that the graphs in $\mathsf{fGC}_2^\circlearrowleft$ must be non-empty, which ours need not be.  
By comparing definitions we therefore obtain isomorphisms of rational vector spaces
\begin{equation} 
\label{eq:comparison-Cvee-fGC}
    (C/ \QQ . 1)^\vee \cong \mathsf{fGC}_2^\circlearrowleft, \quad\quad C^\vee \cong \mathsf{fGC}_2^\circlearrowleft \oplus \QQ.\emptyset^\vee,
\end{equation}
arising from bijections between basis vectors, where $\emptyset$ denotes the graph with no vertices and no edges, representing the unit $1 \in C$.  It is of some importance for us to include the empty graph, in order for our bialgebra to be unital and co-unital.  Geometrically, $(C/\QQ.1)^\vee$ can be interpreted as the relative cellular cochain complex of the pair $(\MGr,\{\emptyset\})$, where $\{\emptyset\} \subset \MGr$ is the path component corresponding to the empty graph, as in Corollary~\ref{cor:bidegree-zero-zero}.  

The usual differential $\delta$ on $\mathsf{fGC}_2^\circlearrowleft$, defined in \cite[Remark 3.3]{willwacher-kontsevich} by ``inserting an edge in all possible ways,'' does not correspond to $\partial^\vee\colon C^\vee \to C^\vee$, though: recall that 
$\partial = \partial_c - \partial_d$
where $\partial_c$ is defined as alternating sum of edge contraction while $\partial_d$ is defined as alternating sum of edge deletion, and most of the terms in $\partial_d$ are not seen in the dual formula for $\delta$ on $\mathsf{fGC}^\circlearrowleft_2$.  

Passing to indecomposables of the bialgebra $C$ corresponds to passing to the subcomplex $\mathsf{fGC}^\circlearrowleft_{2,\mathrm{conn}} \subset \mathsf{fGC}^\circlearrowleft_2$ from \cite[\S3]{willwacher-kontsevich} spanned by connected graphs.  Indeed, $\mathrm{Indec}(C) = K/K^2$ where $K$ is the kernel of the augmentation map $C\to \QQ$; here, $\QQ$ is concentrated in homological degree $0$, and the augmentation sends $[G] \mapsto 1$ for any graph $G$ with $E(G) = \emptyset$. In particular, any graph with edges is in $K$.  Let $[p] \in C$ denote the class of a graph consisting of a single vertex and no edges. We next record, for convenience, a vector space basis for $K^2$: it consists of \enumnow{\item $[G,<]$ where $G$ ranges over graphs that are disjoint unions of two graphs with edges, one per isomorphism class;
\item $[G,<]\cdot f([p])$, where $G$ is any alternating, connected graph with at least one edge, one per isomorphism class, and $f\in \QQ[t]$  ranges over a basis for polynomials with $f(1) = 0$;
\item $g([p])$ where $g\in \QQ[t]$ ranges over a basis for polynomials with $g(1)=g'(1) = 0$.} 
Here a graph is called {\em alternating} if its automorphism group acts on its edge set with only permutations of positive sign.  Then
$K/K^2$ has a basis given by classes of elements $[G,<]$ as $G$ ranges over alternating, connected graphs with at least one edge, one per isomorphism class, together with the class of $[p]-1 \in K$. Removing the class of $[p]-1$ from this set, and adding the classes of both $[p]$ and $1$, gives a basis of $C/K^2$.  Therefore 
the composition of the natural maps 
\begin{equation}\label{eq:compare-to-fgcc} 
  K/K^2\hookrightarrow C/K^2 \twoheadrightarrow C/(K^2 + \QQ. 1)
\end{equation}
is an isomorphism of chain complexes.  Now $C$ has a basis as a rational vector space given by classes $[G,<]$ where we choose one graph $G$ for each isomorphism class, and the subset given by connected (in particular non-empty) graphs maps to a basis of $C/(K^2 + \QQ. 1)$.  
The isomorphism~\eqref{eq:comparison-Cvee-fGC} therefore restricts to an isomorphism $(C/(K^2 + \QQ. 1))^\vee \cong \mathsf{fGC}^\circlearrowleft_{2,\mathrm{conn}}$, which we combine with~\eqref{eq:compare-to-fgcc} to obtain an isomorphism of vector spaces
\begin{equation*}
    \mathsf{fGC}^\circlearrowleft_{2,\mathrm{conn}} \xrightarrow{\cong} \mathrm{Indec}(C)^\vee.
\end{equation*}
We emphasize that the graded vector space $K/K^2$ is 1-dimensional in homological degree 0, spanned by the class of $[p] - 1 \in K$.  The dual basis element is denoted $([p] -1)^\vee$ and corresponds to $[p]^\vee \in \mathsf{fGC}^\circlearrowleft_{2,\mathrm{conn}}$ under the isomorphism induced by~\eqref{eq:compare-to-fgcc}.  As before, the differential $\delta$ on the left hand side of this isomorphism does not correspond to the full  
$\partial^\vee = \partial_c^\vee - \partial_d^\vee \colon C^\vee \to C^\vee$.  
This discrepancy goes away when replacing $C$ with $\Gr(C)$ though, and we have the following precise statement.
\begin{lemma}
\label{lemma:compare-to-fgcc}
    The analogue of the map~\eqref{eq:compare-to-fgcc} for the differential graded Hopf algebra $\Gr(C)$ induces an isomorphism of 
    differential graded Lie algebras
    \begin{equation}\label{eq:indec-gr}
    \mathsf{fGC}^\circlearrowleft_{2, \mathrm{conn}} \xrightarrow{\cong} \mathrm{Indec}(\Gr(C))^\vee.
    \end{equation}
\end{lemma}
\begin{proof}
    $C$ and $\Gr(C)$ are canonically isomorphic as graded bialgebras, using the basis given by the $[G,<]$ which is canonical up to signs.    Passing from $C$ to $\Gr(C)$ has the effect of omitting those terms of $\partial([G,<])$ involving graphs $G'$ with $b_1(G') < b_1(G)$, which amounts to omitting those terms of $\partial_d([G,<])$ in which a non-bridge is deleted and those terms of $\partial_c([G,<])$ in which a tadpole is collapsed.  We now inspect what happens to the boundary homomorphism when passing from $\Gr(C)$ to $\mathrm{Indec}(\Gr(C)) = K/K^2$ where $K = \mathrm{Ker}(\Gr(C) \to \QQ)$, starting with a brief analysis of the relations in $K^2$.  If $\gamma = \gamma' \sqcup \gamma''$ where $E(\gamma') \neq \emptyset \neq E(\gamma'')$, then $[\gamma,<] \in K^2$.  Therefore $[\gamma,<] \equiv 0$ when $\gamma$ has two edges in distinct components of $\gamma$.  In contrast, if $\gamma'' = p$ consists of a single vertex and no edges and $E(\gamma') \neq \emptyset$, then $[p][\gamma',<] - [\gamma',<] \in K^2$.  In other words, modulo $K^2$ we have the relation $[p \sqcup \gamma',<] \equiv [\gamma',<]$ which allows us to remove isolated vertices from any graph with at least one edge.  
    
    Modulo $K^2 \subset \Gr(C)$, most terms of $\partial_d([G,<])$ vanish, because deleting a bridge results in a disconnected graph.  The exception is ``bridges to nowhere'': if $v_0 \in V(G)$ has valence 1 then it is connected to the rest of $G$ by a single edge $e_0 \in E(G)$ and there is a term in $\partial_d([G,<])$ of the form $\pm [\{v_0\} \sqcup G',<'] = \pm [p][G',<']$ where $G' \subset G$ is the subgraph with $V(G') = V(G) \setminus \{v_0\}$ and $E(G') = E(G) \setminus \{e_0\}$.  Assuming $E(G') \neq \emptyset$, we have $[p][(G',<')] \equiv [(G',<')]$ modulo $K^2$ and this term does not vanish.  This analysis shows that $\partial_d$ induces a homomorphism $\mathrm{Indec}(\Gr(C))_n \to \mathrm{Indec}(\Gr(C))_{n-1}$ for $n \geq 2$ given as alternating sum of deleting such bridges-to-nowhere together with their 1-valent vertex.  These terms in $\partial_d([G,<])$ precisely match with those terms in $\partial_c([G,<])$ in which the corresponding edge is collapsed, and we deduce that the boundary homomorphism on $\mathrm{Indec}(\Gr(C))$ induced by 
    $\partial = \partial_c - \partial_d$
    is given on representatives $[G,<]$ with $G$ connected and $n = |E(G)| \geq 2$ by the formula 
    \begin{equation*}
        \partial [G,<] = \sum_{e} (-1)^{\omega(e)} [G/e,<_{\vert E(G) \setminus \{e\}}],
    \end{equation*}
    where $e \in E(G)$ runs over all edges \emph{except} edges that are bridges-to-nowhere or tadpoles, and $\omega\col E(G) \to \{0, 1, \dots, n-1\}$ is an order-preserving bijection.  In the special case $n=1$ we get  
    $\partial([G,<]) = [p]-[p]^2 \equiv 1 - [p]$
    modulo $K^2$ when $G$ consists of two vertices connected by an edge, and $\partial([G,<]) = 0$ when $G$ is a tadpole.  Altogether, dualizing this homomorphism $\partial$ precisely matches with the differential $\delta \colon \mathsf{fGC}_{2,\mathrm{conn}}^\circlearrowleft \to \mathsf{fGC}_{2,\mathrm{conn}}^\circlearrowleft$, see \cite[Remark 3.3]{willwacher-kontsevich} and \cite[Figure 3]{willwacher-kontsevich}.

  The Lie cobracket on $\mathrm{Indec}(\Gr(C)) = K/K^2$ is obtained by anti-symmetrizing the Connes--Kreimer formula, restricting to $K = \mathrm{Ker}(\Gr(C) \to \QQ)$ and reducing modulo $K^2$. 
  We shall first consider the terms of~\eqref{eqn: graphcoproduct} which only involve graphs with at least one edge. 
  If $G$ is a connected graph with at least one edge, then the sum~\eqref{eqn: graphcoproduct} will likely contain terms in which $\gamma \subset G$ is not connected, and such terms must be reduced modulo $K^2$.  
  Since $G$ is connected, reducing the right-hand side of~\eqref{eqn: graphcoproduct} modulo $K^2 \otimes \Gr(C)$, and restricting attention to those terms in which $\gamma$ has at least one edge, lets us remove isolated vertices from $\gamma$ and lets us cancel terms in which $\gamma$ contains two edges in distinct components of $\gamma$.  After anti-symmetrizing and dualizing, the sum of these terms, in which $\emptyset \subsetneq \gamma\subsetneq E(G)$, gives a formula for the Lie bracket between any two elements of $\mathrm{Indec}(\Gr(C))^\vee$ of positive cohomological degree, and the formula matches with the Lie bracket on $\mathsf{fGC}^\circlearrowleft_{2,\mathrm{conn}}$ explained in \cite[\S2]{khoroshkin-willwacher-zivkovic-differentials} and implicit in \cite[\S3]{willwacher-kontsevich}, given by anti-symmetrized insertion.

  Now we inspect the terms in the antisymmetrization of~\eqref{eqn: graphcoproduct} with $E(\gamma) = \emptyset$ or $E(G/\gamma) = \emptyset$.  
  Writing $y = [G,<]$ with $E(G) \neq \emptyset$ and $v = |V(G)|$, the sum of the terms of the antisymmetrized coproduct of $y$ that involve graphs with no edges is
  \[([p]^v - [p]) \otimes y + y \otimes ([p]-[p]^v).\]
  Note $[p]^v - [p] \equiv (v\!-\!1)\cdot ([p]-1)$ modulo $K^2$.  Since $[p]-1 \in K$ reduces to a basis for the degree-zero part of $K/K^2 = \mathrm{Indec}(\Gr(C))$, we obtain a formula for the Lie bracket with the element $([p]-1)^\vee \in \mathrm{Indec}(\Gr(C))^\vee$: 
  for $E(G) \neq \emptyset$, it is given by $[G,<]^\vee \mapsto (v-1) [G,<]^\vee$ where $v = |V(G)|$ is the  number of vertices of $G$. 
    This again agrees with the ``anti-symmetrized insertion'' Lie bracket with $[p]^\vee \in \mathsf{fGC}_{2,\mathrm{conn}}^\circlearrowleft$.
\end{proof}

    \newcommand{\Ltwo}{L^{\mathrm{biv}}}  

\begin{proposition}\label{prop: connection to GC2}
    Let
    \begin{equation*}
        L^\vee = \mathrm{Indec}(\Gr(C))
    \end{equation*}
    be the graph complex obtained as the indecomposables in the differential graded bialgebra 
    $\Gr(C)$.   
    We have an isomorphism of bigraded Lie coalgebras 
    \begin{equation*}
        \mathrm{Indec}\big(\sseq{G}^1_{*,*}\big) \cong H_*(L^\vee) 
    \end{equation*}
    and a dual isomorphism of bigraded Lie algebras
    \begin{equation*}
        \mathrm{Prim}\big(\sseq{G}_1^{*,*}\big) \cong H^*(L),
    \end{equation*}    
    where $L = \mathrm{Indec}(\Gr(C))^\vee \cong \mathsf{fGC}^\circlearrowleft_{2,\mathrm{conn}}$.
    Moreover, there is a split extension of Lie algebras
    \begin{equation}
    \label{eq:semidirect-decomposition}
        H^*(\mathsf{GC}_2) \hookrightarrow H^*(L) \twoheadrightarrow \Ltwo
    \end{equation}
    where $\Ltwo$ 
    denotes the bigraded vector space which is $\QQ$ in bidegrees $(1,4k)$ for all $k\ge 0$ and is $0$ in all other bidegrees, equipped with the trivial Lie algebra structure.
\end{proposition}

\begin{proof} 
    We have already constructed a filtered quasi-isomorphism between $C$ and the singular chains of $\BKGr$, giving an isomorphism
    \begin{equation*}
        \sseq{G}^1_{*,*} \cong H_*(\Gr(C)).
    \end{equation*}
    To get to the first statement in the proposition, we would like to use an isomorphism of the form $\mathrm{Indec}(H_*(A)) \cong H_*(\mathrm{Indec}(A))$ when $A$ is the differential graded bialgebra 
    $\Gr(C)$.  Independently of the coproduct, there is a natural transformation
    \begin{equation*}
        \mathrm{Indec}(H_*(A)) \to H_*(\mathrm{Indec}(A))
    \end{equation*}
    defined for any differential graded algebra $A$ equipped with an augmentation $\epsilon \colon A \to \QQ$, and it is well known that this transformation is an isomorphism whenever
    $H_*(A)$ is a free graded-commutative algebra and $A$ is semifree.  
    This is proved by reducing to the case where $\partial_A = 0$.  (Recall that a CDGA is \emph{semifree} when the underlying graded-commutative algebra, obtained by forgetting the differential, is free.)  

    Now $\Gr(C)$ is certainly semifree: as generators we can take $[G,<]$ with $G$ connected, one such for each isomorphism class which does not possess automorphisms inducing an odd permutation of $E(G)$.  Unfortunately $H_*(\Gr(C))$ is not quite free, because of the relation $x^2 = x$ in degree 0, where $x = [p]$ is the class of a graph with one vertex and no edges, see Corollary~\ref{cor:bidegree-zero-zero}.  To get around this nuisance, we artificially add two algebra generators
    \begin{equation*}
        \Gr(C) \hookrightarrow A = \Gr(C) \otimes \QQ[y] \otimes \Lambda_\QQ[z],
    \end{equation*}
    with $\mathrm{deg}(y) = 0$ and $\mathrm{deg}(z)=1$, and with boundary map extended as $\partial y = 0$ and $\partial z = xy-1$.  We do not attempt to extend the chain level coproduct, but extend the augmentation as $\epsilon(y) = 1$.  Then the class $[x] \in H_0(\Gr(C))$ maps to a unit in $H_0(A)$ and the inclusion of $\Gr(C)$ induces an isomorphism from the localized homology algebra
    \begin{equation*}
        H_*(\Gr(C)) \twoheadrightarrow H_*(\Gr(C))[x^{-1}] \xrightarrow{\cong} H_*(A).
    \end{equation*}
    Since $[x] \in H_0(\Gr(C))$ is a unit modulo $K = \mathrm{Ker}(H_*(\Gr(C)) \to \QQ)$, inverting $x$ does not change the indecomposables $K/K^2$, so we have
    \begin{equation*}
        \mathrm{Indec}(H_*(\Gr(C))) \xrightarrow{\cong} \mathrm{Indec}(H_*(A)).
    \end{equation*}
    By inspection, we also see that the induced chain level map
    \begin{equation*}
        \mathrm{Indec}(\Gr(C)) \hookrightarrow \mathrm{Indec}(A)
    \end{equation*}
    is a quasi-isomorphism (on the chain level, the cokernel has two additive generators $z$ and $y-1$, but these are connected by a non-zero boundary map).
    
    Inverting $x$ turns the ring $H_0(\Gr(C)) = \QQ[x]/(x^2 -x)$ into $\QQ[x]/(x-1) \cong\QQ$, and the localization $H_*(\Gr(C))[x^{-1}]$  inherits a coproduct, making it a \emph{connected} commutative Hopf algebra.  By a classical theorem of Leray (see \cite[Theorem 7.5]{MR0174052}), this implies that the underlying graded algebra is isomorphic to a free graded-commutative algebra, so we finally deduce a commutative diagram of isomorphisms
    \begin{equation*}
        \begin{tikzcd}
        \mathrm{Indec}(\sseq{G}^1_{*,*}) = \mathrm{Indec}(H_*(\Gr(C))) \dar["\cong"]\rar & H_*(\mathrm{Indec}(\Gr(C))) = H_*(L^\vee)\dar["\cong"] \\
        \mathrm{Indec}(H_*(A)) \rar["\cong"] & 
        H_*(\mathrm{Indec}(A)),
        \end{tikzcd}
    \end{equation*}
    since $A$ is semifree and $H_*(A)$ is free.  The formula $\mathrm{Prim}(\sseq{G}_1^{*,*}) \cong H^*(L)$ is obtained by dualizing.
        
Finally,
Willwacher shows \cite[Proposition 3.4]{willwacher-kontsevich} that the cohomology of $L \cong \mathsf{fGC}^\circlearrowleft_{2,\mathrm{conn}}$ splits as a vector space as
\[H^*(L) \cong H^*(\GC_2) \oplus \Ltwo,\]
where the second summand is spanned by ``loop'' graphs $L_{4k+1}$, $k \geq 0$, with $4k+1$ edges connecting $4k+1$ bivalent vertices in a circular pattern (or rather, the dual basis vectors to such graphs, spanning a cochain complex with trivial differential, so $\Ltwo \cong H^*(\Ltwo)$).
The inclusion $H^*(\GC_2) \hookrightarrow H^*(L)$ is a Lie algebra homomorphism by comparing definitions.  For degree reasons, this additive splitting makes 
\begin{equation*}
    H^*(\GC_2) \hookrightarrow H^*(L) \twoheadrightarrow \Ltwo
\end{equation*}
into a split extension of Lie algebras, and $\Ltwo$ must have trivial Lie bracket.
\end{proof}

The map of filtered spaces $\MGr \simeq_\QQ \BKGr \to BK(\ZZ)$ induces a map of spectral sequences of bialgebras, of the form
\begin{equation}\label{eq:8}
  \sseq{G}^r_{*,*} \to \sseq{Q}^r_{*,*}
\end{equation}
abutting to the induced map
\begin{equation*}
      H_*(\MGr) \to H_*(BK(\ZZ)).
\end{equation*}
When passing to primitives in the linear duals of the respective $E^1$-pages, it therefore induces a map of Lie algebras
\begin{equation*}
    \mathrm{Prim}\big(\sseq{Q}_1^{*,*}\big) \to \mathrm{Prim}\big(\sseq{G}_1^{*,*}\big) \cong H^*(L) \cong \Ltwo \oplus H^*(\mathsf{GC}_2)
\end{equation*}
from the primitives in the $E_1$-page of the (cohomological) Quillen spectral sequence to the semidirect product corresponding to the split extension~\eqref{eq:semidirect-decomposition}. 
The bidegrees are such that primitives in $\sseq{Q}_1^{s,t}$ are mapped to graphs with first Betti number $s$ and $s+t$ many edges.
Restricted to diagonal bidegrees, this can be identified with a map of Lie algebras
\begin{equation}\label{eq:prim-Quillen to grt}
    \mathrm{Prim}\Big(\bigoplus_{g}\sseq{Q}_1^{g,g}\Big) \to \mathrm{Prim}\Big(\bigoplus_{g}\sseq{G}_1^{g,g}\Big) \cong H^0(\mathsf{GC}_2) \cong \mathfrak{grt},
\end{equation}
with the last isomorphism following from \cite[Theorem 1.1]{willwacher-kontsevich}.  Here $\grt$ is the Grothendieck--Teichm\"uller Lie algebra (denoted $\grt_1$ in \emph{op.~cit.}).

\subsubsection{Relationship with the double complex of \cite{khoroshkin-willwacher-zivkovic-differentials}} 
The full differential on $\mathrm{Indec}(C)^\vee$ induced by dualizing $\partial$ can be related to a certain double complex considered in \cite{khoroshkin-willwacher-zivkovic-differentials}.  Let us point out a notational difference: in the paper  \cite{willwacher-kontsevich}, Willwacher uses a superscript $\circlearrowleft$ on graph complexes to denote that tadpoles are \emph{allowed} and the lack of such a superscript means that they are disallowed, while \cite{khoroshkin-willwacher-zivkovic-differentials} uses a superscript $\notadp$ to denote that tadpoles are \emph{disallowed} and the lack of superscript means they are allowed.  We follow the notation of the former.
\begin{proposition}\label{prop:comparison-map-to-KWZ-complex}
    Let $\nabla \colon \mathsf{fGC}^\circlearrowleft_{2,\mathrm{conn}} \to \mathsf{fGC}^\circlearrowleft_{2,\mathrm{conn}}$ be the differential given by Lie bracket with a tadpole graph, as in \cite{khoroshkin-willwacher-zivkovic-differentials}.  Then the composition~\eqref{eq:compare-to-fgcc} induces an isomorphism of differential graded Lie algebras
    \begin{equation*}
    (\mathsf{fGC}^\circlearrowleft_{2, \mathrm{conn}},\delta + \nabla) \xrightarrow{\cong} \mathrm{Indec}(C)^\vee.
    \end{equation*}
\end{proposition}
Let us also remark that $[G,<] \mapsto (-1)^{|V(G)|}[G,<]$ defines an isomorphism of complexes $(\mathsf{fGC}^\circlearrowleft_{2, \mathrm{conn}},\delta + \nabla) \to (\mathsf{fGC}^\circlearrowleft_{2, \mathrm{conn}},\delta - \nabla)$, so it does not matter much if this proposition is stated with $\delta + \nabla$ or $\delta - \nabla$.
\begin{proof}[Proof sketch]
    We have already seen that the map is an isomorphism of vector spaces, and the Lie brackets are compared in the same way as Lemma~\ref{lemma:compare-to-fgcc}.  
    
    It remains to verify that the differentials match up.  If $G$ is a connected graph containing tadpoles (edges connecting a vertex to itself), then both $\partial_d([G,<])$ and $\partial_c([G,<])$ contain terms involving deleting (equivalently, collapsing) those tadpoles, and such terms cancel in the difference 
    $(\partial_c - \partial_d)([G,<])$.
    Similarly, if $G$ contains bridges-to-nowhere (edges connecting a 1-valent vertex to the rest of $G$) then deleting such edges agrees with collapsing them modulo $K^2 \subset C$, so such terms also cancel in 
    $(\partial_c - \partial_d)([G,<]) \in \mathrm{Indec}(C) = K/K^2$
    where $K = \mathrm{Ker}(C \to \QQ)$.  Dualizing the signed sum of collapsing edges which are neither tadpoles nor bridges-to-nowhere precisely matches with the differential $\delta$ on $\mathsf{fGC}^\circlearrowleft_{2,\mathrm{conn}}$, and dualizing the signed sum of deleting such edges matches with $\nabla$.
\end{proof}

We can then use the results of \cite{khoroshkin-willwacher-zivkovic-differentials} to deduce results about differentials in our spectral sequences.
\begin{proposition}
    The homology group
    \[H_*(\mathrm{Indec}(C))\]
    with respect to the boundary map induced by 
    $\partial = \partial_c-\partial_d$
    is 1-dimensional in homological degree $1$, generated by the class $[T]$ of a tadpole graph, and is zero in all other degrees.
\end{proposition}
\begin{proof}
    We consider the composition
    \[
    (\mathsf{fGC}_{2,\mathrm{conn}},\delta+\nabla) \hookrightarrow (\mathsf{fGC}_{2,\mathrm{conn}} \oplus \QQ T^\vee,\delta+\nabla) \stackrel{\simeq}{\hookrightarrow} (\mathsf{fGC}^\circlearrowleft_{2,\mathrm{conn}},\delta+\nabla) \cong \mathrm{Indec}(C)^\vee,
    \]
    where the last isomorphism is by Proposition~\ref{prop:comparison-map-to-KWZ-complex}.      
We justify the middle quasi-isomorphism as follows: by \cite[Proposition 3.4]{willwacher-kontsevich}, the inclusion
\[(\mathsf{fGC}_{2,\mathrm{conn}} \oplus \QQ T^\vee, \delta) \xra{\sim} (\mathsf{fGC}^\circlearrowleft_{2,\mathrm{conn}},\delta)\]
is a quasi-isomorphism,
and hence, by a spectral sequence argument (filtering by number of vertices),  it follows that \[(\mathsf{fGC}_{2,\mathrm{conn}} \oplus \QQ T^\vee, \delta+\nabla) \xra{\sim} (\mathsf{fGC}^\circlearrowleft_{2,\mathrm{conn}},\delta+\nabla)\]
is also a quasi-isomorphism.  Finally, \cite[\S3.2, Corollary 4]{khoroshkin-willwacher-zivkovic-differentials} shows that $(\mathsf{fGC}_{2,\mathrm{conn}},\delta+\nabla)$ is acyclic, and the claim follows.
\end{proof}

\begin{corollary}
    Let
    \begin{equation*}
        h \colon \sseq{G}^r_{s,t} \to \sseq{Q}^r_{s,t}
    \end{equation*}
    be the map of spectral sequences induced by $G \mapsto H_1(G;\ZZ)$.  For odd $g \geq 3$, the image
    \begin{equation*}
        h([W_g]) \in \mathrm{Indec}(\sseq{Q}^1_{g,g})
    \end{equation*}
    is a permanent cycle in the spectral sequence obtained by taking indecomposables of the Quillen spectral sequence.
\end{corollary}
\begin{proof}
    It follows from \cite[Proposition 5]{khoroshkin-willwacher-zivkovic-differentials} that the classes $[W_g] \in \mathrm{Indec}(\sseq{G}^1_{g,g})$ survive to represent elements on the $E^{g-1}$-page, where a non-zero differential takes the class of $[W_g]$ to the loop class $L_{2g-1} \in \sseq{G}^{g-1}_{1,2g-2}$.  Therefore the image of $[W_g]$ in $\mathrm{Indec}(\sseq{Q}^1_{g,g})$ also survives to represent an element of $\mathrm{Indec}(\sseq{Q}^{g-1}_{g,g})$ which the differential $d^{g-1}$ takes to an element of bidegree $(1,2g-2)$.  But the Quillen spectral sequence is zero in this bidegree, so there are no possible non-zero differentials out of $h([W_g]) \in \mathrm{Indec}(\sseq{Q}^r_{g,g})$ for any $r$.
\end{proof}

As we will explain in the next section, the integration pairing with $\omega^{2g-1}$ can be used to deduce that $h([W_g]) \neq 0 \in \mathrm{Indec}(\sseq{Q}^1_{g,g})$.  At present we cannot rule out that this class $h([W_g])$ in the corollary above could be in the \emph{image} of a non-zero differential on some page of the spectral sequence, although this would be impossible assuming the conjecture of \cite[\S2]{church-farb-putman-stability} which implies that $\sseq{Q}^1_{s,t} = 0$ for $s > t$ and $(s,t) \neq (1,0)$.  Subject to that conjecture, we can therefore deduce that $h([W_g])$ survives to represent non-zero elements of $\mathrm{Indec}(H_*(BK(\ZZ);\QQ)) = K_{*-1}(\ZZ) \otimes \QQ$ in degrees $\ast \in \{6, 10, 14, \dots\}$.

\subsection{Inverting the class of a point}\label{sec:inverting-the-class-of-a-point} 
We have explained why the spectral sequence $\sseq{G}^r_{*,*}$ comes with a product and coproduct, satisfying a Leibniz rule on each page. Furthermore, it is easy to see that there are unit and counit morphisms $\QQ \to \sseq{G}^1_{0,0} \to \QQ$ induced by filtered maps $\{\ast\} \to \MGr \to \{\ast\}$, making this a spectral sequence of bialgebras.  As it stands, the $E^1$-page does not admit an antipode though.  In this subsection, which is not strictly necessary for the rest of the paper, we investigate the ring $\sseq{G}^1_{0,0} \cong \sseq{G}^r_{0,0}$ and its action on $\sseq{G}^r_{*,*}$ by multiplication.
\begin{lemma}\label{lemma:filtration-zero-homotopy-type}
    The space $\Fil_0\MGr$ consists of two path components, both of which have trivial rational homology.  Consequently, $\sseq{G}^1_{0,0} \cong \QQ^2$ and $\sseq{G}^1_{0,t} = 0$ for $t > 0$.
\end{lemma}
\begin{proof}
    We may instead calculate the homology of the chain complex $\Fil_0 C \subset C$ whose generators $[G,\omega]$ have $b_1(G) = 0$.  In other words, $G$ is a forest.  $H_0(\Fil_0 C)$ may be calculated by hand: if we write $p^n$ for the graph with no edges and 
    $n\ge 0$ many vertices, then these form a basis for $\Fil_0 C_0$.  There is a unique isomorphism class of 1-edge graphs with $n$ many vertices and $b_1 = 0$ for each $n \geq 2$, and its boundary is 
    $p^{n-1} - p^n$
    This shows that the space $\Fil_0\MGr$ 
    is the disjoint union of precisely two path components, one corresponding to the empty graph and one corresponding to all non-empty graphs.

    To see that $H_j(\Fil_0C) = 0$ for $j > 0$ we consider the homomorphism $T\col \Fil_0C_n \to \Fil_0C_{n+1}$ given by
    \begin{equation}
        T([G,\omega]) = \sum_{G = G_1 \vee G_2} [G_1 \vee e \vee G_2],        
    \end{equation}
    where the sum is indexed by all possible ways of writing $G = G_1 \vee G_2$ as the wedge sum of two graphs $G_1$ and $G_2$ at basepoints $v_1 \in V(G_1)$ and $v_2 \in V(G_2)$, with $E(G_1) \neq \emptyset \neq E(G_2)$.   The graph $G_1 \vee e \vee G_2$ is obtained from the disjoint union $G_1 \sqcup G_2$ by inserting a new edge $e$ connecting $v_1$ to $v_2$, and $E(G_1 \vee e \vee G_2) = \{e\} \sqcup E(G_1) \sqcup E(G_2) = \{e\} \sqcup E(G)$ is ordered as the concatenation $\{e\} \ast (E(G),<)$.  Then $T$ is a chain homotopy from the identity to the chain map $S \colon \Fil_0C \to \Fil_0C$ defined by
    \begin{equation*}
        S([G,\omega]) = (\partial \circ T + T \circ \partial - \mathrm{id})([G,\omega]).
    \end{equation*}
    
    If $G$ is a forest all of whose trees have at most $n$ edges for $n \geq 2$, then $S([G,\omega])$ is a linear combination of forests all of whose trees have at most $n-1$ edges.  In other words, $S$ is a chain homotopy from the identity map to a chain map which strictly decreases the maximal number of edges in trees for any tree with more than 1 edge.  We now consider a map of CDGAs
    \begin{equation*}
        \QQ[p,e] \to \Fil_0C
    \end{equation*}
    whose domain is the free graded-commutative algebra on a generator $p$ of degree 0 and a generator $e$ of degree 1, with boundary defined by $\partial p = 0$ and $\partial e = p - p^2$. 
    The map is defined by sending $p$ to a one-vertex graph and $e$ to a graph with two vertices connected by an edge. 
    Iterating the chain homotopy $S$ shows that any cycle in $\Fil_0C$ is homologous to a cycle in the image of this map of CDGAs.  Finally, an easy calculation shows that $H_j(\QQ[p,e],\partial) = 0$ for $j > 0$.      
    See also \cite[Proposition 3.4]{willwacher-kontsevich} for a related result.
\end{proof}
\begin{cor} \label{cor:bidegree-zero-zero}
    Let $x \in \sseq{G}^1_{0,0} = H_0(\Fil_0\MGr;\QQ)$ be the path component corresponding to the non-empty graphs.  Then $\sseq{G}^1_{0,0} = \sseq{G}^\infty_{0,0} \cong \QQ[x]/(x^2 - x)$ as a ring, with coproduct given by $\Delta(x) = x \otimes x$ and counit by $x \mapsto 1$.
\end{cor}
As already explained, this bialgebra represents the functor $R \mapsto \{x \in R \mid x^2 = x\}$, which is a monoid scheme but not a group scheme, so the bialgebra does not admit an antipode.
\begin{proof}
    We have already seen that the ring structure is as stated, with $x = [p]$ in the notation of the above proof.  The equation $\Delta(x) = x \otimes x$ follows from the fact that the coproduct on the $E^\infty$-page is induced by the space-level diagonal map.  Similarly for the counit.
\end{proof}

The ring structure on $E^1_{0,0} \cong E^r_{0,0}$ implies a splitting $E^r_{s,t} \cong x E^r_{s,t} \oplus (1-x)E^r_{s,t}$ for all bidegrees $(s,t)$.  The following lemma implies that the second summand is trivial for $(s,t) \neq (0,0)$ and that multiplication by $x$ acts as the identity in such bidegrees.  Formally inverting $x$  therefore changes only bidegree $(0,0)$, and by exactness of localization results in a spectral sequence, 
which is now a spectral sequence 
of connected Hopf algebras.

\begin{lemma} 
    Let $x \in \sseq{G}^1_{0,0} = \sseq{G}^r_{0,0}$ be as in Corollary~\ref{cor:bidegree-zero-zero}.  Then multiplication by $x \colon \sseq{G}^r_{s,t} \to \sseq{G}^r_{s,t}$ is the identity homomorphism for $(s,t) \neq (0,0)$ and all $r \geq 1$.
\end{lemma}
\begin{proof} 
  It suffices to prove the statement for $r = 1$ where $\sseq{G}^1_{s,t} = H_{s+t}(\Gr_s(C))$, since each isomorphism $\sseq{G}^{r+1} \cong H_*(\sseq{G}^r,d^r)$ is an isomorphism of algebras, indeed bialgebras (Corollary~\ref{cor:graph-spectral-sequence}). 
  The map $\sseq{G}^1_{s,t} \to \sseq{G}^1_{s,t}$ which multiplies by $x$ has 1-dimensional kernel and cokernel for $(s,t) = (0,0)$, we wish to show that the kernel and cokernel vanishes for $s + t > 0$.  To see this we look at the quotient complex $\Gr(C)/(x)$ by the ideal $(x) \subset \Gr(C)$, whose homology sits in a long exact sequence 
  associated to the short exact sequence
  \begin{equation}\label{eq:Gr(C)/(x)}0\to\Gr(C)\xra{\cdot x}\Gr(C)\to\Gr(C)/(x)\to 0.\end{equation}
  
  An additive generator $[G,<] \in \Gr(C)$ is in $(x)$ if and only if $G$ contains a vertex of valence 0, while those basis elements $[G,<] \in \Gr(C)$ for which all vertices have positive valence reduce to basis elements for $\Gr(C)/(x)$.  In the latter case, we shall use the same notation $[G,<]$ for the class modulo $(x)$.  There are two evident cycles in $\Gr(C)/(x)$, namely $1 \in \Gr(C)/(x)$ represented by the empty graph, and $[I,<]$ represented by the ``interval'' graph with two 1-valent vertices connected by a single edge. 
  Both represent non-zero homology classes, and note that the image of the interval graph under the connecting homomorphism $H_1(\Gr(C)/(x))\to H_0(\Gr(C)) \cong \QQ[x]/(x^2 - x)$ is $1-x \ne 0$.
  To prove the Lemma, it will suffice to show that $H_*(\Gr(C)/(x))$ is 1-dimensional in homological degrees 0 and 1, spanned by these two classes, and is zero-dimensional in all other degrees.  Indeed, by inspecting the long exact sequence associated to~\eqref{eq:Gr(C)/(x)}, one sees that multiplication by $x$ is an isomorphism on $H_{s+t}(\Gr(C))$ for $s + t > 0$, 
  but it must then be the identity since we also know that $x \in \sseq{G}^1_{0,0}$ is idempotent.

  To show that $H_*(\Gr(C)/(x))$ is spanned by the homology classes of $1$ and $[I,<]$, we follow a similar strategy to the proof of Lemma~\ref{lemma:filtration-zero-homotopy-type}, but using a different homotopy.  
  For an additive generator $[G,<] \in \Gr(C)/(x)$, i.e., where $G$ has no zero-valent vertices, we introduce the notation $\ov V(G) \subset V(G)$ for the set of vertices of valence $\geq 2$ and $\ov E(G) \subset E(G)$ for the set of edges which connect two distinct elements of $\ov V(G)$.  Define a linear homomorphism
  \begin{equation*}
    (\Gr(C)/(x))_n \xrightarrow{T} (\Gr(C)/(x))_{n+1}
  \end{equation*}
  by the formula
  \begin{equation*}
    T[G,<] = \sum_{v \in \ov V(G)} [I \vee_v G,<']
  \end{equation*}
  where $G$ is a graph with no vertices of valence 0 and $I \vee_v G$ denotes the graph obtained from $G$ by adding a new vertex $v_0$ and connecting it to $v \in G$ by an edge $e_0$, so $V(I \vee_v G) = \{v_0\} \sqcup V(G)$ and $E(I\vee_v G) = \{e_0\} \sqcup E(G)$.  The total order $<'$ is obtained from $<$ by declaring $e_0 <' e$ for $e \in E(G)$ and agreeing with $<$ on elements of $E(G) \subset E(I \vee_v G)$.  Then one verifies that 
  \begin{equation*}
    (T \partial_c + \partial_c T)([G,<]) = |\ov V (G)| [G,<] + \sum_{e \in \ov E(G)} \pm [I \vee_{e/e} (G/e),<']
  \end{equation*}
  where $e/e \in V(G/e)$ is the vertex arising at the collapsed edge.  We shall not need a precise recipe for the sign and the induced ordering $<'$, but note the inequality 
  \begin{equation}\label{eq:decreases-overline-E-1}
    |\ov V(I \vee_{e/e} (G/e))| < |\ov V(G)|.
  \end{equation}
  By a similar argument
  \begin{equation*}
    (T \partial_d + \partial_d T)([G,<]) = \sum_{(e,v)} \pm [I \vee_v (G \setminus e),<'],
  \end{equation*}
  where the sum runs over pairs consisting of an $e \in \ov E(G)$ which is a bridge in $G$, and a 2-valent vertex $v \in V(G)$ adjacent to $e$.  Here $I \vee_v (G \setminus e)$ denotes the result of removing $e$ and attaching a new edge at $v$ whose other endpoint is a new 1-valent vertex.  Here we note the inequalities
  \begin{equation}\label{eq:decreases-overline-E-2}
  \begin{aligned}
    |\ov V(I \vee_v (G \setminus e)| & \leq  |\ov V(G)|,\\
        |\ov E(I \vee_{v} (G\setminus e))| & < |\ov E(G)|.
\end{aligned}
  \end{equation}

  Now, for integers $a,b \geq 0$, let us write $D_a \subset \Gr(C)/(x)$ for the rational span of those $[G,<]$ for which $|\ov V(G)| \leq a$ and  $D_{a,b} \subset D_a$ for the rational span of those $[G,<]$ for which either $|\ov V(G)| <a$, or $|\ov V(G)| = a$ and $|\ov E(G)| \leq b$.  For $a > 0$ let $S_a\colon \Gr(C)/(x) \to \Gr(C)/(x)$ be the chain map
  \begin{equation*}
    S_a = \mathrm{id} - \tfrac1a (T\partial + \partial T).
  \end{equation*}
  The inequalities~\eqref{eq:decreases-overline-E-1} and~\eqref{eq:decreases-overline-E-2} imply that 
  \begin{align*}
      S_a(D_{a,b}) & \subset D_{a,b-1}\\
      S_a(D_{a,0}) & \subset D_{a-1},
  \end{align*}
  and therefore, if $c \in D_a$ is any chain and $a > 0$, there exists some integer $N \geq 0$ such that $S_a^N c \in D_{a-1}$.  We can regard the linear map $\frac1a T$ as a chain homotopy between the identity and $S_a$, and by induction we deduce that any chain $c \in \Gr(C)/(x)$ is homologous to some $c' \in D_0$.  
  
  Graphs $G$ with $\ov V(G) = \emptyset$ are easy to classify: each connected component must be isomorphic to the ``interval'' $I$, the graph consisting of two vertices connected by a single edge.  If $G$ has more than two such components, then $[G,<] = 0$ because $\mathrm{Aut}(G)$ contains an odd permutation.  We have therefore shown that the inclusion
  \begin{equation*}
      \QQ.1 \oplus \QQ.[I,<] \cong D_0 \hookrightarrow \Gr(C)/(x)
  \end{equation*}
  induces a surjection on homology.  
\end{proof}

\begin{corollary}
    The rational homology of $\MGr$ is
    \begin{equation*}
        H_n(\MGr;\QQ) \cong
        \begin{cases}
            \QQ^2 & \text{for $n = 0$}\\
            \QQ & \text{for $n=1$}\\
            0 & \text{for $n > 1$}.
        \end{cases}
    \end{equation*}
\end{corollary}
In other words, $\MGr$ and hence $\BKGr$ are rationally equivalent to the disjoint union of a point, corresponding to $G = \emptyset$, and a circle.  It seems interesting to understand these homotopy types integrally.
\begin{proof}
    We have seen splittings $\sseq{G}^r_{s,t} = x E^r_{s,t} \oplus (1-x) E^r_{s,t}$, in which the second summand is a 1-dimensional rational vector space for $(s,t) = (0,0)$ and vanishes in all other bidegrees.  The first summand is
    \begin{equation*}
      x E^r_{s,t} \cong x^{-1} E^r_{s,t},
    \end{equation*}
    which form the pages of a spectral sequence of connected Hopf algebras.  In Subsection~\ref{sec:comparison-with-fGC}, 
    we identified $\mathrm{Indec}(x^{-1} E^r_{s,t}) \cong \mathrm{Indec}(\sseq{G}^r_{s,t})$ with the pages of the spectral sequence in \cite{khoroshkin-willwacher-zivkovic-differentials}, and for $r = \infty$ we in particular get that this is 1-dimensional in bidegree $(1,0)$ and vanishes otherwise.  Since a connected commutative Hopf algebra is determined by its indecomposables, we deduce that $x^{-1} E^\infty_{*,*}$ is an exterior algebra generated by one class in bidegree $(1,0)$, and hence
    \begin{equation*}
        xE^\infty_{s,t} \cong 
        \begin{cases}
            \QQ & \text{for $(s,t) = (0,0)$ and $(s,t) = (1,0)$}\\
            0 & \text{otherwise.}
        \end{cases}
    \end{equation*}
    Combining with $(1-x)E^\infty_{0,0}$ gives the result.
\end{proof}

\section{Freeness of a Lie algebra generated by \texorpdfstring{$\omega^{4k+1}$}{w4k+1} classes} \label{sec:freeness}

We now prove that the images of the elements $\omega^5, \omega^9, \ldots, \omega^{45}$ generate a free Lie subalgebra in $\Prim(W_0H^*_c(\cA)\otimes \RR).$
We will make use of the map~\eqref{eq:prim-Quillen to grt} from the primitives in the Quillen spectral sequence (along diagonal bidegrees) to $H^0(\mathsf{GC}_2)$ constructed in the previous section, as well as the following results from the literature:
\begin{enumerate} 
\item  Willwacher's theorem  \cite[Theorem 1.1]{willwacher-kontsevich} which gives an isomorphism of graded Lie algebras between $H^0(\GC_2) $  and the Grothendieck--Teichm\"uller Lie algebra, which we shall denote by $\grt$ (denoted by $\grt_1$ in \emph{loc.~cit.}).
  \item An injective map   \cite{brown-mixed}  from the motivic Lie algebra $\gm \to \grt$, where $\gm$ is isomorphic to the free graded Lie algebra on certain generators $\sigma_{2k+1}$ in every degree $2k+1\geq 3$.  These generators $\sigma_{2k+1}$ are canonical modulo Lie words in generators $\sigma_{2j+1}$ with $j < k$. 
  \item\label{it:four}   The integration pairing between $\omega^{2g-1} \in \sseq{Q}^{g,g}_1\otimes \RR$ (see~\eqref{eqn: Omegacinjects}) and the locally-finite homology class of the wheel $[W_g]$ is non-zero for $g>1$ odd \cite{brown-schnetz}.  In other words, the image of $\omega^{2g-1}$ in $H^0(\GC_2)\otimes \RR$, under the map~\eqref{eq:prim-Quillen to grt}, pairs non-trivially with $[W_g]\in H_0(\GC_2^\vee)$.  
  \end{enumerate}

This last point~\eqref{it:four} relies on the discussion of subsection~\ref{sec: Loc-sym-Quillen}, in the following way.  The differential form $\omega^{2g-1}$ most naturally gives a class in the $E_1$-page of the cohomological \emph{tropical spectral sequence}, while the wheel graph $W_g$ gives a class in the homological \emph{Quillen spectral sequence}.  In order to make sense of the pairing, we use the explicit comparison discussed in \S\ref{sec: Loc-sym-Quillen}, which we recall involved a zig-zag
\begin{equation*}
    \Gr_g(BK(\ZZ)) \xleftarrow{\simeq} |N_\bu T_\bu(\QQ^g)\cup \{\infty\}| \xrightarrow{\simeq_\QQ} (P_g/\GL_g(\ZZ)) \cup \{\infty\}.
\end{equation*}
We use this to make sense of pairing a reduced homology class in $\Gr_g(BK(\ZZ))$ with a reduced cohomology class in $(P_g/\GL_g(\ZZ)) \cup \{\infty\}$, or equivalently a compactly supported cohomology class in $P_g/\GL_g(\ZZ)$.
To implement the comparison, we will first factor the second arrow through a space $|\ov T_\bu(\QQ^g) \cup \{\infty\}|$ which we define now.  Let $\ov T_p(\QQ^g)$ be the quotient of the set $N_0 T_p(\QQ^g)$ by the equivalence relation that $(A_p \subset \dots \subset A_0,<) \sim (A'_p \subset \dots \subset A'_0,<')$ provided the resulting maps
\begin{equation*}
    \Delta^p \setminus \partial \Delta^p \to P_g,
\end{equation*}
given by the formula~\eqref{eq:map-simplex-to-g} are in the same orbit for the action of $\GL_g(\ZZ)$-action on the set of such maps.  We point out that the formula~\eqref{eq:map-simplex-to-g} for the map associated to $a = (A_p \subset \dots \subset A_0,<)$ does not involve the total order $<$ at all, and that the elements $\psi \in A_0$ enter the formula only through their squares $v \mapsto (\psi(v))^2$.  Therefore the element of $\ov T_p(\QQ^g)$ represented by $a$ is invariant under any change of $<$, under replacing some $\psi$ by $-\psi$, and by definition also under precomposing all $\psi$'s with the same $X \in \GL_g(\ZZ)$.

The definition of the equivalence relation defining $\ov T_p(\QQ^g)$ ensures that the map from Proposition~\ref{prop:bottom-horizontal} factors as
\begin{equation*}
    \Gr_g(BK(\ZZ)) \xleftarrow{\simeq} |N_\bu T_\bu(\QQ^g)\cup \{\infty\}| \to |\ov T_\bu(\QQ^g) \cup \{\infty\}| \to (P_g/\GL_g(\ZZ)) \cup \{\infty\}.
\end{equation*}
We may also define a map of simplicial sets
\begin{equation*}
    \Gr_g(\cF_\bu) \xrightarrow{f} \ov T_\bu(\QQ^g) \cup\{\infty\}
\end{equation*}
by sending the isomorphism class of a flag $x = (G_{0,0} \hookrightarrow \dots \hookrightarrow G_{0,p}) \in \Fil_g \cF_p$ to the element of $T_p(\QQ^g) \cup\{\infty\}$ given as $f(x) = \infty$ if $b_1(G_{0,p}) < g$ and otherwise by first choosing an isomorphism $H_1(G_{0,p};\ZZ) \cong \ZZ^g$ and then letting 
\begin{equation*}
    f(x) = [(A_p \subset \dots \subset A_0,<)] \in \ov T_p(\QQ^g),
\end{equation*}
where $A_0 \subset H^1(G_{0,p};\QQ) \cong (\QQ^g)^\vee$ is the subset defined by first choosing a section of $H(G_{0,p}) \twoheadrightarrow E(G_{0,p})$ and then letting $A_0$ consist of all non-zero elements in the image of the resulting composition
\begin{equation*}
    E(G_{0,p}) \hookrightarrow H(G_{0,p}) \hookrightarrow C^1(G;\QQ) \twoheadrightarrow H^1(G;\QQ) \cong \QQ^g.
\end{equation*}
We remark that an edge is sent to zero precisely when it is a bridge.
Let $A_i \subset A_0$ consist of the non-zero vectors in the image of the subset $E(G_{0,p-i}) \subset E(G_{0,p})$ for $i = 0, \dots, p$ and choose an arbitrary total ordering $<$ of $A_0$ to define the element $f(x)$.  (The resulting element $f(x)$ does not depend on the choice of $<$, nor on the choice of representative half-edges of each edge, by definition of the equivalence relation defining $\ov T_\bu(\QQ^g)$.)  These maps fit into a homotopy commutative diagram
\begin{equation*}
    \begin{tikzcd}
        {\Gr_g(\BKGr)} \rar[,"H_1"]\arrow[dd,"\simeq_\QQ"] & {\Gr_g(BK(\ZZ))} & \\
        & {|N_\bu T_\bu(\QQ^g) \cup \{\infty\}|} \arrow[u,"\simeq"] \dar \arrow[dr,"\simeq_\QQ"] &\\
        {\Gr_g(\MGr)} \rar["f"] & {|\ov T_\bu(\QQ^g) \cup \{\infty\}|} \rar & {(P_g/\GL_g(\ZZ)) \cup \{\infty\}.}
    \end{tikzcd}
\end{equation*}

Now, a graph $G$ with $b_1(G) = g$ defines a map
\begin{equation*}
    (\Delta^1_\bu)^{E(G)} \to \Gr_g(\ov\cF_\bu)
\end{equation*}
whose geometric realization is one of the cells in the rational cell structure leading to the cellular chain complex $\Gr_g(C)$ calculating the reduced homology of $\Gr_g(\MGr)$ which is the column $\sseq{G}^1_{g,*}$ in the graph spectral sequence.  Composing  with the bottom row of the diagram above, we obtain a map
\begin{equation*}
    (\Delta^1)^{E(G)} \to (P_g/\GL_g(\ZZ)) \cup \{\infty\}.
\end{equation*}
Reparametrizing using the inverse $\ell^{-1}\col (0,\infty) \to \Delta^1 \setminus \partial \Delta^1$ of the diffeomorphism~\eqref{eq:diffeo}, the composition
\begin{equation*}
    (0,\infty)^{E(G)} \xrightarrow{(\ell^{-1})^{E(G)}}(\Delta^1)^{E(G)} \to (P_g/\GL_g(\ZZ)) \cup \{\infty\}
\end{equation*}
is independent of choice of $\ell$.  Tracing through definitions, it may be identified with the \emph{graph Laplacian} $(0,\infty)^{E(G)} \to P_g$ from \cite{brown-schnetz}, which depends on a choice of basis of $H_1(G;\ZZ)$, composed with the quotient map $P_g \to P_g/\GL_g(\ZZ)$.

For $g = 2k+1$, integrating the invariant differential form $\omega^{4k+1}$ along this map for each $G$ with $|E(G)| = 4k+2$ defines a element of degree $4k+2$ in the linear dual cochain complex
\begin{equation*}
    \omega^{4k+1} \in \Gr_g(C)^\vee = \Hom(\Gr_g(C),\QQ).
\end{equation*}
This element is a cocycle and pairs with the cycle $[W_{2k+1}] \in \Gr_g(C)$ by evaluating the integral from \cite{brown-schnetz}.  By the discussion above, this is the desired pairing.

\subsection{Proof of Theorem~\ref{thm:free}} \label{sec:proof3}

  \begin{lemma} \label{lem: Wheelsprimitive} For all $g>1$ odd, the wheel classes $[W_{g}] \in \sseq{G}^1_{g,g}$ are primitive, and hence annihilated by the Lie cobracket.  Consequently, the wheel classes in $H_0(\GC_2^\vee)$ define linear maps
  \[ [W_g] \colon \grt \to \QQ\]
  which vanish on  commutators $[\grt, \grt]$, and hence descend to maps $\grt^{\ab}\to \QQ$, which are nonzero by  \eqref{it:four}.  
  \end{lemma}

\begin{proof} The coproduct on $\sseq{G}^1$ 
is induced by \eqref{eqn: graphcoproduct}.  
Since every proper nonempty subgraph of a wheel $W_g$ has a vertex of degree $<3$ (or, if one prefers, since every non-trivial contraction of a number of edges in a wheel graph produces  a doubled edge),  it follows that any  term $\gamma \otimes W_g/\gamma$ in the reduced coproduct is zero in $C\otimes C$.   
It follows that 
$\Delta [W_g]=1 \otimes [W_g] + [W_g]\otimes 1$.  As a result, the Lie cobracket on $H_*(L^\vee)$, which is obtained by antisymmetrizing $\Delta$, vanishes on $[W_g]$. 
\end{proof}

The fact that the wheel class $[W_{2k+1}] \in H_0(\GC_2^\vee)$ is non-zero in $\mathrm{Hom}( \grt^{\ab},\QQ)$ may also be deduced from a theorem of Rossi and Willwacher \cite{rossi-willwacher-etingof}, asserting that it pairs non-trivially with the image of  $\sigma_{2k+1}$ in $\grt \cong H^0(\GC_2)$ for  every $k>1$.

\begin{prop}  \label{prop: omegasnonzeroingrtab} For $g>1$ odd, the image of the canonical class  $[\omega^{2g-1}] \in \sseq{Q}_1^{g,g}\otimes_{\QQ} \RR $ under the map \eqref{eq:prim-Quillen to grt} 
has non-zero image in the abelianization $\grt^{\ab}\otimes_{\QQ} \RR $. 
\end{prop} 
\begin{proof} Let $g>1$ and suppose  on the contrary that $[\omega^{2g-1}] \in [\grt, \grt]\otimes_{\QQ} \RR$ lies in the subspace of commutators of $H^0(\GC_2)\otimes_{\QQ} \RR$.  Then  by the previous lemma, 
the  pairing $\langle [\omega^{2g-1}], [W_{g}]\rangle $ must vanish, contradicting item~(\ref{it:four}).  
\end{proof} 

We immediately deduce a number of consequences.
The weight filtration on $ \gm$ and $\grt$ is induced by a  grading which we denote by $W$. For the former, 
this is defined to be (minus) half the Hodge-theoretic weight. For the latter,  $\grt$ is by definition embedded as a vector space in the free graded Lie algebra on two generators $\mathbb{L}(X,Y)$ where $X,Y$ are assigned weight $1$. 
This weight has the property that any  generator $\sigma_{2k+1} \in\gm$ has weight  $2k+1$,  and so that the weight coincides with the grading by the  genus (or equivalently by one half of  the cohomological degree) on the diagonal  $\bigoplus_g \sseq{G}^{g,g}_1$.  

\begin{cor} \label{cor: partialfreeness} Let $N>0$  odd such that $\Gr^W_k \gm = \Gr^W_k \grt$ for $ k \leq N$. Then the images of  
\[\omega^{5}, \ldots, \omega^{2N-1}\]
under the map
$\mathrm{Prim}(\bigoplus_{g}\sseq{Q}_1^{g,g}\otimes_\QQ \RR ) \to \grt \otimes_\QQ \RR$
from~\eqref{eq:prim-Quillen to grt} generate a free Lie algebra in $\grt\otimes_{\QQ} \RR$. Consequently
\[ T \left(  \bigoplus_{k=1}^{\frac{N-1}{2}} \omega^{4k+1}[-1] \QQ \right) \to  \bigoplus_{g\geq 0} \sseq{Q}_1^{g,g}\otimes_{\QQ} \RR \]
 is injective.  
\end{cor} 
\begin{proof} We  have established that for all $g>1$ odd, the  image of $[\omega^{2g-1}]$ in $\grt\otimes\RR $
is non-trivial in $\grt^{\ab}\otimes\RR$ and has graded weight  $g$.  If $g\leq N$ then the image 
of $[\omega^{2g-1}] $  lies in $W_N \grt\, \otimes \RR$ and since $W_N \grt = W_N \gm$, it necessarily lies in  the  graded Lie subalgebra $\gm\otimes \RR \subset \grt\otimes\RR$,   and must have non-zero image in $(\gm)^{\ab}\otimes \RR$.
By \cite{brown-mixed}, $\gm$  is a free Lie algebra, with generators given by   any homogeneous choice of representatives for $(\gm)^{\ab}\otimes\RR = H_1(\gm;\RR)$ in $\gm$.  It follows that $\omega^{5}, \ldots, \omega^{2N-1}$ generate a free graded Lie algebra $\mathfrak{g}'$ inside $\grt\otimes \RR$, and hence inside $\mathrm{Prim}(\sseq{Q}_1)$.  The Milnor--Moore theorem then implies that 
the universal enveloping algebra $\mathcal{U}\mathfrak{g}'$ embeds into $\bigoplus \sseq{Q}_1^{g,g}\otimes_{\QQ} \RR$. The last statement follows from  the fact that  the universal enveloping algebra of  a free graded Lie algebra   is isomorphic to the free tensor algebra on its generators.  
 \end{proof}

Establishing Corollary \ref{cor: partialfreeness} for  a given $N$ reduces to a finite, but possibly very large, computation.  In practice, it is equivalent to show the equality of finite-dimensional vector spaces  $W_N \mathcal{O}(\gm)=W_N \mathcal{O}(\grt)$ (where the affine rings $\mathcal{O}(\mathfrak{g})$ are the graded duals of the enveloping algebras $\mathcal{U} \mathfrak{g}$, for $\mathfrak{g}=\gm,\grt$) which can be verified by interpreting elements of  their  affine  rings as formal symbols (corresponding to multiple zeta values) modulo relations and checking that their dimensions agree.  
Indeed,   extensive computer calculations of the latter imply that the assumption of the corollary holds for $N=23$ (see, for example, \cite[p.~2]{MZVdatamine}, where it was asserted that the dimension was computed up to weight $24$).      

One could possibly push this further  using  more recent results and techniques which exploit the Lie algebra structure.  
In any case, this corollary may be combined with previous results to produce a huge amount of cohomology: we deduce that  the symmetric algebra on $\Omega^{*}_{c}[-1]$ and the non-trivial commutators in $\omega^{5}, \ldots, \omega^{45}$  embeds into the $E_1$-page  of the cohomological Quillen spectral sequence.

\subsection{Depth filtration} 
The de Rham  fundamental Lie algebra of the projective line minus 3 points
(\cite{deligne-groupe-fondamental}, see also \cite{brown-depth} and the references therein)
 is canonically isomorphic to the free graded Lie algebra $\mathbb{L}(X,Y)$ on two generators $X,Y$ (corresponding  to generators of $H^1_{\mathrm{dR}}(\mathbb{P}^1\backslash \{0,1,\infty\};\QQ)\cong \QQ X\oplus \QQ Y$). The depth filtration  on $\mathbb{L}(X,Y)$ is the decreasing filtration such that elements of depth $r$  are linear combinations of Lie brackets involving at least $r$ occurrences of  the letter  $Y$.  It 
induces a decreasing depth filtration  $D$ on both 
the Grothendieck--Teichm\"uller and motivic Lie algebras.  It is  known that 
\[ \Gr^1_D \gm =  \Gr^1_D \grt\ ,  \]
since  both sides of this equation are isomorphic to the graded $\QQ$-vector space generated by $\mathrm{ad}(X)^{2n}(Y)$ for $n\geq 1$,
and furthermore  $D^1 \grt= \grt$, and $[\grt, \grt] \subset D^2 \grt$. 
It follows that there is a natural surjective map $\grt^{\ab}\rightarrow \Gr^1_D \grt$. A question of Drinfeld's,  which asks if $\gm\rightarrow \grt$ is surjective, would imply that  $\grt^{\ab}\rightarrow \Gr^1_D \grt$ is an isomorphism (it is known that $\Gr^1_D \gm = (\gm)^{\ab}$ is an isomorphism). 
 We have the following stronger  version of Proposition \ref{prop: omegasnonzeroingrtab}.
 \begin{prop}
    For $g>1$ odd,
   the image of the forms $[\omega^{2g-1}]$ are non-zero in $\Gr^1_D \grt\otimes_{\QQ} \RR$.   
 \end{prop}
\begin{proof}
The proof of  \cite[Proposition 9.1]{willwacher-kontsevich}  shows that the isomorphism 
$ \phi \colon H^0(\GC_2) \rightarrow \grt$ constructed in \emph{loc.~cit.}~lifts to a map from  $\GC_2$ to the free Lie algebra on two generators $\mathbb{L}(X,Y)$, and furthermore has the property that the only graph whose image involves the Lie word $\ad(X)^{2n}(Y)$   is the wheel $W_{2n+1}$.  
It follows that the image $\phi(G)$ for all connected graphs $G$  not isomorphic to a  wheel lies in $D^2 \LL(X,Y)$, since the depth filtration is induced by the degree in $Y$, and hence   $\Gr^1_D \LL(X,Y)$  is generated in  weight $2n+1$ by   precisely  
$\ad(X)^{2n}(Y)$. Denote the increasing filtration on $H_0(\mathsf{GC}_2^\vee)$  dual to the depth filtration on $\grt$ by $D_{\bullet}$. By the above, it has the property that $D_1 H_0(\mathsf{GC}_2^\vee)$ is spanned by the wheel classes $[W_{2n+1}]$ for $n>1$.  Since $[\omega^{2g-1}]$ pairs non-trivially with the wheel $[W_{g}]$, for $g>1$ odd, it follows that $[\omega^{2g-1}] \in \grt\otimes_{\QQ} \RR$ is not contained in $D^2$. 
\end{proof}

Unfortunately,  the bigraded Lie algebra generated by  $\Gr^1_D \grt = \Gr^1_D \gm$ is not free, and has quadratic relations coming from cusp forms (Ihara--Takao). It is nevertheless very large.

\begin{cor}
    The image of the Lie algebra generated by $\{\omega^5,\omega^9,\ldots \} \subset \Prim(W_0 H^*_c(\mathcal{A};\RR) ) $ surjects onto the bigraded Lie subalgebra of $\Gr^*_D \gm$  generated by $\Gr^1_D \gm$ which has one generator in every odd degree $>1$, namely the images of the $\sigma_{2k+1}$ modulo $D^2.$ \end{cor}
    
   This result and, for example, the exact sequence \cite[(7.8)]{brown-depth} (or more precisely, the dimensions of  the  graded pieces of the weight-graded vector space denoted by $\mathbb{D}_2$ in that paper, which is isomorphic to $\Gr^2_D \gm$,    which   can be traced back to   \cite{ZagierDepth2}, equations (11), (12) and following discussion) implies  the following lower bound for the subspace of $\Prim(W_0 H^*_c(\mathcal{A};\RR))$ generated by commutators of generators $\omega^{4k+1}$: 
\begin{equation} \label{eqn: Lotsofcommutators} \dim_{\RR} \left( \sum_{k+\ell=N}   [\omega^{4k+1}, \omega^{4\ell+1}] \RR \right)  \geq  \left\lfloor \frac{N}{3} \right\rfloor  \ .   \end{equation} 
We expect that the  brackets $  [\omega^{4k+1}, \omega^{4\ell+1}]$ for $k<\ell$ are linearly independent if the map in Question \ref{question: diagonaliso} is injective.  That expectation, if true, would imply a lower bound for the dimension on the left hand side of \eqref{eqn: Lotsofcommutators} of  order  $N/2$, while the above argument by depth filtration gives us an unconditional lower bound of order $N/3$.  
This lower bound complements  the statement of  Corollary~\ref{cor: partialfreeness}. That corollary  implies that the  $\omega^{4k+1}$ for small $k$ generate a free Lie algebra, i.e., arbitrary long Lie brackets in $\omega^{4k+1}$ for small $k$ are independent (modulo antisymmetry and Jacobi relations), but leaves open the possibility that the $\omega^{4k+1}$ could in principle commute for large $k$.  The lower bound \eqref{eqn: Lotsofcommutators} rules out this possibility and  proves an orthogonal statement, namely, the independence of many  Lie brackets of length two for  generators $\omega^{4k+1}$  for arbitrarily large $k$.  Both results, namely Corollary~\ref{cor: partialfreeness} and \eqref{eqn: Lotsofcommutators},   point to  the highly non-commutative nature of the product in  $\sseq{Q}^{*,*}_1$ and contribute  to the body of   evidence  in favour of Conjecture \ref{conj: Liealginjects} to be discussed below.

\section{Further results and conjectures}

\subsection{Symmetric products of canonical forms and an announcement  of Ronnie Lee}
An immediate corollary of  Theorem \ref{thm: Tensoralgebramaps} is:

\begin{cor} \label{cor:sym-A} The map \eqref{TOmegamap}  induces  an embedding of  bigraded coalgebras
\[ \Sym \left(\Omega^*_{c}[-1] \right) \hookrightarrow  W_0 H_c^*(\cA;\RR) \] 
where $\Sym \left(\Omega^*_c[-1]\right)  \subset T(\Omega^{*}_{c}[-1])$ is the  vector subspace of (graded-commutative) symmetrized products. By applying the isomorphism $W_0H_c^*(\cA;\QQ) \otimes \QQ[x]/x^2 \cong \sseq{Q}^{*,*}_1$, with $x$ in bidegree $(1,0)$, we deduce an embedding of bigraded coalgebras:
\begin{equation} \label{eqn: SyminjectsHGLn}    \Sym \left(\Omega^*_{c}[-1] \right) \otimes \QQ[x]/x^2    \hookrightarrow \bigoplus_{g,d} H^{d-g}_c(P_g/\GL_g(\ZZ);\RR)   \ . \end{equation}
\end{cor}

\begin{proof} Since $W_0 H_c^*(\cA;\QQ)$ is connected, cocommutative, and finite dimensional in each degree, the Milnor--Moore theorem (\cite[Theorem 5.18]{MR0174052}) gives  an isomorphism of Hopf algebras
$$\mathcal{U}(\mathrm{Prim} (W_0 H_c^*(\cA;\QQ))) \overset{\sim}{\rightarrow} W_0 H_c^*(\cA;\QQ)$$ of bigraded  Hopf algebras. The Poincar\'e--Birkhoff--Witt theorem  implies that  the natural map 
$\mathrm{Sym}( \mathrm{Prim} (W_0 H_c^*(\cA;\QQ)) ) \rightarrow \Gr \mathcal{U}(  \mathrm{Prim} (W_0 H_c^*(\cA;\QQ))    )$ is an isomorphism of algebras, where $\Gr$ is the grading associated to the filtration induced  by  length.  In particular,  the  map  given by symmetrisation of products $\mathrm{Sym}( \mathrm{Prim} (W_0 H_c^*(\cA;\QQ)) ) \rightarrow  \mathcal{U}(  \mathrm{Prim} (W_0 H_c^*(\cA;\QQ))    )$ is an isomorphism of vector spaces (and in fact of coalgebras if given the appropriate prefactor $1/n!$ in length $n$; see \cite[Theorem B2.3]{MR0258031}.  Note that it can sometimes be convenient to choose a different splitting of the length  filtration, such as  in table \ref{tab:E1}, in which case one only obtains an isomorphism of vector spaces). 

Thus we deduce an isomorphism of bigraded coalgebras: 
\[ \mathrm{Sym}( \mathrm{Prim} (W_0 H_c^*(\cA;\RR)) ) \overset{\sim}{\longrightarrow}  W_0H_c^*(\cA;\RR) \]
By  Theorem \ref{thm:canonical-inj}, $\Omega_c^{*}[-1]$ embeds into $ \mathrm{Prim} (W_0 H_c^*(\cA;\RR)) $. 

The second statement follows from the description of the $E_1$-page of the cohomological Quillen spectral sequence in terms of compactly supported cohomology
\[\sseq{Q}_1^{s,t} \cong H_t(\GL_s(\ZZ);\St_s\otimes\QQ)^\vee \cong H^{s+t}_c(P_s/\GL_s(\ZZ);\QQ),\]
see~\eqref{eq:thing} and~\eqref{eq:thing3}.
\end{proof}

A  version of the following injective map, which is  implied by  \eqref{eqn: SyminjectsHGLn},
\[ \Sym \left(\Omega^*_{c}[-1] \right)   \hookrightarrow \bigoplus_{g} H^*_c(P_g/\GL_g(\ZZ);\RR)\]
 was announced by Ronnie Lee in 1978 \cite{lee-unstable}, but no proof has ever appeared in print. The map \eqref{eqn: SyminjectsHGLn} gives rise to infinitely many new classes in the cohomology of the groups  $\SL_g(\ZZ)$.

\begin{example} By Corollary~\ref{cor:sym-A}, we have the following nonzero classes in $H_c^{*}(P_6/\GL_6(\ZZ);\RR)$: 
\[  [\omega^9] .\epsilon \ , \qquad    [\omega^5].[\omega^5]=[\omega^5|\omega^5] \  ,  \qquad   [\omega^5\wedge \omega^9] . \epsilon \ ,    \]
in degrees $11, 12, 16$ respectively.    In  $H_c^{*}(P_7/\GL_7(\ZZ);\RR)$ we obtain:
\[     [\omega^5|\omega^5].\epsilon  \ , \quad   [\omega^{13}]  \ , \quad   [\omega^5 \wedge \omega^{13}] \ , \quad  [\omega^9 \wedge \omega^{13}]  \ , \quad [\omega^5 \wedge \omega^9 \wedge \omega^{13}] \ ,  \]
in degrees $13, 14,  19 ,23 , 28$. Here, a dot denotes the symmetrized product: $a.b= \frac{1}{2} (a\times b + (-1)^{\deg a\cdot \deg b} b\times a)$ where $\times$ is the multiplication in $\sseq{Q}_1.$ It follows from the computations of \cite{elbaz-vincent-gangl-soule-perfect}
that these are the only  non-vanishing classes in the  range for which the compactly supported  cohomology of $P_g/\GL_g(\ZZ)$  has been  computed in its entirety. 
\end{example}

\begin{remark}\label{rem:allknowncohomology}
A much stronger version of the Corollary~\ref{cor:sym-A} holds. By a similar application of the Milnor--Moore theorem, 
the symmetric algebra generated by:
\begin{enumerate}
    \item Independent elements in the Lie algebra generated by $\{\omega^5,\omega^9,\ldots \}$ inside $\mathrm{Prim} \sseq{Q}_1\otimes \RR$
    \item The homogeneous elements in $\Omega_c^*[-1]$ of the form $\omega^{4i_1+1}\wedge \ldots \wedge \omega^{4i_k+1}$ for $k>1$
    \item The generator $\epsilon$ in bidegree $(1,0)$
    \item Infinitely many elements of the form  \eqref{primitivesinvolving59} (see below) 
\end{enumerate}
embeds as a bigraded vector space into $\sseq{Q}_1\otimes \RR$, because these elements are primitive. The Lie algebra in $(1)$ is at least as large as: the free Lie algebra on $\{\omega^5,\ldots, \omega^{45}\}$, and, the part of the depth-graded motivic Lie algebra generated in depth $1$. 
\end{remark}
\subsection{Poincar\'e series and dimensions} \label{sec:Poincare}
Let 
\[ P(s,t) = \sum_{g,n\geq 0} \dim \left( T(\Omega^*_c[-1])_{g,n}\right) \,  s^g t^n\]
denote the Poincar\'e series of the bigraded vector space $T(\Omega^*_c[-1])$ with respect to genus and degree minus genus. 
Let us define for all $k\geq 1$
\[ f_{2k+1} (t) = t^{4k+2} \prod_{i=1}^{k-1}  (1+ t^{4i+1})\ .  \]
Then, for example, we have
\[  f_3(t) = t^6 \ ,  \ f_5(t) = (1+t^5)t^{10} \ , \ f_7(t) = (1+t^5)(1+t^{9})t^{14} \ . \]
Each polynomial $f_{2k+1}(t)$ is the Poincar\'e series for the graded vector space $\Omega^*_c(2k+1)[-1].$ Since $\Omega^*_c[-1]$ is the direct sum of all $\Omega^*_c(2k+1)[-1]$ it follows that 
\[  P(s,t) =  \frac{1}{1- \sum_{k\geq 1} f_{2k+1}(t) s^{2k+1} } = 1 + s^3 t^6 + s^5  (t^{10}+ t^{15}) + t^{12} s^6 + 
  \ldots  \]
  Since $f_{2k+1}(-1)=0$ for all $k>1$, 
an interesting consequence is that the generating function for the Euler characteristic is 
 $P(s,-1) = (1-s^3)^{-1}$.  It follows that the Euler characteristic (with respect to degree) of the genus $g$  component of $T(\Omega^*_c[-1])$ is congruent to $1$ if $g\equiv 0 \pmod 3$ and $0$ otherwise.

\subsection{Diagonals and degrees; proofs and refinements of Corollaries~\ref{cor:tensor}, \ref{cor:expAg}, and \ref{cor:expSL} }\label{subsec:diagonals-and-degrees}
The Poincar\'e series for the tensor algebra $T( \bigoplus_{k\geq 1} \omega^{4k+1} \QQ)$ featuring in Question~\ref{question: diagonaliso}, generated by the classes $\omega^{2g-1}$ in bidegree $(g,g)$ for $g>1$ odd, 
is  
\[    P(s,t) =  \frac{1}{1-s^3 t^6 -s^5 t^{10} - \ldots } = \frac{1-s^2 t^4}{1-s^2t^4-s^3t^6} \]
The coefficient of $s^nt^{2n}$ in $P(t)$ is asymptotically $\alpha^{-n}$ where $\alpha=.7548\cdots$ is the real root of $s^3+s^2-1=0$. Corollary \ref{cor: partialfreeness} and the comments which follow prove that the asymptotic growth of the diagonal 
$\bigoplus_g W_0 H_c^{2g}(\cA_g;\QQ)$ is eventually  bounded below 
by  $\alpha_{\mathrm{approx}}^{-n}$ where $\alpha_{\mathrm{approx}}= 0.7551\cdots$ is the real root of $s^{23}+ s^{21} + \ldots + s^3 - 1=0$, and is therefore very close to what we would  expect  if Question \ref{question: diagonaliso} were true. 

In light of the Poincar\'e--Birkhoff--Witt theorem, by multiplying by any symmetric tensor generated by the elements of $\Omega^*_c[-1]$ with the form $\omega^{4i_1+1} \wedge \ldots \wedge \omega^{4i_k+1}$ where $k>1$,  we obtain an additional  copy of the diagonal Lie algebra generated by $\{\omega^{4k+1}~|~k>1\}$, proving Corollary~\ref{cor:tensor}.  By the computation above, this diagonal Lie algebra has exponential growth (see Remark \ref{rem:allknowncohomology}), proving Corollary~\ref{cor:expAg}.  Here follow two further  applications.

\subsubsection{Refinement of Corollary~\ref{cor:expAg}} \label{rem:kAg}
 The dimension of $W_0H^{2g+k}_c(\cA_g)$ grows at least exponentially with $g$ for all non-negative integers $k$ except possibly $k$ in the set 
\[
S_\cA = \{1, 2,3,4,6,7,8,10,11,12,15,16,19,20,23,24,28,32,36,40\}.
\]
Indeed, by Corollary~\ref{cor:tensor}, it suffices to show that, for each $k \not \in S_\cA$, there is a genus $g$ such that the bigraded vector space $\Omega_c^{*}[-1]$ is nonzero in genus $g$ and degree $2g + k$.

Each homogeneous basis element is of the form
\[
\omega = \omega^{4k_1+1} \wedge \cdots \wedge \omega^{4k_r + 1},
\]
where $k_1 < \ldots < k_r$.  Then $\omega$ is in genus $g = 2k_r + 1$ and degree $(4k_1+1) \cdots + (4k_{r-1} + 1)$. The claim follows, since the numbers that can be written as a sum of distinct integers that are at least 5 and congruent to $1 \mod 4$ are $$5, 9, 13, 14, 17, 18, 21, 22, 25, 26, 27, 29, 30, 31, 33, 34, 35, 37, 38, 39,$$ and all integers greater than or equal to $41$.

\subsubsection{Refinement of  Corollary~\ref{cor:expSL}}  \label{rem:kSL}
The dimension of $H^{\binom{n}{2}-n-k}(\SL_n(\ZZ);\QQ)$ grows at least exponentially for all integers $k \geq -1$ except possibly $k$ in the set  $$S_{\SL} =  \{1,2,3,6,7,10,11,15,19,23 \}.$$

We have already noted the bigraded isomorphism
\[
(W_0 H^*_c(\cA))^\vee \otimes_\QQ \QQ[x] / x^2 \cong \bigoplus_g H_{*} (\GL_g(\ZZ), \St_g \otimes \QQ),
\]
where $x$ has genus $1$ and degree $1$. Likewise, we have noted  that $H^{\binom{n}{2}-k}(\SL_n(\ZZ);\QQ)$ contains $H_{k} (\GL_n(\ZZ), \St_n \otimes \QQ)$ as a summand. Thus, the dimension of $H^{\binom{n}{2}-n-k}(\SL_n(\ZZ);\QQ)$ grows exponentially with $n$ unless both $k$ and $k+1$ are in $S_\cA$.

\subsection{Polynomials in the wheel homology classes} \label{sec:polynomials-wheel} The cohomology classes defined above are only defined over the real numbers, because their pairing with rational homology cycles are given by period integrals which include  odd values of the zeta function. 
However, it was shown in \cite{brown-schnetz} that  the wheel classes give explicit rational homology classes: for all odd $g>1$ there exists an explicit non-zero locally finite homology class: 
\begin{equation} \label{Wheelclasses} [\tau_{W_g}] \in H^{BM}_{2g}( P_g/\GL_g(\ZZ);\QQ) \ .
\end{equation}

A corollary of the existence of the Hopf algebra structure on the $E^1$-page of the homological Quillen spectral sequence implies the following:

\begin{thm} \label{thm:polynomial-wheel} There is an injective map of commutative bigraded algebras
\begin{equation} \label{eqn: PolyWheelsInjects} \QQ[ W_3, W_5,\ldots , W_g,\ldots    ]\otimes \QQ[x]/x^2 \longrightarrow   \sseq{Q}^1     \end{equation}
     where $W_{g}$ denotes the wheel class  \eqref{Wheelclasses} in odd genus $g$ and degree $2g$, and 
     where the element $x$ is in degree $1$ and genus $1$ and maps to $e$. The bigrading on the left hand side is by genus and degree minus genus.
\end{thm}
\begin{proof}
By their definition, the wheel classes \eqref{Wheelclasses} factor through the graph spectral sequence 
  $\sseq{G}^1 \rightarrow \sseq{Q}^1$. They are primitive in $\sseq{Q}^1$ because this is true a fortiori in the graph complex, by Lemma \ref{lem: Wheelsprimitive}.  The element $x$, which represents a one-vertex, one-edge loop in $\sseq{G}^1_{1,0}$, 
  is primitive for reasons of degree.  

  The homological $E^1$-page is graded-commutative but not primitively generated.  The Milnor--Moore theorem (\cite[Theorem 5.18]{MR0174052}) implies that the canonical map
  \begin{equation*}
      \mathrm{Sym}(\Prim ( \sseq{Q}^1_{*,*} )) = \mathcal{U}(\Prim ( \sseq{Q}^1_{*,*} )) \to \sseq{Q}^1_{*,*}
  \end{equation*}
  is an isomorphism onto the subalgebra generated by $\Prim ( \sseq{Q}^1_{*,*} ) \subset \sseq{Q}^1_{*,*}$, and the domain contains $\QQ[W_3,W_5,\ldots ]\otimes \QQ[x]/x^2$ as a subalgebra.
\end{proof}

It was shown in \cite{brown-schnetz} that the wheel classes $W_g$ pair non-trivially with the primitive canonical forms $\omega^{2g-1}$. It follows that the dual to \eqref{eqn: PolyWheelsInjects}, tensored with $\RR$, is the map
\[ W_0 H_c^*(\cA;\QQ) \otimes_{\QQ} \RR \longrightarrow  \RR[ \omega^5,\omega^9,\ldots , \omega^{2g-1}, \ldots ] \ .  \]
which sends all other primitives to zero. In fact, one may replace $\RR$ in the previous map with the $\QQ$-algebra generated by odd  zeta values $\zeta(2n+1)$, for $n\geq 1.$

\subsection{The canonical spectral sequence} \label{sec:canonical-ss}
Recall the map
\[T(\Omega^*_c[-1]\otimes\RR) \to W_0 H^*_c(\cA;\RR)\]
which is injective on $\Omega^*_c[-1]$ by Theorem~\ref{thm: Tensoralgebramaps}.  Here we construct a spectral sequence whose $E_1$-page is $T(\Omega^*_c[-1])$ and compare it to the cohomological Quillen spectral sequence.
Consider the   graded  exterior algebra $\wedge P^* $ on the graded vector space  $P^*=\bigoplus_{k \geq 1} \QQ \beta^{4k+1}$ 
with  generators $\beta^{4k+1}$ in  degree $4k+1$.  Endowed with the zero differential, $\wedge P^*$ defines a connected differential graded algebra. 
The bar construction $B(\wedge P^*)$  is  a   graded commutative Hopf algebra over $\QQ$ generated by symbols $[p_1| \ldots | p_n]$   in degree $\deg(p_1)+ \ldots + \deg(p_n)+n$ where $\deg(p_i)>0$ (note that a more standard convention is to have a minus before the $n$; this is not the case here),  where $p_i$ are homogeneous  generators of $ \wedge P^*$, and is isomorphic as a vector space to the tensor algebra $T(\wedge (P^*_+))$,  where $P^*_+$ denotes the part of $P^*$ in positive degree. In addition it is equipped with the (graded-commutative) signed  shuffle product, and the deconcatenation coproduct
\[ \Delta [p_1 |\ldots | p_n] = \sum_{i=1}^n [p_1| \ldots | p_i ] \otimes [p_{i+1} | \ldots | p_n] \ .\]
In addition, there is  
 an internal differential 
\begin{eqnarray}
d_I \colon B(\wedge P^* ) & \longrightarrow & B(\wedge P^*)  \\
 d_I ( [ p_1 | \ldots | p_n]) & = &  \sum_{i=1}^{n-1} (-1)^i  [s p_1 |  \ldots | s p_{i-1} | s p_i \wedge p_{i+1} | p_{i+2}|  \ldots | p_n]\nonumber
\end{eqnarray}
where $s \colon \wedge P^* \rightarrow \wedge P^*$ is the linear map of graded vector spaces which multiplies  by   $(-1)^{n}$ in degree $n$. The differential $d_I$ has degree $-1$  (owing to the plus sign in our convention for the degree of bar elements) and satisfies $d_I^2=0$. It 
 is compatible with the Hopf algebra structures.
 
Define an increasing  filtration $G$ on $\wedge P^*$ as follows:  $G_g(\wedge P^*)$ is spanned by elements 
$ \omega^{4i_1+1} \wedge \ldots \wedge \omega^{4i_k+1}$ such that $i_1< \ldots < i_k$ and $4i_k+1\leq 2g-1$.  
It defines a filtration $G (\wedge P^*)$ of differential graded algebras.  In fact, it follows from the definition that $G_g  \wedge G_h \subset G_{\max \{g,h\}}$. The  filtration $G$ induces a filtration on $B(\wedge P^*)$, which we also denote by $G$. 
It is  a filtration of graded Hopf algebras, which is respected  by the differential  $d_I$.

\begin{prop}  The filtered  complex  $(B(\wedge P^*), d_I)$ with filtration $G$ defines a spectral sequence $\sseq{c}^r_{s,t}$ of commutative  bigraded Hopf algebras such that 
    \[    \sseq{c}^1_{s,t} =  \Gr_s^G  B_{s+t} ( \wedge P^*)     \ .  \]
This spectral sequence  converges  to the bigraded Hopf algebra 
$\Gr^G \mathrm{Sym}(P_+^*[-1])$ which is  isomorphic to the polynomial ring in primitive generators $\beta^{4k+1}$ in degree $4k+2$ and genus $2k+1$.  Furthermore, the differentials $d_1,d_2$ vanish,  and  $d_r  [\beta^{4k+1}] =0 $ for all $k,r$. 
\end{prop}

\begin{proof}
 The spectral sequence  $\sseq{c}^1_{s,t}$ defined by the filtration $G$ on  $(B(\wedge P^*), d_I)$ has $E^1$-page isomorphic to 
 \[    \sseq{c}^1_{s,t} =  H_{s+t} (   \Gr_s^G  B ( \wedge P^*) )  \]
Since the differential $d_I$ strictly decreases the genus, it is identically zero on the associated graded of $B(\wedge P^*)$, and hence 
$ H_{s+t} (   \Gr_s^G  B ( \wedge P^*) )=   \Gr_s^G  B_{s+t} ( \wedge P^*)$. 
It follows from  the Koszul duality between the symmetric and exterior algebras (\cite{LodayVallette}, Proposition 3.48 and Theorem 3.44) that there is an isomorphism of graded commutative  Hopf algebras 
\[  H_{*} \left( B( \wedge P^*), d_I \right)  \cong   \mathrm{Sym}(P_+^*[-1])  \ .\]
The spectral sequence therefore converges to 
 $\Gr^G H(B(\wedge P^*), d_I)  \cong \Gr^G \mathrm{Sym}(P_+^*[-1])$,  as bigraded Hopf algebras. 
 The fact that the differential $d_r$ annihilates $\beta^{4k+1}$ is clear from the definition of $d_I$, which acts trivially on $[\beta^{4k+1}]$. The fact that $d_1,d_2$ vanish follows from the definition of  $d_I$, and the fact that the map   $[\beta^{4a+1} | \beta^{4b+1}] \mapsto \beta^{4a+1} \wedge \beta^{4b+1}$ sends a term of genus $2(a+b)+2$ to one of genus $2 \max\{a,b\}+2$, and therefore decreases the genus by at least $3$. 
\end{proof}

By identifying $P^*$ with the graded dual of $\Omega^*_c$, 
 we may interpret $B(\wedge P^*)$ as the graded dual of  the tensor Hopf algebra $T(\Omega^*_c[-1])$. The proposition  therefore  defines by  duality a spectral sequence on $T(\Omega^*_c[-1])$, which we  call the {\em canonical spectral sequence}, denoted by $\sseq{c}^{s,t}_r$.   The differentials in this spectral sequence vanish on elements $[\omega^{4k+1}]$ and, for   $r= 2 \min \{a,b\}+1$, send  $[\omega^{4a+1}\wedge \omega^{4b+1}]  $ to the commutator $[\omega^{4a+1}, \omega^{4b+1} ] = [\omega^{4a+1}| \omega^{4b+1}] -[\omega^{4b+1}| \omega^{4a+1}]. $

We expect that the map in Theorem~\ref{thm: Tensoralgebramaps} may be promoted (possibly after rescaling the action of the  differentials)  to a map 
of spectral sequences    
\[   \sseq{c}^{*,*}_r \otimes_{\QQ} \RR \longrightarrow   \sseq{Q}_r^{*,*} \otimes_{\QQ} \RR \  \]
which induces an isomorphism on their abutments (which are formally isomorphic, by the previous proposition).  This provides yet more evidence of a different kind for Conjecture \ref{conj: Tinjects}.

\subsection{Illustration}\label{subsec:illustration}
 Table~\ref{table:spectral-sequence} depicts
  $T(\Omega^*_c[-1])$. The entries that are known to be isomorphic to $W_0H^*_c(\cA;\QQ)$ are highlighted, in particular for $g\leq 7$. Blank entries vanish for dimension reasons. 
    Entries in $\computed{\hbox{low genus}}$ follow from computer calculations of \cite{elbaz-vincent-gangl-soule-perfect} and \cite{dutour-sikiric-elbaz-vincent-kupers-martinet-voronoi}. 

    There are two infinite ranges in which the cohomology of $W_0H^*_c(\cA;\QQ)$ has been completely determined. 
        The $\vanish{\hbox{zero}}$ entries in the bottom three rows follow from \cite{gunnels-symplectic, bruck-patzt-sroka, bruck-miller-patzt-sroka-wilson-codimension}, which imply that
        \[  W_0H^{g+n}_c(\cA_g;\QQ)=0  \quad \hbox{ for }  0\le n\leq 2, g \geq 1\ . \]  The bottom two rows also follow from \cite{lee-szczarba-homology} and \cite{church-putman-codimension}.
       A conjecture of Church--Farb--Putnam implies that $W_0H^n_c(\cA_g;\QQ)$ vanishes for $n<2g$ (below the diagonal line $n=2g$).         The entries  in $\stable{\hbox{high degrees}}$ follow from \cite[\S14.5]{brown-bordifications}, \cite{bbcmmw-top} and  imply that for  all $g>1$ odd 
        \[  W_0H_c^{n}(\cA_g;\QQ)\cong \Omega_c^n(g)[-1] \quad \hbox{  and  } \quad  W_0H_c^{n-1}(\cA_{g-1};\QQ) =  0\]
        for $n\geq d_g -\kappa(g)$, where $d_g$ is the dimension of $\cA_g$ and  $\kappa(g)$ is the stable range for the cohomology of the general linear group (which is  currently  known to be $\kappa(g) \leq g - 1$ by \cite{li-sun-lowdegree}).

\begin{table}[ht]\label{table:spectral-sequence}
\addtolength{\tabcolsep}{-1pt}
    \centering
    \resizebox{1.00\hsize}{!}{
    \begin{tabular}{c|ccccccccccc}
    45   & \blank{0} & \blank{0} & \blank{0} & \blank{0} & \blank{0} & \blank{0} & \blank{0} & \blank{0} & \blank{0}  &  \blank{0}  & $\stable{0}$  \\     
    44   & \blank{0} & \blank{0} & \blank{0} & \blank{0} & \blank{0} & \blank{0} & \blank{0} & \blank{0} & \blank{0}  &  \blank{0}  & $\stable{0}$  \\     
43   & \blank{0} & \blank{0} & \blank{0} & \blank{0} & \blank{0} & \blank{0} & \blank{0} & \blank{0} & \blank{0}  &  \blank{0}  & $0$  \\     
42   & \blank{0} & \blank{0} & \blank{0} & \blank{0} & \blank{0} & \blank{0} & \blank{0} & \blank{0} & \blank{0}  &  \blank{0}  & 0  \\     
41   & \blank{0} & \blank{0} & \blank{0} & \blank{0} & \blank{0} & \blank{0} & \blank{0} & \blank{0} & \blank{0}  &  \blank{0}  & 0  \\     
40   & \blank{0} & \blank{0} & \blank{0} & \blank{0} & \blank{0} & \blank{0} & \blank{0} & \blank{0} & \blank{0}  &  \blank{0}  & 0  \\     
39   & \blank{0} & \blank{0} & \blank{0} & \blank{0} & \blank{0} & \blank{0} & \blank{0} & \blank{0} & \blank{0}  &  \blank{0}  & 0  \\     
38   & \blank{0} & \blank{0} & \blank{0} & \blank{0} & \blank{0} & \blank{0} & \blank{0} & \blank{0} & \blank{0}  &  \blank{0}  & 0  \\     
37   & \blank{0} & \blank{0} & \blank{0} & \blank{0} & \blank{0} & \blank{0} & \blank{0} & \blank{0} & \blank{0}  &  \blank{0}  & 0  \\     
36   & \blank{0} & \blank{0} & \blank{0} & \blank{0} & \blank{0} & \blank{0} & \blank{0} & \blank{0} & \blank{0}  &  $\stable{[\omega^5\!\wedge\! \omega^9\!\wedge\! \omega^{13}\!\wedge\! \omega^{17}]}$  & 0  \\     
35   & \blank{0} & \blank{0} & \blank{0} & \blank{0} & \blank{0} & \blank{0} & \blank{0} & \blank{0} & \blank{0} & $\stable{0}$  & 0   \\ 
34   & \blank{0} & \blank{0} & \blank{0} & \blank{0} & \blank{0} & \blank{0} & \blank{0} & \blank{0} & \blank{0} & $\stable{0}$  & 0   \\ 
33   & \blank{0} & \blank{0} & \blank{0} & \blank{0} & \blank{0} & \blank{0} & \blank{0} & \blank{0} & \blank{0} & $\stable{0}$  & 0   \\ 
32   & \blank{0} & \blank{0} & \blank{0} & \blank{0} & \blank{0} & \blank{0} & \blank{0} & \blank{0} & \blank{0} & $\stable{0}$ & 0   \\ 
31   & \blank{0} & \blank{0} & \blank{0} & \blank{0} & \blank{0} & \blank{0} & \blank{0} & \blank{0} & \blank{0} & $\stable{[\omega^9 \!\wedge\! \omega^{13}\!\wedge\!\omega^{17}]}$   & 0   \\ 
30   & \blank{0} & \blank{0} & \blank{0} & \blank{0} & \blank{0} & \blank{0} & \blank{0} & \blank{0} & \blank{0} &  $\stable{0}$  & 0   \\ 
29   & \blank{0} & \blank{0} & \blank{0} & \blank{0} & \blank{0} & \blank{0} & \blank{0} & \blank{0} & \blank{0} &  $\stable{0}$  & 0   \\ 
28   & \blank{0} & \blank{0} & \blank{0} & \blank{0} & \blank{0} & \blank{0} & \blank{0} & \blank{0} &  $\stable{0}$ &  $\stable{0}$  & 0   \\ 
27   & \blank{0} & \blank{0} & \blank{0} & \blank{0} & \blank{0} & \blank{0} & \blank{0} & \blank{0} &  $\stable{0}$ &  $[\omega^5\!\wedge\! \omega^{13}\!\wedge\! \omega^{17}]$  & 0   \\ 
26   & \blank{0} & \blank{0} & \blank{0} & \blank{0} & \blank{0} & \blank{0} & \blank{0} & \blank{0} &  $0$ &  0 &  0  \\ 
25   & \blank{0} & \blank{0} & \blank{0} & \blank{0} & \blank{0} & \blank{0} & \blank{0} & \blank{0} & 0 &  0 &  0  \\ 
24   & \blank{0} & \blank{0} & \blank{0} & \blank{0} & \blank{0} & \blank{0} & \blank{0} & \blank{0} & 0 &  0 & $\substack{[\omega^5|\omega^5\wedge\omega^9\wedge\omega^{13}] \\ \textcolor{red}{[\omega^5\wedge\omega^9\wedge\omega^{13} | \omega^5]}}$  \\ 
23   & \blank{0} & \blank{0} & \blank{0} & \blank{0} & \blank{0} & \blank{0} & \blank{0} & \blank{0} & 0 &  $[\omega^5\!\wedge\!\omega^9\!\wedge\!\omega^{17}]$  & 0   \\ 
22   & \blank{0} & \blank{0} & \blank{0} & \blank{0} & \blank{0} & \blank{0} & \blank{0} & \blank{0} & 0 & $[\omega^{13}\!\wedge\!\omega^{17}]$  & 0   \\ 
 21   & \blank{0} & \blank{0} & \blank{0} & \blank{0} & \blank{0} & \blank{0} & \blank{0} & $  \stable{[\omega^5\!\wedge\! \omega^9 \!\wedge\! \omega^{13}]} $ & 0 &  0 & 0 
 \\ 
 20   & \blank{0} & \blank{0} & \blank{0} & \blank{0} & \blank{0} & \blank{0} & \blank{0} &  $\stable{0}$ & 0  &  0 & $\textcolor{red}{[\omega^5 \!\wedge\! \omega^9|  \omega^5 \!\wedge\! \omega^9 ]}$  \\ 
 19   & \blank{0} & \blank{0} & \blank{0} & \blank{0} & \blank{0} & \blank{0} & \blank{0} &  $\stable{0}$ &  0 &  0&  $\substack{[\omega^{5} |  \omega^9 \wedge \omega^{13} ] \\  \textcolor{red}{[\omega^9\wedge \omega^{13} | \omega^5]}  }$ \\ 
 18   & \blank{0} & \blank{0} & \blank{0} & \blank{0} & \blank{0} & \blank{0} & \blank{0} &  $\stable{0}$ & 0 & $[\omega^9\!\wedge\! \omega^{17}]$  &  0 \\ 
 17   & \blank{0} & \blank{0} & \blank{0} & \blank{0} & \blank{0} & \blank{0} & \blank{0} &  $\stable{0}$ & 0 & 0 & 0   \\ 
 16   & \blank{0} & \blank{0} & \blank{0} & \blank{0} & \blank{0} & \blank{0} & \blank{0} & $\stable{[\omega^9 \!\wedge\! \omega^{13}]}$ & 0 & 0 & 0  \\ 
 15   & \blank{0} & \blank{0} & \blank{0} & \blank{0} & \blank{0} & \blank{0} &  $\stable{0}$ &  $\stable{0}$ &  0  & 0   & $ \substack{[\omega^9 |  \omega^5 \wedge \omega^9] \\  [\omega^5 \wedge \omega^9 | \omega^9]  }, \substack{[\omega^5 |  \omega^5 \wedge \omega^{13}] \\  \textcolor{red}{[\omega^5 \wedge \omega^{13} | \omega^5]}  }$   \\ 
 14   & \blank{0} & \blank{0} & \blank{0} & \blank{0} & \blank{0} & \blank{0} &  $\stable{0}$ &  $\computed{0}$ & 0  &  $[\omega^5\!\wedge\! \omega^{17}]$ & 0   \\ 
 13   & \blank{0} & \blank{0} & \blank{0} & \blank{0} & \blank{0} & \blank{0} &  $\computed{0}$ & $\computed{0}$ & $\substack{[\omega^5 | \omega^5 \wedge \omega^9] \\  [\omega^5 \wedge \omega^9 | \omega^5]  } $  & 0 &  0  \\ 
 12   & \blank{0} & \blank{0} & \blank{0} & \blank{0} & \blank{0} & \blank{0} & $\computed{0}$ & $\computed{[\omega^5 \!\wedge\! \omega^{13}]} $ & 0  & 0   &  0  \\ 
 11   & \blank{0} & \blank{0} & \blank{0} & \blank{0} & \blank{0} & \blank{0} & $\computed{0}$ & $\computed{0}$ & 0 &  0  & 0   \\ 
 10   & \blank{0} & \blank{0} & \blank{0} & \blank{0} & \blank{0} & $\stable{[\omega^5 \!\wedge\! \omega^9]} $ & $\computed{0}$ & $\computed{0}$ & 0   & 0  &  $ \substack{[\omega^9 | \omega^9] \\  [\omega^5 | \omega^{13}] \ , \ [\omega^{13} | \omega^5]    }$ \\ 
 9   & \blank{0} & \blank{0} & \blank{0} & \blank{0} & \blank{0} &  $\stable{0}$ & $\computed{0}$ & $\computed{0}$ &  0  &   $\substack{[\omega^{17} ] \\  [\omega^5 | \omega^5|\omega^5]  }$   &0  \\ 
 8   & \blank{0} & \blank{0} & \blank{0} & \blank{0} & \blank{0} &  $\stable{0}$ & $\computed{0}$ & $\computed{0}$ &  $ \substack{[\omega^5 | \omega^9] \\  [\omega^9 | \omega^5]  }$  &  0 & 0 \\ 
 7   & \blank{0} & \blank{0} & \blank{0} & \blank{0} & \blank{0} &  $\stable{0}$ & $\computed{0}$ & $\computed{[\omega^{13}]}$   & 0   & 0 &  0\\ 
 6   & \blank{0} & \blank{0} & \blank{0} & \blank{0} &  $\stable{0}$ &  $\stable{0}$ & $\computed{[\omega^5 | \omega^5]}$ & $\computed{0}$ &  0  & 0  & 0 \\ 
 5   & \blank{0} & \blank{0} & \blank{0} & \blank{0} &  $\stable{0}$ & $\computed{[\omega^9]}$ & $\computed{0}$ & $\computed{0}$ & 0 & 0 & 0  \\ 
 4   & \blank{0} & \blank{0} & \blank{0} & \blank{0} & $\computed{0}$ & $\computed{0}$ & $\computed{0}$ & $\computed{0}$ & $\computed{0}$ & 0 & 0  \\ 
 3   & \blank{0} & \blank{0} & \blank{0} & $\stable{[\omega^5]}$ & $\computed{0}$ & $\computed{0}$ & $\computed{0}$ & $\computed{0}$ & $\computed{0}$ & 0 & 0   \\ 
 2   & \blank{0} & \blank{0} & \blank{0} &  $\vanish{0}$ & $\vanish{0}$ & $\vanish{0}$ & $\vanish{0}$ & $\vanish{0}$ & $\vanish{0}$ & $\vanish{0}$ & $\vanish{0}$ \\ 
 1   & $\blank{0}$ & $\blank{0}$& $\vanish{0}$& $\vanish{0}$& $\vanish{0}$& $\vanish{0}$& $\vanish{0}$& $\vanish{0}$& $\vanish{0}$& $\vanish{0}$& $\vanish{0}$\\ 
 0   & $\vanish{1} $ & $\vanish{0}$& $\vanish{0}$& $\vanish{0}$& $\vanish{0}$& $\vanish{0}$& $\vanish{0}$& $\vanish{0}$& $\vanish{0}$& $\vanish{0}$& $\vanish{0}$\\ 
    \hline
    & 0 & 1 & 2 & 3 & 4 & 5 & 6 & 7 & 8 & 9 & 10 \\
    \end{tabular}
    }
    \medskip
    \caption{Without the red entries, a table showing generators for a largest known subspace of $W_0H^*_c(\cA;\RR)$ in genus up to $10$, deduced from Theorems~\ref{thm:canonical-inj} and \ref{thm:free}, and the discussion in Section \ref{subsec:illustration}.  The $(s,t)$ entry shows linearly independent elements of $W_0 H^{s+t}_c(\cA_s;\RR)$.  An entry $(s,t)$ is highlighted in $\stable{\hbox{green}}$, $\computed{\hbox{blue}}$, or $\vanish{\hbox{purple}}$ if it is known to be all of $W_0 H^{s+t}_c(\cA_s;\RR)$; the color coding is given in Section~\ref{subsec:illustration}.  
    The \textcolor{red}{red} entries are additional elements of $T(\Omega^*_c[-1])$ appearing as elements of $E_1$ of the canonical spectral sequence (Section~\ref{sec:canonical-ss}), and conjecturally (Conjecture~\ref{conj: Tinjects}) new linearly independent generators of $W_0 H^*_c(\cA;\RR)$.  Thus, the table in its entirety shows $\sseq{c}_1$ and verifies Conjecture~\ref{conj: Tinjects} for all $g\le 9$.}
    \label{tab:E1}
\end{table}

\subsubsection{Application of the Quillen spectral sequence}\label{sec:applications-of-quillen-ss}
Studying the Quillen spectral sequence further allows us to deduce some nonvanishing classes in $W_0 H^*(\cA;\RR)$ and verifying Conjecture~\ref{conj: Tinjects} up to $g = 9$.  Recall that the cohomological Quillen spectral sequence abuts to a free polynomial algebra with generators in bidegree $(2k+1,2k+1)$, for positive integers $k$. 
By inspection of  Table \ref{tab:E1}, we observe that $[\omega^5]$ and $[\omega^9]$ (and in fact, $[\omega^{13}]$, since a class in genus 7 and degree 14 must appear in the abutment of the Quillen spectral sequence) are  annihilated by the differentials ${}^Qd_r$  in the Quillen spectral sequence for all $r$.  Furthermore, we know that the primitive element $[\omega^5,\omega^9] =[\omega^5| \omega^9] - [\omega^9|\omega^5] \in \sseq{Q}_1^{8,8}$ of degree $16$, which is non-zero by Corollary \ref{cor: partialfreeness}, does not appear in the abutment. Since it is  annihilated by all differentials ${}^Qd_r$, it must be in the image of  a class  of degree 15 in genus $\leq 7$.  The only such class is $[\omega^5\wedge \omega^9]$ in genus $5$. We conclude that
\begin{equation*}
    {}^Qd_r \left( [\omega^5\wedge \omega^9] \right)= \begin{cases} \alpha [\omega^5,\omega^9] \qquad \hbox{ if } r=3 \\ 0 \qquad \qquad  \qquad \hbox{ else}
    \end{cases} 
\end{equation*}
for some $\alpha \in \QQ^{\times}$. Since  $[\omega^5],[\omega^9]$ generate a free Lie algebra in $\Prim \big(\sseq{Q}_1\big)$ (Corollary~\ref{cor: partialfreeness}), we may deduce the non-vanishing of many more elements, and the non-triviality of Lie brackets involving the generator $\omega^5 \wedge \omega^9$, which is noteworthy as $\omega^5\wedge\omega^9$ has odd degree. 

As a warm-up example,
let us first show that $[\omega^5, \omega^5\wedge\omega^9]\in \sseq{Q}_1^{8,13}$ is nonzero. As argued above, both $\omega^5$ and $\omega^5\wedge\omega^9$ represent nonzero classes on $\sseq{Q}_3$, and the Leibniz rule gives ${}^Q d_3[\omega^5, \omega^5 \wedge\omega^9] = [\omega^5, [\omega^5,\omega^9]]$, which is nonzero by freeness of the Lie algebra generated by $[\omega^5]$ and $[\omega^9]$.  Therefore $[\omega^5,\omega^5\wedge\omega^9]\ne 0$.  
More generally,
we can obtain an infinite supply of linearly independent primitive elements by writing down  a Hall basis for the free Lie algebra on two generators $\mathbb{L}(x,y)$, choosing for each one  an occurence of $[x,y]$ and replacing it with $\omega^5\wedge \omega^9$, and finally replacing $x$ with $\omega^5$ and  $y$ with $\omega^9$. For example:  
\begin{eqnarray} \label{primitivesinvolving59}  
  (xy)x  \, ,  \, (xy)y   & \rightarrow & [\omega^5\wedge \omega^9, \omega^5]  \, , \,  [\omega^5\wedge \omega^9, \omega^9]     \\
  ((xy)x)x \, ,  \, ((xy)y)x) \, , \, ((xy)y)y & \rightarrow &  [[\omega^5\wedge \omega^9, \omega^5], 
\omega^5] \, , \,  [[\omega^5\wedge \omega^9, \omega^9],\omega^5]  \, , \,   [[\omega^5\wedge \omega^9, \omega^9],\omega^9]   \nonumber 
   \end{eqnarray} where, in standard notation $(ab)$ denotes $[a,b]$.  In length four the Hall basis
   \[  ((xy)x)x)x \ , \ ((xy)y)x)x \ , \   ((xy)y)y)x \ , \ ((xy)y)y)y \ ,  \  ((xy)y)(xy) \ ,  \  ((xy)x)(xy)
   \]
  may be lifted to, for example, a set of six elements:
   \[   [[[\omega^5\wedge \omega^9,\omega^5],\omega^5],\omega^5]    \ , \       [[[\omega^5\wedge \omega^9,\omega^9],\omega^5],\omega^5]   \ , \   [[[\omega^5\wedge \omega^9,\omega^9],\omega^9],\omega^5]  \]
   \[    [[[\omega^5\wedge \omega^9,\omega^9],\omega^9],\omega^9]  \ , \  [[\omega^5,\omega^9],\omega^9], \omega^5\wedge \omega^9] \ , \  [[\omega^5,\omega^9],\omega^5], \omega^5\wedge \omega^9]  \]
They are independent since their images under ${}^Qd_3$ are part of a Hall basis for the free lie algebra $\mathbb{L}(\omega^5,\omega^9)$, and hence are independent. In particular,  since the Lie bracket $[\omega^5\wedge \omega^9,\omega^5]$ is non-zero,  the Milnor--Moore theorem  implies that  the two classes  in genus $8$  and degree 21 of the form $[\omega^5|\omega^5\wedge \omega^9]$, $[\omega^5\wedge \omega^9|\omega^5]$ are linearly independent, which proves  Conjecture \ref{conj: Tinjects} up to and including genus $9$. 

In genus $10$, injectivity of the map $T(\Omega^*_c[-1])\to W_0 H^*_c(\cA;\RR)$ in Conjecture~\ref{conj: Tinjects} is not known for four cohomological degrees, namely $d=34,30,29,$ and $25$.  Conjecture \ref{conj: Tinjects} predicts the existence of $2,1, 2,$ and $4$ independent classes respectively, while we prove that the dimension of $W_0 H^d_c(\cA_{10};\RR)$ is at least  $1,0,1,$ and $3$ in these degrees, respectively.   For example, in degree $30$, it is not known if the primitive element $[\omega^5\wedge \omega^9| \omega^5\wedge\omega^9]$ is non-zero.   See Table~\ref{tab:E1}.

\subsection{Proof of Theorem~\ref{thm:extension}}\label{sec:proof-mhs}   Following the notation suggested by Namikawa, we shall write $\cA_g^\Sat$ for the minimal Satake, or Baily--Borel, compactification of $\cA_g$.

\medskip

\begin{proof}  
The Satake compactification of $\cA_h$ has a natural stratification
\[
\cA_h^\Sat = \bigsqcup_{g \leq h} \cA_g.
\]
This induces a spectral sequence of mixed Hodge structures $\sseq{\Sat}_*$ abutting to the cohomology  $H^*(\cA_h^\Sat)$ whose $E_1$-page is
\[
\sseq{\Sat}_1^{g,k} = H^{g+k}_c (\cA_g),
\]
for $g \leq h$, and $0$ otherwise.  Let us consider the $E_1$-page of the induced spectral sequence on the weight zero subspaces. Choose $h = 6$, which is sufficiently large so that $H^6(\cA_h^\Sat)$ is stable. Then $\sseq{\Sat}_1$  agrees with the truncation of $W_0 H^*_c(\cA)$ after the $g=6$ column, which is depicted in Table~\ref{tab:E1}.  Similarly, the corresponding spectral sequence abutting to $W_0H^*(\cA_3^\Sat)$ is obtained by truncating after the third column. 
For degree reasons, both of these groups are given by $\sseq{\Sat}_1^{3,3}$. Thus, we have isomorphisms
$$W_0 H^6(\cA_6^\Sat) \cong W_0 H^6_c(\cA_3) \cong W_0 H^6(\cA_3^\Sat).$$ 

In particular, the inclusion $\cA_3^\Sat \subset \cA_6^\Sat$ induces an isomorphism on $W_0H^6$. Now we show that the induced map on all of $H^6$ is injective.
To see this, note that $H^6(\cA_6^\Sat)$ has rank $2$, with
\[
H^6(\cA_6^\Sat)/W_0H^6(\cA_6^\Sat)
\]
spanned by the Goresky--Pardon lift $\tilde{c}_3$ of $c_3(\Lambda)$, the third Chern class of the Hodge bundle \cite{chen-looijenga-stable, looijenga-goresky-pardon}. Thus, it will suffice to show that the restriction of $\tilde{c}_3$ to $\cA_3^\Sat$ is not zero. 

The Goresky--Pardon lift $\tilde{c}_3$ has the following property: fix $g$ and choose a toroidal compactification $\cA_g^\Sigma$, as in \cite{amrt}. Let $\Lambda_g$ denote the Hodge bundle on $\cA_g$ and let $\widetilde\Lambda_g$ be the extension to $\cA_g^\Sigma$ constructed in \cite{mumford-proportionality}. Then there is a unique projection $
\pi_\Sigma \colon \cA_g^\Sigma \to \cA_g^\Sat
$
that extends the identity on the shared open subset $\cA_g$, and
\begin{equation}\label{eq:pi*}
\pi_\Sigma^* (\tilde c_3) = c_3(\widetilde \Lambda_g).
\end{equation}
See \cite{goresky-pardon-chern}. We can choose a toroidal compactification so that the Torelli map extends to a morphism from the moduli space of stable curves $\overline \cM_g \to \cA_g^\Sigma$ and the pullback of $\tilde c_3$ is the Hodge class $\lambda_3$. Then $\lambda_3 \neq 0$ in $H^{6}(\overline \cM_3)$ because it appears as a multiplicative factor in the integrand of explicit nonzero Hodge integrals \cite{faber-pandharipande-hodge}. It follows that $\tilde{c}_3$ is nonzero in $H^6(\cA_3^\Sat),$ which proves the claim.

Next, we claim that $\tilde{c}_3$ is in the image of the push-forward for the open inclusion
\[
\iota_* \colon H^6_c(\cA_3) \to H^6(\cA_3^\Sat).
\]
The property \eqref{eq:pi*} for $g = 2$ shows that the restriction of $\tilde c_3$ to $\cA_2^\Sat$ vanishes, since $\Lambda_2$ has rank 2. Applying excision gives a long exact sequence 
\[
\cdots \to H^6_c(\cA_3) \to   H^6(\cA^\Sat_3) \to  H^6(\cA^\Sat_2) \to \cdots 
\]
Since $\tilde c_3$ is in the kernel of the map to $H^6(\cA_2^\Sat)$, it must be in the image of $H^6_c(\cA_3)$, as claimed.

We have shown that the rank two stable cohomology group $H^6(\cA_6^\Sat)$ injects into $H^6(\cA_3^\Sat)$ and the image is contained in the image of $H^6_c(\cA_3)$, which also has rank 2 \cite{hain-rational}. Thus we have a zig-zag of isomorphisms of mixed Hodge structures 
\[
H^6_c(\cA_3) \xrightarrow{\sim} \ \mathrm{Im}\,\big(H^6(\cA_6^\Sat) \to H^6(\cA_3^\Sat)\big) \ \xleftarrow{\sim} H^6(\cA_\infty^\Sat)
\]
Since they come from algebraic maps, they also induce isomorphisms in algebraic de Rham cohomology, and hence remain true in any  suitable category of realisations.
The main result of \cite{looijenga-goresky-pardon} shows that $H^6(\cA_6^\Sat)$ is a  nontrivial extension of $\QQ(-3)$ by $\QQ$, whose extension class is  given by a nonzero rational multiple of $\zeta(3)$. This means that, in a suitable choice of basis for Betti and de Rham cohomology, the period matrix $P$ is upper-triangular with $\zeta(3)$ above the diagonal, and $1$, $(2\pi i)^3$ along the diagonal. Following the zig-zag and applying Poincar\'e duality to $H^6_c(\cA_3)$ shows that $H^6(\cA_3)$ is likewise a nontrivial extension of $\QQ(-6)$ by $\QQ(-3)$. More precisely, its  period matrix, with respect to the dual basis, is the inverse transpose of $P$ times $(2\pi i)^6$, and hence lower-triangular with $-\zeta(3) (2\pi i)^3$ below the diagonal.  
\end{proof}

\bibliographystyle{amsalpha}
\bibliography{./my}
\end{document}